\documentclass[11pt]{amsart}
\usepackage{amsmath,amsfonts,amssymb}

\usepackage{latexsym}
\usepackage{eufrak}
\usepackage{graphics,epsf,psfrag,epsfig}
\usepackage{xcolor}
\usepackage{hyperref}
\usepackage{longtable}
\usepackage{appendix}

\theoremstyle{plain}
\newtheorem{lemma}{Lemma}[section]
\newtheorem{theorem}[lemma]{Theorem}

\newtheorem{proposition}[lemma]{Proposition}

\theoremstyle{definition}
\newtheorem{remark}[lemma]{Remark}

\numberwithin{equation}{section}

\def\var{\text{var}}

\newcommand{\CC}{\mathbb{C}}
\newcommand{\ZZ}{\mathbb{Z}}

\newcommand{\RR}{\mathbb{R}}

\newcommand{\EE}{\mathbb{E}}
\newcommand{\BB}{\mathbb{B}}

\newcommand{\mkB}{\mathfrak{B}}
\newcommand{\mC}{\mathcal{C}}

\newcommand{\mG}{\mathcal{G}}
\newcommand{\mkG}{\mathfrak{G}}
\newcommand{\mI}{\mathcal{I}}

\newcommand{\mM}{\mathcal{M}}

\newcommand{\mO}{\mathcal{O}}

\newcommand{\mS}{\mathcal{S}}
\newcommand{\mT}{\mathcal{T}}

\newcommand{\mZ}{\mathcal{Z}}
\newcommand{\bdn}{\mathbf{n}}
\newcommand{\bdm}{\mathbf{m}}
\newcommand{\bdz}{\mathbf{0}}

\newcommand{\ep}{\epsilon}

\newcommand{\pa}{\partial}
\newcommand{\bs}{\backslash}

\newcommand{\dis}{{\rm dis}}

\newcommand{\Vol}{{\rm Vol}}

\newcommand{\ol}{\overline}
\newcommand{\yt}{y_\theta}
\newcommand{\ytz}{y_{\theta_0}}
\newcommand{\fatdot}{\boldsymbol{\cdot}}
\newcommand{\tth}{\tilde\theta}

\begin{document}

\title[FPP geodesics in general dimension]{Geodesics, bigeodesics, and coalescence in first passage percolation in general dimension}
\author{Kenneth S. Alexander}
\address{Department of Mathematics \\
University of Southern California\\
Los Angeles, CA  90089-2532 USA}
\email{alexandr@usc.edu}


\begin{abstract} 
We consider geodesics for first passage percolation (FPP) on $\mathbb{Z}^d$ with iid passage times. As has been common in the literature, we assume that the FPP system satisfies certain basic properties conjectured to be true, and derive consequences from these properties.  The assumptions are roughly as follows: (i) the fluctuation scale $\sigma(r)$ of the passage time on scale $r$ grows approximately as a positive power $r^\chi$, in the sense that two natural definitions of $\sigma(r)$ and $\chi$ yield the same value $\chi$, and (ii) the limit shape boundary has curvature uniformly bounded away from 0 and $\infty$ (a requirement we can sometimes limit to a neighborhood of some fixed direction.)   The main a.s.~consequences derived are the following, with $\nu$ denoting a subpolynomial function and $\xi=(1+\chi)/2$ the transverse wandering exponent:  (a) for one-ended geodesic rays with a given asymptotic direction $\theta$, starting in a natural halfspace $H$, for the hyperplane at distance $r$ from $H$, the density of ``entry points'' where some geodesic ray first crosses the hyperplane is at most $\nu(r)/r^{(d-1)\xi}$, (b) the system has no bigeodesics, i.e.~two-ended infinite geodesics, (c) given two sites $x,y$, and a third site $z$ at distance at least $\ell$ from $x$ and $y$, the probability that the geodesic from $x$ to $y$ passes through $z$ is at most $\nu(\ell)/\ell^{(d-1)\xi}$, and (d) in $d=2$, the probability that the geodesic rays in a given direction from two sites have not coalesced after distance $r$ decays like $r^{-\xi}$ to within a subpolynomial factor. Our entry-point density bound compares to a natural conjecture of $c/r^{(d-1)\xi}$, corresponding to a spacing of order $r^\xi$ between entry points, which is the conjectured scale of the transverse wandering.
\end{abstract}

\maketitle

\tableofcontents

\section{Introduction and statement of main results.}\label{intro}
We consider coalescence of geodesic rays, and other related properties of geodesics, in a family of models of first passage percolation (FPP) on $\ZZ^d$ with $d\geq 2$; geodesic rays are semi-infinite paths for which every finite subpath is a geodesic.  A doubly infinite path with the same property is called a \emph{bigeodesic}.  The \emph{asymptotic direction} of a geodesic ray with sites (in order) $v_0,v_1,\dots$ is 
\[
  \lim_n \frac{v_n}{|v_n|} \in S^{d-1}
\]
when this limit exists, where $S^{d-1}$ is the $(d-1)$-sphere and $|\cdot|$ denotes Euclidean length; a $\theta$--\emph{ray} is a geodesic ray with asymptotic direction $\theta$.

We will use percolation ``site/bond'' terminology in the lattice $\ZZ^d$, rather than ``vertex/edge.'' 
By the \emph{limit shape} we mean that of \cite{CD81}, which is the unit ball of the norm $g$ given in \eqref{gdef} below.

For $d=2$ it is known for lattice FPP that under varying hypotheses,
\begin{itemize}
\item[(i)] every geodesic ray has an asymptotic direction \cite{N95};
\item[(ii)] for a fixed direction $\theta$, with probability 1, there exists a unique $\theta$--ray starting from each site (\cite{AH16},\cite{DH14},\cite{DH17},\cite{LN96}, \cite{N95});
\item[(iii)] for a fixed direction $\theta$, with probability 1, for all sites $x,y$, the $\theta$--rays from $x$ and $y$ eventually coalesce (\cite{DH17}, \cite{LN96});
\item[(iv)] for a fixed direction $\theta$, with probability 1 there is no bigeodesic for which either asymptotic direction is $\theta$ (\cite{AH16},\cite{DH17}; a weaker form is in \cite{LN96}.)
\end{itemize}
Note that (iv) does not rule out the existence of all bigeodesics, as it allows a random null set of $\theta$ values for which such bigeodesics exist.  
Given a halfspace $H$ we call a $\theta$--ray a \emph{halfspace $\theta$--ray from} $H$ if its first site, and no other, is contained in $H$.  We may omit the ``from $H$'' if the appropriate halfspace is either apparent or not relevant.

Historically, in a number of works on FPP, unproven properties believed to hold widely have been used as theorem hypotheses.  Newman \cite{N95} assumed unproven curvature of the limit shape boundary to establish the existence of $\theta$--rays ((ii) above.)  In  \cite{LN96}, \cite{LNP96} such curvature was assumed in proving relations between the primary scaling exponents of FPP ($\chi$ governing passage time fluctuations, and $\xi$ governing transverse wandering of geodesics; definitions below.)  In \cite{C13}, to make full sense of the exponent relation $\chi=2\xi-1$ proved there, one effectively must assume an unproven exponential bound for passage time deviations, on the scale of the standard deviation.  We will refer informally  to such ``unproven properties as hypotheses'' as \emph{strong assumptions}, and to others as \emph{weak assumptions}.  Under weak assumptions it is at present quite non-trivial even to prove fundamental things like the existence of $\theta$--rays (\cite{AH16},\cite{DH14},\cite{DH17}.) So there is a tradeoff between the much greater territory explorable under strong assumptions, and the undesirability of having to make such assumptions.  Here we come down on the side of wide exploration, operating in strong--assumptions mode, while keeping the strong assumptions as minimal as possible.

Specifically, we will assume that (i) there is a well--defined fluctuation exponent $\chi$, in the sense that two different natural definitions of the fluctuation scale of passage times both ``grow like $r^\chi$'' for the same exponent $\chi$, and (ii) the boundary of the limit shape has uniform curvature in a certain sense.  See A2, A3 below for details.  Our purpose is to show how certain strong conclusions about geodesic behavior, including aspects of coalescence, flow from essentially just the basic properties.

For $d\geq 3$ most of (i)--(iv) above has not been proved to date---only the existence result for $\theta$--rays in \cite{N95} is known.  Certain relations between particular versions of the exponents $\chi$ and $\xi$ have also been proved in \cite{C13}, \cite{LN96}, \cite{LNP96}.
Under our hypotheses we will, among other things, prove (i), and the existence part of (ii), but it's not clear (though certainly plausible) that (iii) and the uniqueness part of (ii) should be true for general $d$.  Assuming continuous distributions of passage times to prevent ties between paths, \emph{a priori}, for any two geodesic rays $\Gamma,\tilde\Gamma$, any of 3 things may happen:
\begin{itemize}
\item[(1)] $\Gamma,\tilde\Gamma$ are disjoint, i.e.~they have no bonds in common;
\item[(2)] $\Gamma,\tilde\Gamma$ coalesce, that is, there is a site $v\in\Gamma\cap\tilde\Gamma$ such that the segments of $\Gamma,\tilde\Gamma$ up to $v$ are disjoint, and the two rays from $v$ onwards coincide;
\item[(3)] $\Gamma\cap\tilde \Gamma$ is a single segment consisting of finitely many bonds.
\end{itemize}
We refer to the phenomenon in (3) as \emph{temporary touching}; when it occurs, the last site in the segment is called a \emph{branching site}; see Figure \ref{figentry}.
Despite the complications temporary touching and branching create, we can still quantify some aspects of the coalescence of $\theta$--rays in the following way.  Below we will associate to each direction $\theta$ a vector $z_\theta$, chosen so the hyperplane $\{x\in\RR^d: x\cdot z_\theta=1\}$ is tangent to the unit ball of the norm associated to the FPP process at the boundary point of the ball in direction $\theta$.  
We define the hyperplanes and halfspaces \label{Hhyp}
\[
  H_{\theta,s} = \{x\in\RR^d: x\cdot z_\theta=s\},\qquad H_{\theta,s}^+ = \{x\in\RR^d: x\cdot z_\theta\geq s\}, \qquad 
    H_{\theta,s}^- = \{x\in\RR^d: x\cdot z_\theta\leq s\}.
\] 
For $\theta,\theta_0\in S^{d-1}$ (close together), consider the $H_{\theta_0,s}^+$--\emph{entry points} of $\theta$--rays, meaning those lattice sites in $H_{\theta_0,s}^+$ where some $\theta$--ray from $H_{\theta_0,0}^-$ first intersects $H_{\theta_0,s}^+$. 
We may ask, what is the density of such $H_{\theta_0,s}^+$--entry points per unit ($d-1$)--volume near $H_{\theta_0,s}$, and how does it change as $s\to\infty$? Does it approach 0, and if so how fast?  We call this density the $H_{\theta_0,s}$-\emph{crossing density}, postponing a precise definition for later.  Informally we say \emph{bundling} occurs if this density approaches 0 as $s\to\infty$.

Note that $\theta$--rays passing through different $H_{\theta_0,s}$--entry points cannot be assumed disjoint up to those entry points, due to the possibilities of temporary touching and branching.

\begin{figure}
\includegraphics[width=12cm]{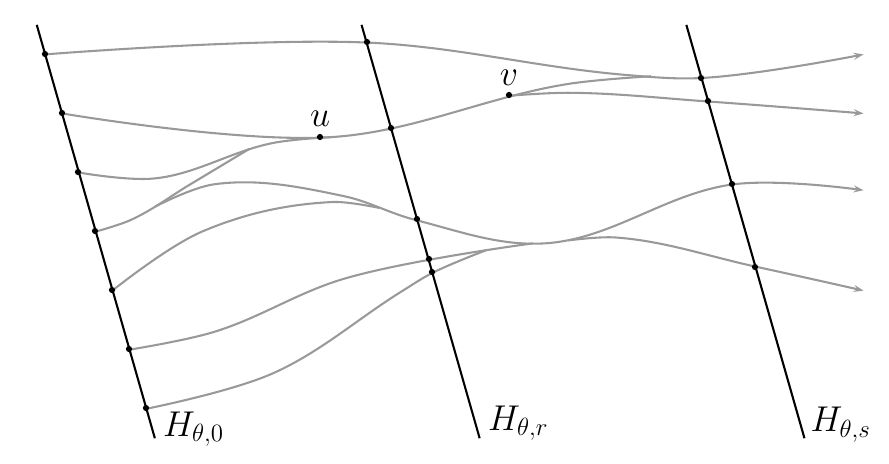}
\caption{ Halfspace $\theta$--rays showing coalescence and possible temporary touching (for example $u$ to $v$) and branching sites (for example $v$.) The dots next to $H_{\theta,r}$ and $H_{\theta,s}$ are $H_{\theta,r}^+$-- and $H_{\theta,s}^+$--entry points. }
\label{figentry}
\end{figure}

If all halfspace $\theta$--rays from $H_{\theta_0,0}^-$ coalesce a.s., then bundling must occur.  The converse is true for $d=2$, but for $d\geq3$ bundling does not in itself guarantee coalescence of all the $\theta$--rays.  We can make equivalence classes of halfspace $\theta$--rays from $H_{\theta_0,0}^-$ by writing $\Gamma\sim\tilde\Gamma$ if $\Gamma,\tilde\Gamma$ eventually coalesce; bundling must occur if the number of equivalence classes is finite, but it is not clear that the converse holds.  One could also ask about the existence of finite equivalence classes; if they exist with positive probability then their starting points must have a positive density near $H_{\theta_0,0}$, but again the possibility of temporary touching means such a positive density of starting points is not immediately ruled out by bundling.  In general it is reasonable to view bundling as a weaker surrogate for coalescence, but as we will see, it is sufficient for proving the absence of bigeodesics.  We will show that bundling occurs in general dimension, with a rate bound that is optimal up to subpolynomial factors.

For $d=2$, we can predict the $H_{\theta_0,s}$--crossing density heuristically from the transverse wandering exponent 
of geodesics, that is, the value $\xi$ such that for a geodesic of length $s$, the maximum distance of the geodesic from the straight line connecting its endpoints is typically of order $s^\xi$.  There must exist halfspace $\theta$--rays, one passing through each $H_{\theta_0,s}^+$--entry point, and any such rays must remain disjoint at least until they cross $H_{\theta_0,s}$, since there is a.s.~no branching or temporary touching for a fixed $\theta$ in two dimensions.
Heuristically, under weak assumptions, to remain disjoint until $H_{\theta_0,s}^+$ the $\theta$--rays should be spaced apart by order $s^\xi$, so the $H_{\theta_0,s}$--crossing density should be $s^{-\xi}$. For $d\geq 3$ this predicts an $H_{\theta_0,s}$--crossing density of at least $s^{-(d-1)\xi}$, but it is not clear \emph{a priori} that the crossing density shouldn't be greater, since the $\theta$--rays can weave around one another without meeting, and we cannot rule out the branching of some $\theta$--rays each into multiple $\theta$--rays.  
We will show that in fact the $H_{\theta_0,s}$--crossing density approaches 0 faster than $s^{-(d-1)\xi+\ep}$ for all $\ep>0$.  Along the way we will need to obtain a variety of results about the regularity of geodesics, and their transverse fluctuations.  

In addition to $\xi$, the other exponent of central interest is the $\chi$ for which the standard deviation of the passage time over distance $r$ ``grows like $r^\chi$.'' Our precise assumptions related to this standard deviation are given in A2 below.  It is known \cite{C13} that under strong assumptions and reasonable definitions of the exponents, $\chi,\xi$ are related by $\chi=2\xi-1$, in all dimensions.

As mentioned above, our assumptions of basic model properties cannot be verified for any specific model of FPP; they would imply an exponential bound on scale $r^\chi$ for the true $\chi$, whereas the best known exponential bounds for $d=2$ are on scale $r^{1/2}$ (\cite{Ke93}, \cite{Ta95}, \cite{DK14}) compared to the conjectured value of $\chi=1/3$.  An exponential bound on scale $r^{1/3}$ is known for certain integrable models of last passage percolation (LPP) for $d=2$, however---in \cite{BSS16} (extracted from \cite{BFP14}) and in \cite{LR10} for LPP on $\ZZ^2$ with exponential passage times, in \cite{LM01}, \cite{LMR02} for LPP based on a Poisson process in the unit square, and in \cite{CLW16} for LPP on $\ZZ^2$ with geometric passage times. 

For $d\geq 3$ there is no generally-agreed-upon value of $\chi$ in the physics literature.  Simulations suggest that $\chi$ should decrease with dimension; simulations in \cite{ROM15} for a model believed to be in the same (KPZ) universality class as FPP show a decrease from $\chi=.33$ to $\chi=.054$ as $d$ increases from 2 to 7.  Some have predicted the existence of a finite upper critical dimension, possibly as low as 3.5, above which $\chi=0$ (\cite{Fo08},\cite{LW05}); others predict that $\chi$ is positive for all $d$ (\cite{AOF14},\cite{MPPR02}), with simulations in \cite{KK14} showing $\chi>0$ all the way to $d=12$, decaying approximately as $1/(d+1)$. Our results here require $\chi>0$ so they only have content below the upper critical dimension, should it be finite. 

In \cite{BHS18}, Basu, Hoffman, and Sly show that there are no bigeodesics for last-passage percolation (LPP) on $\ZZ^2$ when passage times are exponential, essentially by following Newman's heuristic of bounding the density of entry points, which in turn involves bounding the probability of overcrowded parallel geodesics.  (See also the earlier papers \cite{P15}, \cite{GRS17}, and \cite{BBS19} for a proof avoiding results from integrable probability.) The paper \cite{BHS18} exploits key ingredients not available in our general context---planarity, the above--mentioned exponential bound on the scale of the standard deviation, and the fact that the rescaled passage time distributions converge to a limit (Tracy-Widom) which has negative mean.  We need here a completely different heuristic and proof to control overcrowded geodesics, and this is the core of our main proof; see section \ref{outl} for proof sketches.

In two dimensions one can use bounds on the probability of overcrowded geodesics to bound the probability of non-coalescence before traveling distance $cr$, for $\theta$-rays which start at separation $r^\xi$.  For the integrable case of LPP in $d=2$ with exponential passage times, such a bound on non-coalescence probabilities was proved in \cite{BSS19}, with further results in \cite{Zh20}.  But again, strong use is made of planarity and bounds obtained through integrable probability, so the methods do not apply in our context. The results are at the optimal rate, analogous to removing the subpolynomial factors in our Theorems \ref{alldim} and \ref{coalesce}. The reliance on integrable probability was removed in \cite{SS19}, but the results are still restricted to LPP in $d=2$ with exponential passage times.

For Brownian LPP with $d=2$, in \cite{SS21} a.s.~coalescence of all $\theta$--rays, simultaneously in all $\theta$, was established, along with the a.s.~absence of bigeodesics for each fixed $\theta$, but again, the methods do not extend to our context.  Analogs of the absence of bigeodesics in polymer models with $d=2$ are established in \cite{BS20}.

Let us now formalize our definitions.
Let $\EE^d$ \label{EEd} denote the set of all bonds (i.e.~ nearest-neighbor pairs) of $\ZZ^d$.  The \emph{passage times} of bonds are a collection of nonnegative iid variables $\tau = \{\tau_e:e\in\EE^d\}$. For $x,y\in V$, a \emph{(self-avoiding) path} $\Gamma$ from $x$ to $y$ is a finite sequence of alternating sites and bonds of $\ZZ^d$, of the form $\Gamma=(x=x_0, \langle x_0,x_1 \rangle, x_1,\dots,x_{n-1},\langle x_{n-1},x_n \rangle, x_n=y)$, with all $x_i$ distinct; we may identify a path by just the sequence of sites, or view it as a string of line segments, as convenient, when clear from the context.  The \emph{passage time} of $\Gamma$ is 
\[
  T(\Gamma) := \sum_{e\in\Gamma} \tau_e,
\]
and the passage time from $x$ to $y$ is
\begin{equation}\label{Tinf}
  T(x,y) := \inf\{T(\Gamma): \Gamma \text{ is a path from $x$ to $y$ in } \ZZ^d\}.
\end{equation}
A path $\Gamma$ achieving the infimum in \eqref{Tinf} is called a \emph{geodesic} from $x$ to $y$.  For technical convenience we extend \eqref{Tinf} to $x,y\in \RR^d$ as follows.  Define $Z:\RR^d\to\ZZ^d$ by $Z(x)=z$ for all $z\in\ZZ^d$ and $x\in z+[-1/2,1/2)^d$, where $+$ denotes translation, and set
\[
  T(x,y) = T(Z(x),Z(y)), \quad x,y \in \RR^d.
\]

We assume the following.\\

\noindent{\bf A1. $\tau_e$ properties.}
\begin{itemize}
\item[(i)] $\tau_e$ is a continuous random variable.
\item[(ii)] There exists $\lambda>0$ such that $Ee^{\lambda\tau_e}<\infty$.
\end{itemize}
A1 guarantees that there is exactly one geodesic from $x$ to $y$ a.s., for each $x,y$; we denote it $\Gamma_{xy}$.  As is standard, since passage times $T(x,y)$ are subadditive ($T(x,y) \leq T(x,z)+T(z,y)$), A1(ii) guarantees the a.s.~existence (positive and finite for $x\neq 0$) of the limit \label{gx}
\begin{equation}\label{gdef}
  g(x) = \lim_n \frac{T(0,nx)}{n} = \lim_n \frac{ET(0,nx)}{n} = \inf_n \frac{ET(0,nx)}{n}
\end{equation}
for $x\in\ZZ^d$; $g$ extends to $x$ with rational coordinates by considering only $n$ with $nx\in\ZZ^d$, and then to a norm on $\RR^d$ by uniform continuity.  We let $\mkB_g$ denote the unit ball of $g$, and write $y_\theta$ \label{yth} for the positive multiple of $\theta$ which lies in $\partial\mkB_g$ (so $g(y_\theta)=1$.)  The tangent hyperplane to $\partial\mkB_g$ at $y_\theta$ will be unique under our hypotheses, and there is a vector $z_\theta$ such that this tangent hyperplane is $\{x\in\RR^d:x\cdot z_\theta=1\} = H_{\theta,1}$.

An infinite self-avoiding path $\Gamma=(x=x_0, \langle x_0,x_1 \rangle, x_1, \langle x_1,x_2 \rangle,\dots)$ is a \emph{geodesic ray} if every finite segment of $\Gamma$ is a geodesic.  Given $\theta$ in the sphere $S^{d-1}$, we say $\Gamma$ is a $\theta$--\emph{ray} if $\lim_n x_n/|x_n| =\theta$, and a \emph{subsequential $\theta$--ray} if there exists a subsequence $\{x_{n_k}\}$ for which $x_{n_k} /|x_{n_k}|\to \theta$.

Throughout the paper, $c_1,c_2,\dots$ and $C_1,C_2,\dots$, and $\ep_0,\ep_1,\dots$ represent unspecified constants, and $\nu_i(\cdot)$ \label{nui} represent unspecified subpolynomial functions, which depend only on $d$ and the distribution of the passage times $\tau_e$.  We use $C_i$ for constants which occur outside of proofs and may be referenced anywhere; any given $C_i$ has the same value at all occurrences.  We use $c_i$ for those which do not recur and are only needed inside one proof.  For the $c_i$'s we restart the numbering with $c_1$ in each proof, and the values are different in different proofs.

To avoid technical clutter, at various times we will assume (sometimes tacitly) that certain points of $\RR^d$ are lattice sites, and certain (large) real numbers are integers; the modifications to be made when this fails are trivial.

As mentioned above, we cannot formally establish simple hypotheses on the distribution of $\tau_e$ under which our conclusions hold, due to the lack of any results establishing an exponential bound on the scale of the standard deviation.
Instead we will assume certain ``macroscopic'' properties which one expects to be consequences of such hypotheses, as follows. For $\chi\geq0$ we say a function $\lambda: (0,\infty) \to (0,\infty)$ \emph{grows with power} $\chi$ if 
\begin{equation}\label{growpower}
  \lim_{r\to\infty} \frac{\log\lambda(r)}{\log r} = \chi;
\end{equation}
we say $\lambda$ is \emph{subpolynomial} if it grows with power 0. 
We define \emph{upper and lower fluctuation exponents} $\chi^\pm$ by
\[
  \chi^+ = \inf\left\{ \chi>0: Ee^{|T(0,re_1)-ET(0,re_1)|/r^\chi} \text{ is bounded in } r\right\}, \quad
    \chi^- = \liminf_{r\to\infty} \frac{\log\var(T(0,re_1))}{2\log r}.
\]
It is easy to see that $\chi^-\leq \chi^+$; our hypothesis will usually be essentially that a well--defined $\chi$ exists: \label{chis}\\

\noindent{\bf A2. Unique fluctuation exponent.} $\chi^-=\chi^+=\chi\in(0,1)$. \\

\noindent A2 is very close to the assumptions (A1)--(A4) in Chaterjee's (\cite{C13}, Theorem 1.1).
In general, under A2' we can define $\lambda^\pm(r)\in(0,\infty)$ by \label{lams}
\begin{equation}\label{lamdef}
  E\exp\left( \frac{|T(0,re_1)-ET(0,re_1)|}{\lambda^+(r)} \right) = 4,\quad \lambda^-(r) = \var(T(0,re_1))^{1/2}.
\end{equation}
Then $|T(x,y) - ET(x,y)|$ has exponential tails on scale $\lambda^+(|y-x|)$, 
and it follows readily that $\lambda^-(r)/\lambda^+(r)$ is bounded. Therefore
\[
  \chi^- \leq \limsup_{r\to\infty}  \frac{\log\lambda^-(r)}{\log r} \leq \limsup_{r\to\infty} \frac{\log\lambda^+(r)}{\log r} \leq \chi^+,
\]
and the same for lim inf in place of lim sup. Hence under A2 both $\lambda^-$ and $\lambda^+$ grow with power $\chi$. At times in place of A2 we assume the weaker\\

\noindent {\bf A2'. Existence of upper fluctuation exponent.} $\lambda^+$ grows with power $\chi\in(0,1)$.\\

We do not know any particular regularity for $\lambda^\pm$ of \eqref{lamdef} beyond \eqref{growpower}, but we can avoid this problem by working with more regular (in a sense defined in section \ref{prelim}, including continuous and strictly increasing) functions $\sigma_-\leq\lambda^-,\,\lambda^+\leq\sigma$ which, from \cite{Al20a}, necessarily exist and which still grow with power $\chi$; we call such $\sigma$ a \emph{regular upper bound}. It is straightforward that for $\lambda:(0,\infty)\to(0,\infty)$ and $\chi\geq 0$,
\begin{equation}\label{subexist}
  \lambda(r) = O(r^\alpha) \ \text{as $r\to\infty$ for all } \alpha>\chi \implies \lambda(r) = O(r^\chi\nu(r)) \ 
    \text{for some subpolynomial } \nu,
\end{equation}
with an analogous statement for lower bounds on $\lambda(r)$. 

Heuristically we expect $\sigma(r)/\sigma_-(r)$ to be bounded in $r$, meaning all natural ways of measuring the fluctuation scale are the same, to within a constant, and they grow regularly with $r$; in this case all subpolynomial functions $\nu_i(r)$ in the paper become powers of $\log r$.

Our final assumption will be the following.\\

\noindent {\bf A3. Local curvature near $\theta_0$.}
For a specified $\theta_0\in S^{d-1}$, for some $\ep_0>0$ and constants $C_i>0$, for all $\theta\in S^{d-1}$ with $|\theta-\theta_0|<\ep_0$, and all $y\in H_{\theta,1}$ with $|y-y_\theta|<\ep_0$, we have
\begin{equation}\label{curv}
  C_5|y-y_\theta|^2 \leq d(y,\mkB_g) \leq C_6|y-y_\theta|^2.\\
\end{equation}

\noindent Sometimes in place of A3 we assume:\\

\noindent {\bf A3'. Globally uniform curvature.}
For some $\ep_0>0$, for all $\theta\in S^{d-1}$ and all $y\in H_{\theta,1}$ with $|y-y_\theta|<\ep_0$, \eqref{curv} holds.\\

\noindent Note that A3 guarantees that $H_{\theta,1}$ is unique for $\theta$ near $\theta_0$, and A3' guarantees the same thing for all $\theta\in S^{d-1}$.

We now formalize the notion of crossing densities.  

\begin{remark}\label{rational}
Let us call a hyperplane $H\subset \RR^d$ \emph{rationally oriented} if $H\cap\ZZ^d$ spans $H$; then $H\cap\ZZ^d$ is an infinite lattice. When $H_{\theta,0}$ is rationally oriented, we can apply the multidimensional ergodic theorem (see \cite{Ge88}, Appendix 14.A) to help obtain the existence of a (nonrandom) crossing density for $\theta$--rays.  For other $\theta$ an additional argument would be required for this; since we are primarily interested in upper bounds, we will avoid the matter by defining the crossing density as a lim sup.  
\end{remark}

Let $B_r(x)$ denote the closed Euclidean ball of radius $r$ centered at $x$, and recall $z_\theta$ is perpendicular to the tangent plane $H_{\theta,1}$ to $\mkB_g$ at $y_\theta$.  Define the cones \label{jth}
\[
  J(\theta,\ep) := \left\{ u\in\RR^d: u\neq 0, \left| \frac{u}{|u|} - \theta\right| < \ep \right\}.
\]
Let $L_\theta(u)$ \label{lth} denote the line though a point $u$ in direction $\theta$, and write $L_\theta$ for $L_\theta(0)$.  The $\theta$--\emph{projection} of a point $x$ into a hyperplane $H_{\theta,s}$ is the projection along $L_\theta$.

To derive the nonexistence of bigeodesics from bundling, we will need to establish bundling not just for $\theta$--rays with a fixed $\theta$, but also in terms of the combined $H_{\theta,s}^+$--entry points for $\alpha$--rays in a small cone of directions $\alpha$ around some $\theta$.  To state the necessary formalities, for $A\subset H_{\theta,s}$ let $\mC_{\theta,s}(A)$ be the set of all $H_{\theta,s}^+$--entry points $x$ of halfspace $\theta$--rays from $H_{\theta,0}^-$, for which the $\theta$--projection of $x$ into $H_{\theta,s}$ lies in $A$. Similarly let $\mC_{J(\theta,\ep),s}(A)$ be the set of all sites $x$ which are $H_{\theta,s}^+$--entry points of halfspace $\alpha$--rays from $H_{\theta,0}^-$ for some $\alpha\in J(\theta,\ep)$, for which the $\theta$--projection of $x$ into $H_{\theta,s}$ lies in $A$.  Formally, the \emph{mean $H_{\theta,s}$--crossing density} $\ol\rho_\theta(s)$ for $\theta$--rays, and the \emph{mean combined $H_{\theta,s}$--crossing density} $\ol\rho_{J(\theta,\ep)}(s)$, are given by \label{rhos}
\begin{equation}\label{rhobar}
  \ol\rho_\theta(s) = \limsup_{r\to\infty} \frac{E(|\mC_{\theta,s}(B_r(sy_\theta)\cap H_{\theta,s})|)} {\Vol_{d-1}(B_r(sy_\theta)\cap H_{\theta,s})},
\end{equation}
\begin{equation}\label{rhobarJ}
  \ol\rho_{J(\theta,\ep)}(s) 
    = \limsup_{r\to\infty} \frac{E(|\mC_{J(\theta,\ep),s}(B_r(sy_\theta)\cap H_{\theta,s})|)} {\Vol_{d-1}(B_r(sy_\theta)\cap H_{\theta,s})},
\end{equation}
where $\Vol_{d-1}(\cdot)$ denotes $(d-1)$--dimensional volume.  The corresponding almost-sure values are the \emph{$H_{\theta,s}$--crossing density} $ \rho_\theta(s)$ and the \emph{combined $H_{\theta,s}$--crossing density} $\rho_{J(\theta,\ep)}(s)$ given by
\begin{equation}\label{rho}
  \limsup_{r\to\infty} \frac{|\mC_{\theta,s}(B_r(sy_\theta)\cap H_{\theta,s})|} {\Vol_{d-1}(B_r(sy_\theta)\cap H_{\theta,s})} = \rho_\theta(s)\ 
    \text{ a.s.},
\end{equation}
\begin{equation}\label{rhoJ}
  \limsup_{r\to\infty} \frac{|\mC_{J(\theta,\ep),s}(B_r(sy_\theta)\cap H_{\theta,s})|} {\Vol_{d-1}(B_r(sy_\theta)\cap H_{\theta,s})} 
    = \rho_{J(\theta,\ep)}(s)\ \text{ a.s.};
\end{equation}
such nonrandom constants exist because the lim sup is a tail random variable.  

\begin{remark}\label{rational2}
Continuing from Remark \ref{rational}, for rationally oriented $\theta$ we may replace lim sup in \eqref{rhobar}--\eqref{rhoJ} with limit, and we have 
\begin{equation}\label{rhoequal}
  \rho_\theta(s)=\ol\rho_\theta(s),\quad \rho_{J(\theta,\ep)}(s)=\ol\rho_{J(\theta,\ep)}(s).
\end{equation}
Existence of a limit in \eqref{rhobar} and \eqref{rhobarJ} follows here from periodicity in $x$ of $P(x\in \mC_{\theta,s}(\RR^d))$ and $P(x\in \mC_{J(\theta,\ep),s}(\RR^d)$, and the equality with almost sure limits follows from the multidimensional ergodic theorem (see \cite{Ge88}, Appendix 14.A.)
\end{remark}

The following is our main result.

\begin{theorem}\label{alldim}
Suppose for some FPP process on $\ZZ^d$, A1, A2, A3 hold for some $\theta_0,\ep_0$.  There exists $\ep_2$ as follows.
\begin{itemize}
\item[(i)] With probability 1, for all $\theta\in J(\theta_0,\ep_2)$ and $x\in\ZZ^d$, there is at least one $\theta$-ray from $x$.
\item[(ii)] There exist subpolynomial $\nu_i$ such that 
\begin{equation}\label{combdens}
  \ol\rho_{J(\theta_0,\ep_2)}(s) \leq \frac{\nu_1(s)}{s^{(d-1)\chi}} \quad\text{for all $s\geq C_{12}$,}
\end{equation}
and for all $\theta\in J(\theta_0,\ep_2)$,
\begin{equation}\label{crossdens}
  \ol\rho_{\theta}(s) \leq \frac{\nu_2(s)}{s^{(d-1)\xi}} \quad\text{for all $s\geq C_{12}$.}
\end{equation}
\item[(iii)] With probability 1, there exists no bigeodesic containing a subsequential $\theta$--ray with $\theta\in J(\theta_0,\ep_2)$.
\item[(iv)] Suppose also A4 holds.  Then with probability 1, (a) every geodesic ray has an asymptotic direction, (b) for every $\theta\in S^{d-1}$ and every $x\in\ZZ^d$ there is at least one $\theta$--ray from $x$, and (c) there are no bigeodesics.
\end{itemize}
\end{theorem}

Parts (i), (iv)(a), and (iv)(b) require only A2' in place of A2.
We will prove (i), (iv)(a), and (iv)(b) in Section \ref{rays}, (ii) in Section \ref{cpdens}, and (iii) and (iv)(c) in Section \ref{nobigeo}.

Theorem \ref{alldim}(iv) improves on existing results even for $d=2$ (though under stronger hypotheses), as it rules out bigeodesics in all directions simultaneously, instead of almost surely for a fixed direction as in \cite{AH16}, \cite{DH17}, \cite{LN96}.

As we have noted, one expects the spacing of entry points at distance $R$ to be of the same order as the transverse fluctuation of geodesics at the same distance $R$; in other words, two geodesics which are close enough that their transverse fluctuations allow them to coalesce should generally do so.  This means the bound \eqref{crossdens} should be sharp up to the subpolynomial function in the numerator.  The bound \eqref{combdens} is likely not sharp, though, as we expect the combining of a small sector of directions should not significantly increase the number of entry points; a bound like \eqref{crossdens} should apply to the mean combined crossing density as well.  But for the purpose of banning bigeodesics, bundling, in the sense that \eqref{combdens} approaches 0 as $s\to\infty$, is sufficient.

\begin{remark}\label{branching}
Suppose we fix $\theta$ and consider the collection of all halfspace $\theta$--rays from $H_{\theta,0}^-$.  By the time these reach $H_{\theta,R}$, based on \eqref{crossdens} enough coalescence or temporary touching must occur so that on average at least order $R^{(d-1)\xi}/\nu_2(R)$ $\theta$--rays pass through any given $H_{\theta,R}^+$--entry point $x$.  To the extent this is due to temporary touching rather than coalescence, these $\theta$--rays will later separate again.  But this separating can occur at most only slowly: it can be shown using Proposition \ref{jammed1} that the number of $H_{\theta,2R}^+$--entry points of $\theta$--rays passing through a given $H_{\theta,R}^+$--entry point is with high probability at most of subpolynomial order. One must also consider the additional coalescence and/or temporary touching initiated between $H_{\theta,R}$ and $H_{\theta,2R}$, since the crossing density is lower for $H_{\theta,2R}$, at least asymptotically.  So it is interesting to ask, do the conclusions of Theorem \ref{alldim} already force coalescence of all $\theta$--rays? We do not know the answer.
\end{remark}

For $\sigma$ a regular upper bound for $\lambda^+$, the corresponding transverse wandering function is defined as \label{del}
\[
  \Delta(r) = (r\sigma(r))^{1/2},
\]
which, under A2, grows like $r^\xi$ for $\xi=(1+\chi)/2 \in (\frac 12,1)$. There therefore exists a subpolynomial function $\nu_\Delta$ such that
\begin{equation}\label{subp2}
  \Delta(r) = r^\xi \nu_\Delta(r).
\end{equation}
To motivate the designation ``transverse wandering function,'' consider two sites $x,y$ separated by distance $r$, and a third site $z$ at some distance $\Delta \ll r$ from the line segment connecting $x$ to $y$, somewhere near the middle of this line.  The Euclidean distance via $z$, that is, $|y-z|+|z-x|$, exceeds the straight distance $|y-x|$ by order $\Delta^2/r$, and under the local curvature assumption A3, the same will be true for the distance in the norm $g$.  Heuristically, the geodesic may nonetheless pass through $z$, meaning $T(x,z)+T(z,y) = T(x,y)$, if the fluctuation scale $\sigma(r)$ of the random ``distance'' $T(\cdot,\cdot)$ is larger than the deterministic excess distance, that is, $\sigma(r) \geq \Delta^2/r$ or equivalently $\Delta \leq \Delta(r)$. This is simply the usual heuristic for $\chi=2\xi-1$, but allowing for $\sigma(r)$ possibly not being a pure power.

When $\theta$ is fixed it will be convenient to express a general $u\in\RR^d$ in terms of a basis $\BB_\theta$ in which the first vector is $y_\theta$, and the other $d-1$ form an orthonormal basis for $H_{\theta,0}$. (The particular choice of orthonormal basis does not matter.)  In a mild abuse of notation we will simply write $u= (u_1^\theta,u_2^\theta)_\theta$ \label{coord} for the representation in this basis, with $u_1^\theta = u\cdot z_\theta\in\RR$ and $u_2^\theta\in\RR^{d-1}$; we call these $\theta$--\emph{coordinates}.  The corresponding decomposition of $u$ is 
\begin{equation}\label{thdecomp}
  u = (u_1^\theta,0)_\theta + (0,u_2^\theta)_\theta = u_1^\theta\yt + (u-u_1^\theta\yt)
\end{equation}
(see Figure \ref{figdist}), and we refer to $u_1^\theta\yt$ and $u-u_1^\theta\yt$ as the first and second $\theta$--\emph{components} of $u$.

For $d=2$ it is known (\cite{DH14},\cite{LN96}) that, under hypotheses weaker than A1--A3, for each fixed $\theta\in J(\theta_0,\ep_0)$, with probability one there is a unique $\theta$--ray, which we denote $\Gamma_x^\theta$, starting from each $x\in\ZZ^2$, and any two such $\theta$--rays eventually coalesce; there is no temporary touching or branching.  (We note again, however, that under the curvature assumption A3' there must be a random countable set of directions $\theta$ for which branching does occur, producing multiple $\theta$--rays from a single site.)  So we may ask, how far do two $\theta$--rays go before they coalesce?  To formulate the question more precisely we first make some definitions. Fix $\theta,\tth\in J(\theta_0,\ep_0)$. A $\tth$--\emph{start site} is a site in $H_{\tth,0}^-$ which is adjacent to a site in the interior of $H_{\tth_0,0}^+$. A $\tth$--start site $x$ is a $\theta$--\emph{source} (in a configuration $\tau$) if $\Gamma_x^\theta$ is a halfspace $\theta$--ray from $H_{\tth,0}^-$.  With probability one, for any two $\tth$--start sites $x,y$, there exists a unique coalescence site $U_{xy}^\theta$ such that $\Gamma_x^\theta$ and $\Gamma_y^\theta$ are disjoint up to $U_{xy}^\theta$ and coincide from $U_{xy}^\theta$ onward.  The $\tth$--\emph{coalescence time} of $\Gamma_x^\theta$ and $\Gamma_y^\theta$ is the $\tth$--coordinate $(U_{xy}^\theta)_1^{\tth}$. Throughout the preceding, it would be convenient to take $\theta=\tth$, but we will need $\tth$ to be rationally oriented.

We can bound the tail of the coalescence time as a consequence of Theorem \ref{alldim}(ii), as follows.

\begin{theorem}\label{coalesce}
Suppose for some FPP process on $\ZZ^2$, A1, A2, A3 hold for some $\theta_0,\ep_0$. There exist constants $\ep_3,C_i$ and subpolynomial functions $\nu_i$ as follows. Fix $\theta\in J(\theta_0,\ep_3)$ and $r>0$, and suppose $x,y$ are $\theta$--start sites with $C_{13} \leq |x-y|\leq r^\xi\nu_\Delta(r)$.  Then
\begin{equation}\label{coaltime}
  C_{14}\frac{|x-y|}{r^\xi\nu_3(r)} \leq P\Big( (U_{xy}^\theta)_1^\theta \geq r \Big) 
    \leq \frac{|x-y|}{r^\xi\nu_4(r)}.
\end{equation}
\end{theorem}

The proof will show that we can take $\nu_3(r)= \nu_\Delta(r)\log r$.
One expects that the probability in \eqref{coaltime} is actually of order $|x-y|/r^\xi$, uniformly in $|x-y|$, without any subpolynomial corrections involved; such a result is proved for $2d$ LPP in \cite{BSS19}, with related bounds in \cite{P15} and \cite{SS19}.

Let $e_j$ denote the $j$th unit coordinate vector.  Under assumptions much milder than ours, it is proved in \cite{AH16} that for all sequences $v_k$ in $\ZZ^2$ with $|v_k|\to\infty$ we have $P(0 \in \Gamma_{-v_k,v_k}) \to 0$ as $k\to\infty$.  In particular, taking $v_k=ke_1$ solves a conjecture made in \cite{BKS03}.  Here, under stronger hypotheses, in general dimension we can establish a rate at which this probability converges to 0; as with \eqref{crossdens} we expect this rate to be optimal up to the subpolynomial factor in the numerator.  The statement is as follows.

\begin{theorem}\label{hitpoint}
Suppose for some FPP process on $\ZZ^d$, A1, A2, A3 hold for some $\theta_0,\ep_0$. There exist constants $\ep_4,C_i$ and a subpolynomial $\nu_5$ as follows. Suppose $u,v\in\ZZ^d$ with $|v_1^{\theta_0}|\geq |u_1^{\theta_0}|\geq C_{15}$, and $(v-u)/|v-u| \in J(\theta_0,\ep_4)$.  Then
\begin{equation}\label{hitpoint2}
  P(0\in\Gamma_{uv}) \leq \frac{ \nu_5(|u|) }{ |u|^{(d-1)\xi} }.
\end{equation}
If we replace assumption A3 with A3', then the same is true without the assumption $(v-u)/|v-u| \in J(\theta_0,\ep_4)$.
\end{theorem}

Appendix \ref{tableapp} provides a table of notation.

\section{Outline of the paper and proof sketches} \label{outl}

In this section some details and definitions are altered to reduce technicalities.  As discussed in the introduction, one expects the spacing of $H_{\theta_0,r}^+$--entry points to be of the same order $\Delta(r)$ as geodesic transverse fluctuations, meaning one entry point per volume $\Delta(r)^{d-1}$ of $H_{\theta_0,r}$, so the density of such entry points (after projection into $H_{\theta_0,r}$) should be of order $1/\Delta(r)^{d-1}$. All our main results make central use of control of the crossing density. To obtain such control, we show in section \ref{crowd} that for some $c$, a density greater than a certain subpolynomial order $(\pi(r)\log r)^c$ (see \eqref{zetadef}) entails that, modulo certain low--probability events, there must be large numbers of ``crowded'' disjoint geodesics, spaced much more tightly than the natural spacing $\Delta(r)$, and then we show that such crowded geodesics are unlikely; more on that below.

Once control of the crossing density is established in sections \ref{crowd} and \ref{alldimpf}, we
use it in section \ref{hitpointpf} to prove Theorem \ref{hitpoint}, as follows.  When $v-u$ has direction near $\theta_0$, $P(0\in\Gamma_{uv})$ can be interpreted as the density of sites $x$ next to $H_{\theta_0,0}$ for which $x\in\Gamma_{x+u,x+v}$. But for such $x$, again modulo certain low--probability events, $x\in\Gamma_{x+u,x+v}$ entails that $x$ is near the $H_{\theta_0,0}^+$--entry point of $\Gamma_{x+u,x+v}$. Therefore control of the crossing--point density translates into control of the density of sites $x\in\Gamma_{x+u,x+v}$.

We next use control of the crossing density, or more specifically the bundling property, in section \ref{nobigeo} to rule out bigeodesics, proving Theorem \ref{alldim}(iii),(iv).  In this we roughly follow a heuristic of Newman, presented at the AIM 2015 workshop ``First-passage percolation and related models.''  We first show that any bigeodesic must be nearly straight, with some asymptotic direction $\theta$ one way and $-\theta$ the other.  If bigeodesics exist with direction $\theta$ near some $\theta_0$, then so do $H_{\theta_0,0}$--entry points of such bigeodesics, so there must be a positive density of such entry points among sites next to $H_{\theta_0,0}$.  But for every $r>0$ and $\theta$ near $\theta_0$, each $\theta$--bigeodesic contains a $\theta$--ray from $H_{\theta_0,-r}^-$, so bigeodesic entry points are a subset of $H_{\theta_0,0}^+$--entry points of $\theta$--rays from $H_{\theta_0,-r}^-$. But bundling, valid for the combined ray directions $\theta$ near $\theta_0$, means that letting $r\to\infty$ sends the combined $\theta$--ray crossing density to 0, so the bigeodesic crossing density must be 0, contradicting the existence of $\theta$--bigeodesics simultaneously over $\theta$ in a neighborhood of $\theta_0$. Compactness ensures that we need only consider finitely many $\theta_0$ to simultaneously rule out bigeodesics in all directions.

Lastly in section \ref{coal} we use control of the crossing density to bound coalescence probabilities. In $d=2$ there is a unique $\theta$--ray from each site, a.s.  For the moment we restrict attention to those $\theta$--rays from $H_{\theta_0,0}^-$ originating at $\theta$--sources, meaning those $\theta$--rays which have only their first site in $H_{\theta_0,0}^-$.  We can group together the $\theta$--sources for which the corresponding $\theta$--rays have the same $H_{\theta_0,r}^+$--entry point. Due to planarity these groups (after projection into $H_{\theta_0,0}$) correspond to intervals in the line $H_{\theta_0,0}$, with the intervals separated by gaps; see Figure \ref{fig1-7A}.  Thus failure of geodesics from $\theta$--sources $x,y$ to coalesce before reaching $H_{\theta_0,r}^+$ is equivalent to $x,y$ being in different groups, i.e.~to there being a gap between $x$ and $y$.  For general $x,y$ next to $H_{\theta_0,0}$ which are not necessarily $\theta$--sources, modulo certain low--probability events, a gap must be at least close to $x$ or $y$, if not between.  Either way, the density of pairs $u,v$ of sites next to $H_{\theta_0,0}$, having the same separation $v-u=y-x$, for which the $\theta$--rays from $u,v$ do not coalesce before reaching $H_{\theta_0,r}^+$ can be bounded relative to the density of gaps, which in turn can be bounded relative to the crossing density in $H_{\theta_0,r}$.

Since control of the crossing density is so central, we describe next how it is established in sections \ref{crowd} and \ref{alldimpf}.  To start we consider slab geodesics from $H_{\theta_0,0}^-$ to $H_{\theta_0,2R}^+$. Each splits into a ``primary'' geodesic from $H_{\theta_0,0}^-$ to $H_{\theta_0,2R}$ and a ``secondary'' geodesic from $H_{\theta_0,R}$ to $H_{\theta_0,2R}^+$.
Since the geodesic between an two sites is unique, two such geodesics with different ``midway'' $H_{\theta_0,R}^+$--entry points cannot intersect in both their primary and secondary halves.  Using this and a combinatorial argument we show that if we have $n^3$ such slab geodesics all with diferent midway entry points, then, restricting attention to either the primary or secondary halves, 
there is a subset of size roughly $n$ within which the number of geodesics passing through any given site is limited.  This entails a degree of disjointness, which we quantify, between sub--segments of the $n$ geodesics.  
Thus it is enough to consider collections of geodesics from $H_{\theta_0,0}^-$ to $H_{\theta_0,R}^+$, having a certain type of local near--disjointness.  In what follows, $\beta_i\in (0,1)$ are various exponent values, all small, and $n$ is at least a certain subpolynomial function of $R$.   We divide $H_{\theta_0,0}$ and $H_{\theta_0,R}$ each into boxes of side $n^{-\beta_0}\Delta(R)$ (smaller than the natural spaceing $\Delta(R)$ of geodesics) and consider such a box $B_0$ in $H_{\theta_0,0}$ and a box $B_R$ in $H_{\theta_0,R}$, such that the direction from $B_0$ to $B_R$ is near $\theta_0$.
We end up with the need to bound the probability of crowded geodesics,
\begin{align}\label{keyprob2}
  P&\Big( \text{there exist $n$ $\theta_0$--slab geodesics from $B_0$ to $B_R$ with ``local near--disjointness''}\notag\\
  &\qquad \text{ and distinct $H_{\theta_0,R}^+$--entry points } \Big).
\end{align}
To do this, divide each slab geodesic into $n^{\beta_2}$ segments of equal length $\ell = n^{-\beta_2}R$, cutting at the $H_{\theta_0,i\ell}^+$--entry points for each $i\leq n^{\beta_2}$, and divide each $H_{\theta_0,i\ell}$ into boxes as above.  The natural transverse spacing of these segments is $\Delta(\ell)=\Delta(n^{-\beta_2}R)$; for given $i$ let us call the $i$th segments of two slab geodesics $\Gamma^{(1)},\Gamma^{(2)}$ \emph{neighbors} if their endpoints at each end lie in the same box, ensuring the endpoint separation is $\ll \Delta(\ell)$. We show that with very high probability, no pair of neighbors (for any $i$) have passage times that differ by more than a small multiple of $\sigma(\ell)$. The neighbors have close passage times because, except near their endpoints, they are ``geodesics chosen from the same set of possible paths.''  The close passage times mean that, modulo a small-probability event, if we fix any one slab geodesic $\Gamma$ and the passage times of its segments, those segment passage times in $\Gamma$ effectively nearly determine the passage times of all neighbor segments of other slab geodesics.  In particular, we say a segment is \emph{fast} if its passage time is at most $n^{-\beta_2}ET(0,R\ytz) + \frac \eta 8 \sigma(\ell)$; up to a small error, for every fast segment of our fixed slab geodesic, all neighbor segments in other slab geodesics must be fast.  We then show that a fixed one of our $n$ slab geodesics (let's call it ``special'') likely has at least of order $n^{\chi_2\beta_3}$ fast segments before crossing $H_{\theta_0,R}$, and every one of those fast segments has one or more nearly--disjoint neighbors, all of which are also fast--somewhat as in Figure \ref{fig4-10B}.  However, by conditioning on the special geodesic and the bond passage times thereon, we can use the FKG inequality to say that the presence of the special slab geodesic (because it's a geodesic) stochastically increases the passage times of bonds not in the geodesic; hence the probability is small that all the neighbors will be fast.

In summary, the fast segments of length $\ell$ in the special geodesic force all their neighbor segments to be fast (with high probability), while simultaneously reducing the neighbor segments' probability to be fast, in the FKG sense.  These contradictory aspects can only coexist if the probability of crowded geodesics as in \eqref{keyprob2} is very small.  The lack of need for planarity in this argument is what enables extension to general dimension.

At various points in the description so far, we have referenced small--probability events that must be avoided, for the above heuristic to be realized in a proof. Most of these events involve some kind of irregularity in some geodesic having some direction $\theta$ from starting point to its end---it wanders away from the straight line connecting its endpoints, it has a segment that is unusually fast or slow or not near direction $\theta$, it has too few segments that are at least moderately fast, and so on.  Much of our (unfortunately extensive) technical work involves verifying that these truly are small-probability events; this occupies sections \ref{cost} and \ref{densely}.  All this requires knowing that the norm $g$ of \eqref{gdef} is ``like the Euclidean norm'' in ways related to the boundary curvature assumption A3; such results appear in section \ref{geom}.  The key result on irregularity of geodesics, which underlies most of the others, is Proposition \ref{transfluct2}. For geodesics from 0 to $r\yt$ for some $r,\theta$, we can consider the probablity cost of $\Gamma_{0,r\yt}$ passing through a given point $u$ unusually far from the straight line from 0 to $r\yt$.  To within a log factor, this cost is given by the function $D_{\theta,r}(u)$ of \eqref{Dtr}, which has two formulas, depending on whether the deviation creates an angle $u/|u|$ far from $\theta$, or close to $\theta$.  ``Most'' level sets of $D_{\theta,r}$ look like the bottom row of Figure \ref{fig2-4}; the cylinder--like regions at the two ends correspond to deviation directions far from $\theta$. The essential aspect of Proposition \ref{transfluct2} is that it does not require a fixed $u$; it bounds the probability that there exists \emph{any} $u$ in $\Gamma_{0,r\yt}$ with large $D_{\theta,r}(u)$.  To establish this, we consider a point $U$, necessarily on the level set boundary, where the maximum of $D_{\theta,r}(u)$ is achieved, as (in one case) in the bottom row of Figure \ref{fig2-4}.  We show that there must be a third point $W\in\Gamma_{0,r\yt}$, equal to $r\yt$ in the bottom--row case, such that $0,U,W$ form the corners of what we call a $\delta$--fat triangle, meaning that the distance from $U$ to the line $\overline{0W}$ is at least $\delta$ fraction of $|W|$.  Analogous considerations apply to the other cases in Figure \ref{fig2-4}, and in each case we can sum over possible values of $U,W$ to obtain the uniform bound.

\section{Geometry and the cost of bad geodesic behavior}

\subsection{Some definitions and properties}\label{prelim}
We begin with regularity assumptions on the bounding functions $\sigma,\sigma_-$. 
We say $\lambda:(0,\infty)\to(0,\infty)$ is \emph{monotone sublinearly powerlike (with exponent $\chi$)}, abbreviated \emph{$\lambda\in$ MSL($\chi$)}, if $\lambda$ grows with power $\chi\in(0,1)$, $\lambda(r)$ is strictly increasing, $\lambda(r)/r^{1-\ep}$ is strictly decreasing for some $\ep>0$, and for every $0<\chi_1<\chi<\chi_2<1$ there exist constants $C_i$ such that \label{chii}
\begin{equation}\label{powerlike}
  \text{for all } s\geq r\geq C_1, \quad 
    C_2\left( \frac sr \right)^{\chi_1} \leq \frac{\lambda(s)}{\lambda(r)}
     \leq C_3 \left( \frac sr \right)^{\chi_2}.
\end{equation}
As proved in \cite{Al20a}, under A2 there exist $\sigma,\sigma_-\in$ MSL($\chi$) \label{sigs} satisfying $\sigma_-\leq\lambda^-,\,\lambda^+\leq\sigma$, and $\sigma_-\leq\sigma$. The function
\begin{equation}\label{zetadef}
  \pi(r) = \frac{\sigma(r)}{\sigma_-(r)} \in (1,\infty)
\end{equation}
then grows with power 0. We fix some such $\sigma,\sigma_-$ to remain the same from this point forward.  Note that \eqref{powerlike} implies that for all $r,s\geq C_1$ (including $r>s$),
\begin{equation}\label{powerlike2}
  \frac{\sigma(s)}{\sigma(r)} \leq \frac{1}{C_2} + C_3\frac sr.
\end{equation}
From \eqref{lamdef}, under A2' we have
\begin{equation}\label{expbd}
  P\Big( |T(x,y) - ET(x,y)| \geq t\sigma(|y-x|) \Big) \leq 4e^{-t} \quad \text{for all $t>0$ and } x,y\in\RR^2.
\end{equation}
From \eqref{expbd} and (\cite{Al20a} Lemma 1.2) it follows that there exist constants $\eta<1/2,C_7>0$ such that for all $x, y \in \RR^d$ with $|x-y|\geq C_4$ the following both hold:
\begin{equation}\label{minfluct}
  P\Big( T(x,y) - ET(x,y) > \eta\sigma_-(|x-y|) \Big) \geq \frac{C_7}{\pi(|x-y|)^2(1+\log\pi(|x-y|))}, 
\end{equation}
\begin{equation}\label{minfluct2}
  P\Big( T(x,y) - ET(x,y) < -\eta\sigma_-(|x-y|) \Big) \geq \frac{C_7}{\pi(|x-y|)^2(1+\log\pi(|x-y|))}.
\end{equation}
The essential aspects are that this gives a lower bound for one-sided probabilities, based on the two-sided assumption \eqref{expbd}, and that the scale of deviations is $\sigma_-(\cdot)$ rather than $\sigma(\cdot)$.

\subsection{Mean passage time vs.~its asymptotic approximation}
Let \label{hofx}
\[
  h(x) = ET(0,x), \quad x\in\ZZ^d.
\]
Then $h(x)-g(x)$ is nonnegative by subadditivity of $h$.  A variant of the following bound was proved in \cite{Al97}, with error term $C_{16}|x|^{1/2}\log |x|$; in \cite{Te18} this was improved to $C_{16}|x|^{1/2}(\log |x|)^{1/2}$.  The proof in \cite{Al97} adapts readily to the present situation and we don't need improvement to $(\log |x|)^{1/2}$, so we will work here without that improvement.

\begin{proposition}\label{hmu}
Assume A1 and A2'. There exists $C_{16}$ such that for all $|x|\geq 2$,
\begin{equation}\label{hmu2}
  g(x) \leq h(x) \leq g(x) + C_{16}\sigma(|x|) \log |x|.
\end{equation}
\end{proposition}
\begin{proof}
Defining
\[
   \tilde\sigma(r) = r^\chi\sup_{s\leq r}\frac{\sigma(s)}{s^\chi}
\]
we have from \eqref{powerlike} that $\sigma(r)\leq\tilde\sigma(r)\leq \sigma(r)/C_2$, while $\tilde\sigma(r)/r^\chi$ is nondecreasing.  Therefore \eqref{expbd} is valid for $\tilde\sigma$ in place of $\sigma$.

The proof of (\cite{Al97}, Proposition 3.4) only uses the existence of an exponential bound on scale $|x|^{1/2}$ from \cite{Ke93}, so under A2' it is valid for $\tilde\sigma(|x|)$ in place of $|x|^{1/2}$. To obtain (\cite{Al97}, Theorem 3.2) from (\cite{Al97}, Proposition 3.4), both with the replacement error term $C_{16}\tilde\sigma(|x|) \log |x|$, we need only the fact that this error term can be expressed as $C_{16}|x|^\chi\varphi(|x|)$ with $\chi>0$ and $\varphi(r) = \tilde\sigma(r)(\log r)/r^\chi$ nondecreasing.  
\end{proof}

\subsection{Geometric observations} \label{geom}
Recall that $\Gamma_{xy}$ denotes the geodesic between $x$ and $y$.
In general we view $\Gamma_{xy}$ as an undirected path, but at times we will refer to, for example, the first point of $\Gamma_{xy}$ with some property.  Hence when appropriate, and clear from the context, we view $\Gamma_{xy}$ as a path from $x$ to $y$.

Turning to curvature--related matters, when A3 holds, for $y\in H_{\theta,1}$ with $|y-y_\theta|\geq\ep_0$ we have by \eqref{curv} and local strict convexity that
\begin{equation}\label{curv2}
  d(y,\mkB_g) \geq C_5\ep_0 |y-y_\theta|,
\end{equation}
so for all $y\in H_{\theta,1}$,
\begin{equation}\label{curv3}
  d(y,\mkB_g) \geq C_5\ep_0^2\left( \frac{|y-y_\theta|}{\ep_0} \wedge \left( \frac{|y-y_\theta|}{\ep_0} \right)^2 \right)
\end{equation}

\begin{remark}\label{spheres}
An equivalent way to state the local curvature condition A3 is as follows.  Let $B_r(x)$ denote the closed Euclidean ball of radius $r$ centered at $x$, and recall $z_\theta$ is perpendicular to the tangent plane $H_{\theta,1}$ to $\mkB_g$ at $y_\theta$.  Recall also the cones $J(\theta,\ep)$. There exist contants $C_8 < C_9$ as follows.  For all $\theta \in J(\theta_0,\ep_0)\cap S^{d-1}$ we have
\begin{equation}\label{spheres2}
  B_{C_8|z_\theta|}(y_\theta - C_8z_\theta) \cap J(\theta,\ep) \subset \mkB_g \cap J(\theta,\ep)
    \subset B_{C_9|z_\theta|}(y_\theta - C_9z_\theta) \cap J(\theta,\ep).
\end{equation}
Equation \eqref{spheres2} says that near $y_\theta$, $\mkB_g$ is sandwiched between two balls which also have tangent plane $H_{\theta,1}$ at $y_\theta$.  From this we see that for some $\ep_1$ determined by $\ep_0$, the angle $\psi_{z_\theta,z_\alpha}$ between $z_\theta$ and $z_\alpha$ (equivalently between hyperplanes $H_{\theta,0}$ and $H_{\alpha,0}$) grows at most linearly in $|\theta-\alpha|$:
\begin{equation}\label{zangle}
  |\theta-\theta_0|<\ep_1, |\alpha-\theta| < \ep_1 \implies \psi_{z_\theta,z_\alpha} \leq C_{10}|\alpha-\theta|,\quad
    \big| y_\alpha - \yt \big| \leq C_{11}|\alpha-\theta|.
\end{equation}
The bound on $|y_\alpha - \yt|$ actually requires \eqref{yzangle} below, but we include it here for convenience.
\end{remark}

Note that $|u|$ always refers to the Euclidean length of $u$, not to the length of the vector of $\theta$--coordinates.  For a norm based on $\theta$--coordinates, we use \label{tnorm}
\[
  |u|_{\theta,\infty} := \max(|u_1^\theta|,|u_2^\theta|),
\]
which satisfies
\begin{equation}\label{thetanorm}
  (1+|\yt|) |u|_{\theta,\infty} \geq |u_1^\theta| |\yt| + |u_2^\theta| \geq |u|.
\end{equation}
We define ``distance via hyperplane'': noting that $u\in H_{\theta,u_1^\theta}$, for $A\subset\RR^d$ with $A\cap H_{\theta,u_1^\theta}\neq\emptyset$ let \label{hyp}
\[
  d_\theta(u,A) = d(u,A\cap H_{\theta,u_1^\theta}),
\]
and note that $d_\theta(u,L_\theta)= |u_2^\theta|$. Finally define the $\theta$-\emph{ratio} of $0\neq u\in\RR^e$ to be $|u_2^\theta|/|u_1^\theta|\in [0,\infty]$, and note that if $H_{\theta,0}\perp\theta$ then this is just the tangent of the angle between $u$ and $L_\theta$; in general the $\theta$--ratio is a surrogate for this tangent.

A number of our arguments involve the following theme: for points $x,y\in \RR^2$, if $y$ is far enough from the straight line from 0 to $x$, then $\Gamma_{0x}$ is unlikely to pass through $y$, because $g(y) + g(x-y)$ significantly exceeds $g(x)$.  Since \eqref{expbd} involves centering at the expectation, to express this theme in more than a crude way we also need to quantify how increases in $g(x)$ relate to increases in $h(x)$, but for the moment we consider just $g$.  For Euclidean distance the following is useful: 
\begin{equation}\label{triangle}
  \ell + \frac{m^2}{2\ell} \geq (\ell^2+m^2)^{1/2} \geq \ell + \min\left(\frac m3,\frac{m^2}{3\ell} \right) \qquad\text{for all } \ell,m> 0.
\end{equation}
For an analog under the local curvature assumption A3, we first note that the angle between $\theta$ and $H_{\theta,0}$ is always at least $\arcsin 1/\sqrt{d}$; see the discussion around \eqref{yzangle}.
The following analog of \eqref{triangle} for $g$-distance is straightforward, proved similarly to Lemma \ref{gtri} below (see also Remark \ref{spheres}), so we omit details.
There exist constants $\ep_1>0$ and $C_i>0$ such that for $|\theta-\theta_0|<\ep_1$ and $u= (u_1^\theta,u_2^\theta)\in\RR^d$ satisfying $u_1^\theta>0$ and $|\theta - u/|u||<\ep_1$, we have
\begin{equation}\label{triangleg}
  u_1^\theta + C_{17} \min\left( |u_2^\theta|, \frac{|u_2^\theta|^2}{u_1^\theta} \right) \leq g(u) \leq 
    u_1^\theta + C_{18} \min\left( |u_2^\theta|, \frac{|u_2^\theta|^2}{u_1^\theta} \right).
\end{equation}
This is really the significance of the local curvature assumption for the boundary of $\mkB_g$: it says that in dealing with vectors with direction near $\theta_0$, the norm $g$ is ``Euclidean-like'' in that \eqref{triangleg} holds, and, most importantly, there is consequently a discrepancy, as in \eqref{gtri1} below, in the triangle inequality.

\begin{figure}
\includegraphics[width=9cm]{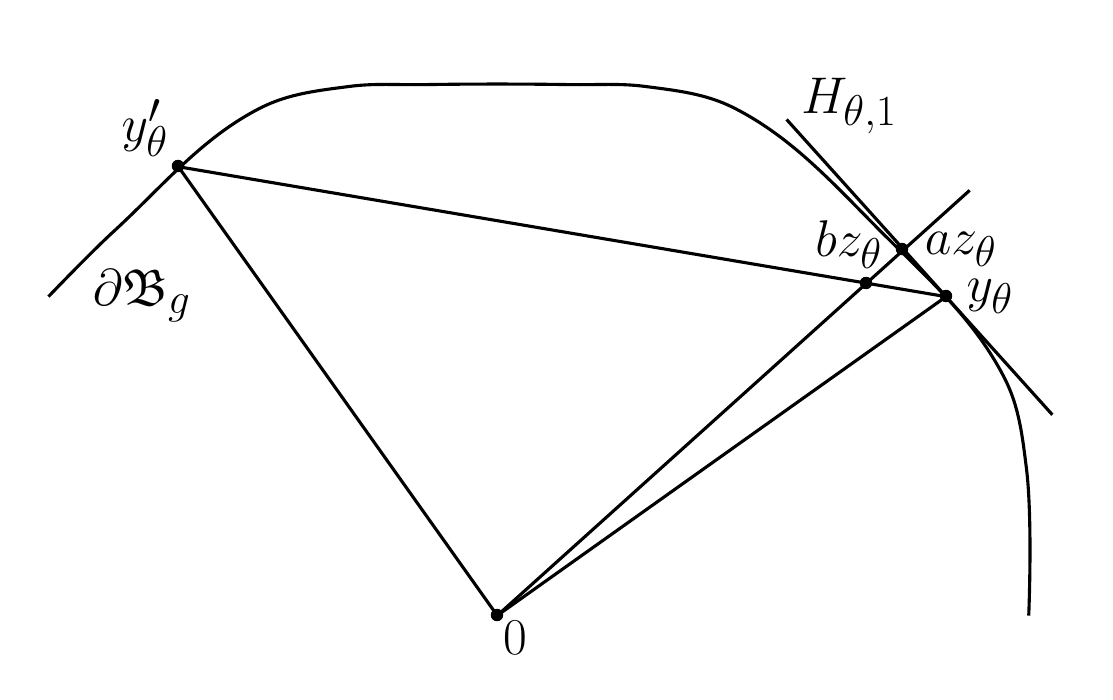}
\caption{ Illustration for \eqref{yzangle}. $az_\theta$ lies in the tangent plane $H_{\theta,1}$ at $\yt$ and is perpendicular to $H_{\theta,1}$. $\{\yt,\yt'\}$ form an orthonormal basis of equal--length vectors. }
\label{fig220}
\end{figure}

Let $\psi_{ab}$ \label{ang} denote the angle, taken in $[0,\pi]$, between nonzero vectors $a$ and $b$.  
The vector $\theta$ (or its multiple $y_\theta$) and the vector $z_\theta \perp H_{\theta,0}$ need not be parallel, but we can bound the angle between them as follows. Let $a>0$ be such that $az_\theta\in H_{\theta,1}$; then $1 = az_\theta\cdot z_\theta$ so $a=1/|z_\theta|^2$.  See Figure \ref{fig220}. Also, $g(az_\theta)\geq 1$, and $az_\theta$ is the orthogonal projection of $y_\theta$ onto the line through 0 and $z_\theta$ so
\[
  y_\theta \cdot \frac{z_\theta}{|z_\theta|} = |az_\theta|.
\]
From lattice symmetry, there exists an othonormal basis for $\RR^d$ containing $y_\theta$ and consisting of vectors in $\partial\mkB_g$ having the same Euclidean length; by inverting basis vectors we may assume $z_\theta$ has all nonnegative coefficients in this basis.  Then the boundary of the convex hull of the basis vectors includes a multiple $bz_\theta$ with $b>0$; the convex hull 
is contained in $\mkB_g$ so we must have $b\leq a$.  The minimum Euclidean length of vectors in this convex hull boundary is $|y_\theta|/\sqrt{d}$, and hence
\begin{equation}\label{yzangle}
  \frac{y_\theta}{|y_\theta|} \cdot \frac{z_\theta}{|z_\theta|} = \frac{|az_\theta|}{|y_\theta|} \geq \frac{|bz_\theta|}{|y_\theta|} \geq \frac{1}{\sqrt{d}}
    \quad\text{so}\quad \psi_{y_\theta,z_\theta}=\psi_{\theta,z_\theta}\leq\arccos \frac{1}{\sqrt{d}}.
\end{equation}
Therefore $\psi_{y_\theta,z_\theta}=\psi_{\theta,z_\theta}\leq\arccos 1/\sqrt{d}$; alternatively we can say the angle between $\theta$ and $H_{\theta,0}$ is at least $\arcsin 1/\sqrt{d}$. This has several consequences.  First, for all $u\in\RR^d$,
\begin{equation}\label{dvsdth}
   \frac{|u_2^\theta|}{\sqrt{d}} = \frac{d_\theta(u,L_\theta)}{\sqrt{d}} \leq d(u,L_\theta) \leq |u_2^\theta|, \hskip .2cm
     |u_1^\theta| |\yt| \leq \sqrt{d}|u|, \hskip .2cm |u_2^\theta| \leq \sqrt{d-1}|u|, \hskip .2cm
     \Big|u\cdot\theta - u_1^\theta|y_\theta| \Big| \leq \sqrt{\frac{d-1}{d}}|u_2^\theta|
\end{equation}
(see Figure \ref{figdist}.) Second, there exists $\ep_5,\ep_6>0$ and $C_{21}$ as follows.  Suppose $\alpha,\theta\in S^{d-1}$ and the angle $\psi_{z_\alpha,z_\theta}$ between $H_{\alpha,0}$ and $H_{\theta,0}$ is at most $\ep_5$.  Suppose also that  for some $v$, $L_\theta(v)$ intersects $H_{\alpha,0}$ and $H_{\theta,0}$ in points $x_\alpha$ and $x_\theta$ respectively.  Then (see Figure \ref{figdist})
\begin{equation*}
  |x_\alpha - x_\theta| \leq C_{21}\psi_{z_\alpha,z_\theta}|x_\theta|,
\end{equation*}
which by \eqref{zangle} means that 
\begin{equation}\label{linegap2}
  \psi_{\alpha\theta}<\ep_6 \implies |x_\alpha - x_\theta| \leq C_{22}\psi_{\alpha\theta}|x_\theta|.
\end{equation}
In addition, for $\alpha'\in S^{d-1}$,
letting $w_\alpha=L_\alpha(v)\cap H_{\theta,0},w_{\alpha'}=L_{\alpha'}(v)\cap H_{\theta,0}$ we have  
\begin{equation}\label{Vgap}
  \psi_{\alpha\theta}<\ep_6, \psi_{\alpha'\theta}<\ep_6 \implies |w_\alpha - w_{\alpha'}| \leq C_{22}\psi_{\alpha\alpha'}|v - x_\theta|.
\end{equation}
Here and in what follows, for a line $L$, a hyperplane $H$, and a point $w$, we write $w = L\cap H$ as a shorthand for $\{w\} = L\cap H$.

\begin{figure}
\includegraphics[width=12cm]{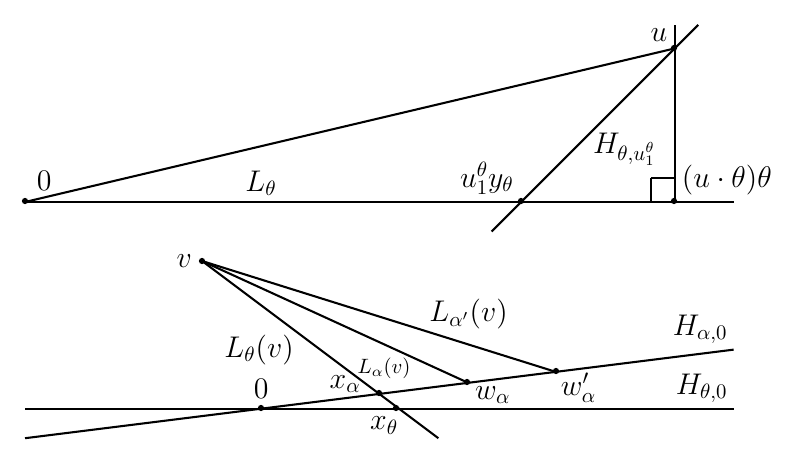}
\caption{ {\it Top:} Illlustration of the relationships in \eqref{dvsdth}. $u_1^\theta\yt$ and $u-u_1^\theta\yt$ are the $\theta$--components of $u$. $u_1^\theta\yt$ may instead lie on the opposite side of $(u\cdot\theta)\theta$.  {\it Bottom:} Illustration of \eqref{linegap2} and \eqref{Vgap}. The angle between $L_\theta(v)$ and $H_{\theta,0}$ is bounded below by $\arcsin 1/\sqrt{d}$. }
\label{figdist}
\end{figure}

Note that in both \eqref{linegap2} and \eqref{Vgap}, we can view the context as starting with the line $L_\theta(v)$ through $v$ that intersects the hyperplane $H_{\theta,0}$ at an angle of at least $\arcsin 1/\sqrt{d}$.  Equation \eqref{linegap2} bounds the change in the intersection point if we rotate the hyperplane around 0 from $H_{\theta,0}$ to $H_{\alpha,0}$, keeping the line fixed.  Equation \eqref{Vgap} bounds the change in the intersection point if we instead rotate the line through $v$ from direction $\alpha$ to $\alpha'$ (both near $\theta$), keeping the hyperplane fixed.

Another consequence of \eqref{linegap2} and \eqref{Vgap} is the following, relating change in $\theta$--coordinate values to change in the angle $\theta$.  When we change from $\theta$ to $\alpha$, in comparing $x_2^\theta$ to $x_2^\alpha$ it is not appropriate to simply consider $|x_2^\theta-x_2^\alpha|$, as these are coordinate vectors under different bases, used in different spaces ($H_{\theta,\fatdot}$ vs $H_{\alpha,\fatdot}$).  Instead we compare them in $\RR^d$ by considering $|(0,x_2^\theta)_\theta-(0,x_2^\alpha)_\alpha|$.

\begin{lemma}\label{coordchg}
Suppose A3 holds for some $\theta_0$ and $\ep_0>0$, and let $\ep_6$ be as in \eqref{linegap2}, \eqref{Vgap}.  There exist $C_i$ as follows. Suppose $\alpha,\theta\in S^{d-1}$ with $\psi_{\alpha\theta}\leq \ep_6$, and $0\neq x\in\RR^d$.  Then
\begin{equation}\label{coordchg1}
  \max\left( |x_1^\theta-x_1^\alpha|, |(0,x_2^\theta)_\theta-(0,x_2^\alpha)_\alpha| \right) \leq C_{23}\psi_{\alpha\theta}|x|.
\end{equation}
\end{lemma}

\begin{proof}
Let $0\neq x\in\RR^d$ and $\alpha,\theta\in S^{d-1}$ with $\psi_{\alpha\theta}\leq \ep_6$.  Let $w_\theta= x_1^\theta\yt$ be the first $\theta$--component of $x$, so that
\begin{equation}\label{compsize}
  |x-w_\theta| = |x_2^\theta|,\quad |w_\theta| = |\yt| |x_1^\theta|.
\end{equation}
See Figure \ref{fig2-3}. 
Similarly let $w_\alpha=x_1^\alpha y_\alpha$, and let $q=L_\theta\cap H_{\alpha,x_1^\alpha}$ (noting $x\in H_{\theta,x_1^\theta}$.)  Then from \eqref{zangle}, \eqref{dvsdth}, \eqref{linegap2}, and \eqref{Vgap},
\begin{align}\label{coord1}
  |\yt| |x_1^\theta - x_1^\alpha| &= |w_\theta - x_1^\alpha \yt| \notag\\
  &\leq |w_\theta - q| + |q-w_\alpha| + |x_1^\alpha|\ |y_\alpha - \yt| \notag\\
  &\leq C_{22}\psi_{\alpha\theta}|w_\theta| + C_{22}\psi_{\alpha\theta}|w_\alpha| + c_1\psi_{\alpha\theta}|x_1^\alpha| \notag\\
  &\leq c_2\psi_{\alpha\theta}|x|, 
\end{align}
and using the last two inequalities in \eqref{coord1},
\begin{align} \label{coord2}
  |(0,x_2^\theta)_\theta-(0,x_2^\alpha)_\alpha| &= |(x-w_\theta) - (x-w_\alpha)| \leq |w_\theta - q| + |q-w_\alpha| \leq c_2\psi_{\alpha\theta}|x|.
\end{align}

\begin{figure}
\includegraphics[width=10cm]{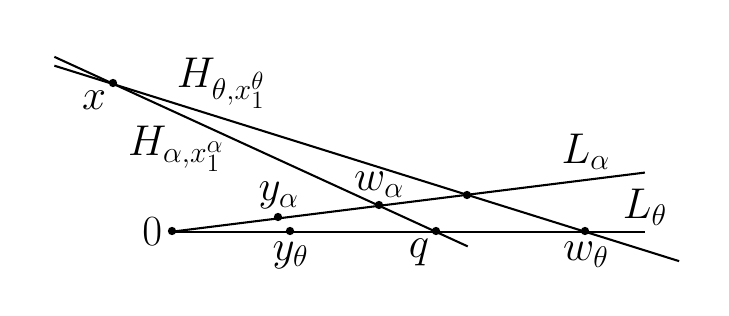}
\caption{ Illustration for the proof of Lemma \ref{coordchg}. $y_\alpha,\yt$ lie in $\partial \mkB_g$. }
\label{fig2-3}
\end{figure}

\end{proof}

We can use \eqref{dvsdth} to relate the $\theta$-ratio of $u$ to the tangent of $\psi_{u\theta}$:
\begin{equation}\label{tanvsratio}
  u_1^\theta>0,\ \frac{|u_2^\theta|}{u_1^\theta} \leq \frac{|y_\theta|}{2} \implies \tan\psi_{u\theta} = \frac{d(u,L_\theta)}{u\cdot\theta} \leq 
    \frac{|u_2^\theta|}{u_1^\theta|y_\theta| - |u_2^\theta|} \leq \frac{2}{|y_\theta|}\ \frac{|u_2^\theta|}{u_1^\theta}.
\end{equation}
In the other direction, by \eqref{yzangle}, the tangent of angle between $\theta$ and $u-u_1^\theta\yt$ has magnitude at least $1/\sqrt{d-1}$ so letting $w$ be the closest point to $u$ in $L_\theta$, satisfying $|w|=u\cdot\theta$, we have
\[
  |u_1^\theta\yt - w| \leq \sqrt{d-1}|u-w| = \sqrt{d-1}|w|\tan \psi_{u\theta}
\]
so (see Figure \ref{figdist}, top diagram)
\begin{align}\label{tanvsratio2}
  u\cdot\theta>0,\ \tan\psi_{u\theta} = \frac{d(u,L_\theta)}{u\cdot\theta} \leq \frac{1}{2\sqrt{d-1}} &\implies \frac{|u_2^\theta|}{|y_\theta|u_1^\theta}
    \leq \frac{\sqrt{d}|u-w|}{|w| - |u_1^\theta\yt - w|} \leq \frac{2\sqrt{d}d(u,L_\theta)}{u\cdot\theta} \notag\\
  &\implies \frac{|u_2^\theta|}{u_1^\theta} \leq 2\sqrt{d}|y_\theta| \tan \psi_{u\theta}.
\end{align}

Let $\mu=g(e_1)$. \label{mug} Convexity and lattice symmetry yield that $\mu\mkB_g$ contains the $\ell^1$--unit ball in $\RR^d$ and is contained in the $\ell^\infty$--unit ball, so
\begin{equation}\label{gsize}
   \frac{\mu}{\sqrt{d}}|x| \leq g(x) \leq \mu\sqrt{d}|x| \quad\text{for all } x\in\RR^d.
\end{equation}
In addition, from the triangle inequality,
\begin{equation}\label{magsize}
  |x_1^\theta|\ |y_\theta| - |x_2^\theta| \leq |x| \leq |x_1^\theta|\ |y_\theta| + |x_2^\theta|.
\end{equation}
Finally, given $\lambda\geq 1$, by \eqref{sigmaineq}, for some $C_{24}(\lambda)$, for all $s\geq C_{24}$,
\begin{equation}\label{XiPhi2}
  \Phi(s) \geq \lambda^{(1-\chi)/2} \Phi\left( \frac s\lambda \right).
\end{equation}
Hence suppose $c$ is large and $x$ lies in the tube portion of $E_{\theta,r,c}$, that is, $|x_2^\theta|\leq c^{1/2}\Xi(x_1^\theta)$, and suppose $\lambda \Phi^{-1}(c) \leq x_1^\theta \leq r/2$ for some $\lambda\geq 1$. Then using \eqref{XiPhi} and \eqref{XiPhi2}, the $\theta$--ratio of $x$ satisfies
\begin{equation}\label{tuberatio}
  \frac{|x_2^\theta|}{x_1^\theta} \leq \frac{c^{1/2}\Xi(x_1^\theta)}{x_1^\theta} \leq \frac{c^{1/2}}{\Phi(\lambda\Phi^{-1}(c))^{1/2}}
    \leq \lambda^{-(1-\chi)/4}.
\end{equation}

Define the slabs \label{oth}
\[
  \Omega_\theta(s,t) = \{x\in\RR^d: s\leq x_1^\theta \leq t\}.
\]
Of particular interest are $H_{\theta,s}^{\rm fat} := \Omega_\theta(s,s+\mu\sqrt{d})$ and $H_{\theta,s}^{\rm rfat} := \Omega_\theta(s-\mu\sqrt{d},s)$, \label{fat} which we call the \emph{ fattened} and \emph{backwards-fattened} $H_{\theta,s}$, respectively; generically we call any such slab of thickness $\mu\sqrt{d}$ a \emph{fattened hyperplane}.  This thickness is chosen so that, by \eqref{gsize}, any lattice path crossing a fattened hyperplane must have at least one site in it.  If $x\in H_{\theta,s}^{\rm rfat} \cup H_{\theta,s}^{\rm fat}$ then by \eqref{gsize} the $\theta$-projection $\hat x$ of $x$ into $H_{\theta,s}$ satisfies 
\begin{equation}\label{closeproj}
  |x-\hat x| \leq \frac{\sqrt{d}}{\mu}g(x-\hat x) \leq d.
  \end{equation} 
More generally, for a set $B$ contained in some $H_{\theta,s}$ we write $B^{\rm rfat}$ and $B^{\rm fat}$ for $[s-\mu\sqrt{d},s]\times B$ and $[s,s+\mu\sqrt{d}]\times B$ (in $\theta$--coordinates), respectively.  The values $\theta,s$ will be uniquely determined by $B$ in all instances here.

Given $\delta>0$, a geodesic $\Gamma_{xy}$, and a site $u$ preceding a site $v$ in $\Gamma_{xy}$, we say that $x,u,v$ form a $\delta$-\emph{fat triangle} in $\Gamma_{xy}$ if $d(u,\Pi_{xv}) \geq \delta |v-x|$. For $0<\delta<1/2$ it follows straightforwardly from \eqref{triangle} that 
\begin{equation}\label{extradist}
  |u-x| + |v-u| - |v-x| \geq (\delta^2\wedge\delta) |v-x|,
\end{equation}
that is, the extra distance associated with this triangle is at least $(\delta^2\wedge\delta) |v-x|$.  An analog for $g$ is the following variant of \eqref{triangleg}.

\begin{lemma}\label{gtri}
Suppose A3 holds for some $\theta_0$ and $\ep_0>0$.  There exists $C_{25}$ as follows. For all $\delta>0$, all $\theta\in S^{d-1}$ with $\psi_{\theta\theta_0}<\ep_0$, all $u,v\in\RR^d$ with $v/|v|=\theta$ and $d(u,\Pi_{0v})\geq \delta |v|$, we have
\begin{equation}\label{gtri1}
  g(u) + g(v-u) - g(v) \geq C_{25}(\delta^2\wedge\delta)|v|.
\end{equation}
\end{lemma}

\begin{proof}
Let $u,v$ be as in the lemma statement, so $g(v)=v_1^\theta$. Consider first the case of $g(u) \geq g(v)$; we then have from \eqref{gsize} that
\begin{equation}\label{gugv}
  g(u) + g(v-u) \geq g(v) + \frac{\mu}{\sqrt{d}}|v-u| \geq g(v) + \frac{\delta\mu}{\sqrt{d}}|v|.
\end{equation}

Consider next $u\notin \Omega_\theta(0,g(v))$, noting that this slab has $0,v$ in its boundary. From symmetry we may assume $u\in H_{\theta,g(v)}^+$.  But then $g(u)\geq g(v)$ so \eqref{gugv} applies.

\begin{figure}
\includegraphics[width=10cm]{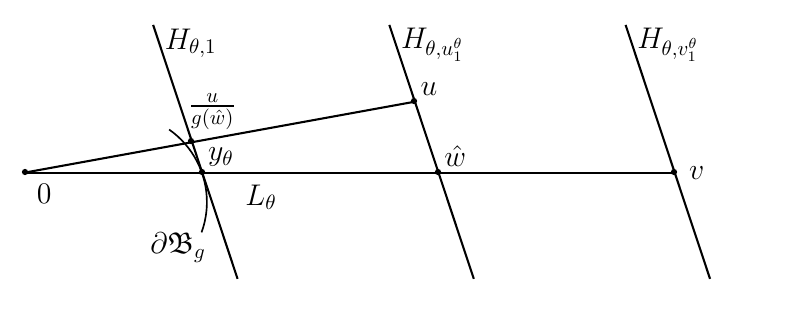}
\caption{ Illustration for the proof of Lemma \ref{gtri}. }
\label{fig3-4}
\end{figure}

Finally consider $g(u)<g(v)$ with $u\in \Omega_\theta(0,g(v))$, so $u_1^\theta \in [0,g(v)]$.  See Figure \ref{fig3-4}. We let $\hat w = u_1^\theta y_\theta$ be the first $\theta$--component of $u$. 
Then $\hat w \in \Pi_{0v}$ so from symmetry we may assume $|\hat w|\geq |v|/2$, so that using \eqref{gsize},
\begin{equation}\label{offtheta}
  \left| \frac{u}{g(\hat w)} - y_\theta \right| = \frac{|u-\hat w|}{g(\hat w)} 
    \geq \frac{d(u,\Pi_{0v})}{g(v)} \geq \frac{\delta}{\mu\sqrt{d}}.
\end{equation}
It therefore follows from \eqref{offtheta} and \eqref{curv3} that
\[
  d\left( \frac{u}{g(\hat w)},\mkB_g \right) \geq C_5\ep_0^2\left( \frac{c_1\delta}{\ep_0} \wedge \left( \frac{c_1\delta}{\ep_0} \right)^2
    \right) \geq c_2 (\delta\wedge\delta^2),
\]
and then using \eqref{gsize}, since $g(\hat w) \geq g(v)/2$,
\begin{equation}\label{guv}
  g(u) \geq g(\hat w)\left( 1 + d_g\left( \frac{u}{g(\hat w)},\mkB_g \right) \right) 
    \geq g(\hat w)\left( 1 + \frac{c_2\mu}{\sqrt{d}}(\delta\wedge\delta^2) \right) \geq g(\hat w) + c_3(\delta\wedge\delta^2)|v|.
\end{equation}
Now $H_{\theta,u_1^\theta}$ contains $u$ and is tangent to the boundary of the translate $v+g(\hat w-v)\mkB_g$ at $\hat w$.  It follows that $g(u-v) \geq g(\hat w - v)$.  
Therefore using \eqref{guv}, 
\[
  g(u) + g(v-u) \geq g(\hat w) + g(\hat w -v) + c_3(\delta\wedge\delta^2)|v|\geq g(v) + c_3(\delta\wedge\delta^2)|v|,
\]
as desired.
\end{proof}

The following is a purely deterministic result about norms on $\ZZ^d$ when the local curvature assumption is satisfied.

\begin{lemma}\label{gdiff1}
Suppose the norm $g$ satisfies the local curvature assumption A3' for some $\theta_0,\ep_0$.  There exist constants $\ep_7$ and $C_i$ as follows.  Suppose $\ell>C_{28}$, $\psi_{\theta\theta_0}<\ep_7$, and $u,v \in H_{\theta,\ell}$ with 
\begin{equation}\label{ep1bound}
  \frac{|u-\ell y_\theta|}{\ell} < \ep_7, \quad  \frac{|v-u|}{\ell} < \ep_7.
\end{equation}
Then
\begin{equation}\label{gdiff2}
  |g(v)-g(u)|\leq C_{29}\left( \frac{|v-u|\,|u-\ell y_\theta|}{\ell} + \frac{|v-u|^2}{\ell} \right).
\end{equation}
\end{lemma}

\begin{proof}
We bound $g(v)-g(u)$ as the opposite bound is nearly symmetric.  Let $\alpha = u/|u|$.  
By A3 and \eqref{zangle}, provided we take $\ep_7$ small, the first inequality in \eqref{ep1bound} guarantees that the angle between $H_{\theta,0}$ and $H_{\alpha,0}$ is at most $c_1|u-\ell\yt|/\ell$. Since $v-u\in H_{\theta,0}$, it follows that the orthogonal projection $a$ of $v-u$ into $H_{\alpha,0}$ satisfies $|a| \leq |v-u|$ and 
\begin{align}\label{adist1}
  |(v-u)-a| \leq |v-u|\sin\left( \frac{c_1|u-\ell\yt|}{\ell} \right) \leq c_1\frac{|v-u|\,|u-\ell y_\theta|}{\ell}
\end{align} 
(see Figure \ref{fig2-5}.) We have
\[
  \sin \psi_{\theta,u} \leq \frac{|u-\ell\yt|}{\ell |\yt|} < \frac{\ep_7}{|\yt|},
\]
so provided $\ep_7$ is small enough,
\begin{equation}\label{uangle}
  \psi_{u,\theta_0} \leq \psi_{u\theta} + \psi_{\theta\theta_0} \leq \left(\frac{2}{|\yt|} + 1 \right) \ep_7 < \ep_0.
\end{equation}
Now $u+a$ lies in the tangent plane $H_{\alpha,u_1^\alpha}$ to $g(u)\mkB_g$ at $u$, with
\[
  \frac{|a|}{g(u)} \leq  \frac{|v-u|}{\ell} < \ep_7,
\]
so by A3', \eqref{gsize}, and \eqref{uangle},
\[
  g(u+a) - g(u) = d_g(u+a,g(u)\mkB_g) \leq \mu\sqrt{d}\,d(u+a,g(u)\mkB_g) \leq c_2 \frac{|a|^2}{g(u)} \leq c_2\frac{|v-u|^2}{\ell}.
\]
\begin{figure}
\includegraphics[width=10cm]{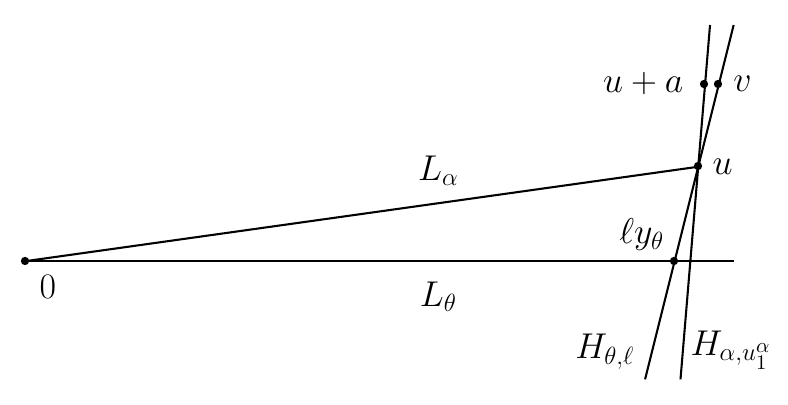}
\caption{ Illustration for the proof of Lemma \ref{gdiff1}. }
\label{fig2-5}
\end{figure}
Combining this with \eqref{gsize} and \eqref{adist1} yields
\[
  g(v) - g(u) \leq g(u+a) + g((v-u) -a) - g(u)\leq c_2\frac{|v-u|^2}{\ell} + c_3\frac{|v-u|\,|u-\ell y_\theta|}{\ell}.
\]
\end{proof}

\subsection{Cost of basic types of bad geodesic behavior} \label{cost}
In \eqref{triangleg} we can view the ``min'' term as a lower bound for the cost, in extra distance, of deviating by $|u_2^\theta|$ from a target point at $g$-distance $u_1^\theta$ in direction $\theta$.  To obtain a probability cost of such deviation by a geodesic, in view of \eqref{expbd} roughly we can divide the extra distance by $\sigma(u_1^\theta)$; we will incorporate an extra log factor to handle the entropy that arises when handling many scales of $|u|$ with a single bound.  Keeping this in mind, we define first $\sigma^*(s)$ and $\Phi(s)$ by \label{cphi}
\begin{equation}\label{Phidef}
    \Phi(s) = \frac{s}{C_3\sigma^*(s)\log(2+s)} 
    = \frac{s^{\chi_2}}{C_3} \sup_{t\leq s} \frac{t^{1-\chi_2}}{\sigma(t)\log(2+t)}.
\end{equation}
Here factoring out a power of $s$ ensures that $\Phi$ is strictly increasing, and $C_3,\chi_2$ are from \eqref{powerlike}.  Note that by \eqref{powerlike} we have
\begin{equation}\label{sigmaineq}
  C_3^{-1}\sigma(s) \leq \sigma^*(s) \leq \sigma(s),
\end{equation}
the first inequality being valid for $s\geq C_{19}$ for some $C_{19}$.  We then define \label{cxi}
\begin{equation}\label{Ddef}
  \Xi(s) = (s\sigma(s) \log(2+s))^{1/2}, \quad
  D_\theta(u) = \begin{cases} \min\left( \frac{|u_2^\theta|^2}{\Xi(u_1^\theta)^2}, \Phi(|u|_{\theta,\infty}) \right) 
    &\text{if } u_1^\theta\geq 0,\\\Phi(|u|_{\theta,\infty}) &\text{if } u_1^\theta<0. \end{cases}
\end{equation}
Roughly speaking, if we ignore the above-mentioned entropy-controlling log factors in $\Phi$ and $\Xi$, then for a $\theta$--ray or a finite geodesic with ultimate direction $\theta$, $D_\theta(u)$ represents the cost of that geodesic deviating from direction $\theta$ to pass through $u$.  When the direction of $u$ is far from $\theta$, this cost has form $\Phi(|u|_{\theta,\infty})$, and when it is close to $\theta$ the cost has form $|u_2^\theta|^2/\Xi(u_1^\theta)^2$, as formalized in Proposition \ref{transfluct2}.
Note that by \eqref{powerlike} and \eqref{sigmaineq}, for some $C_{20}$, for all $s\geq C_{20}$,
\begin{equation}\label{XiPhi}
  \frac{1}{C_3\Phi(s)} \leq \left( \frac{\Xi(s)}{s} \right)^2 \leq \frac{1}{\Phi(s)}.
\end{equation}
This tells us in part which term in the ``min'' in \eqref{Ddef} is smaller: we have for $|u|\geq 2C_{20}$ that
\begin{equation}\label{minwhich}
  D_\theta(u) = \begin{cases} \Phi(|u|_{\theta,\infty}) &\text{if } |u_2^\theta|\geq u_1^\theta,\\ \frac{|u_2^\theta|^2}{\Xi(u_1^\theta)^2} &\text{if } |u_2^\theta|\leq C_3^{-1/2}u_1^\theta. \end{cases}
\end{equation}
Here we used the fact that $0\leq u_1^\theta\leq |u_2^\theta|$ implies $|u|_{\theta,\infty}= |u_2^\theta|$, and $|u_2^\theta|\leq u_1^\theta$ implies $|u|_{\theta,\infty}=u_1^\theta$.

Let $\Pi_{xy}$ denote the line segment from $x$ to $y$. \label{pil} To deal with paths from 0 to some $ry_\theta$ it is useful to have the following symmetric version of $D_\theta$:
\begin{equation}\label{Dtr}
  D_{\theta,r}(u) = \begin{cases} D_\theta(u) &\text{if } u_1^\theta \leq \frac r2\\ D_\theta(r\yt-u) &\text{if } u_1^\theta > \frac r2. \end{cases}
\end{equation}
This makes the right half of the level set \label{tcyl}
\[
  E_{\theta,r,c} := \{u:D_{\theta,r}(u)\leq c\}
\]
symmetric with the left half; this region is a ``tube'' surrounding $\Pi_{0,ry_\theta}$ bounded by the rounded--cylindrical shell $\{u:|u_2^\theta|=c^{1/2}\Xi(u_1^\theta)\}$, augmented by a ``tilted cylinder'' around each endpoint; see Figure \ref{fig2-4}, excluding Case 2b.  (In a mild abuse of terminology, we will simply call a tilted cylinder a cylinder.)  The cylinder around 0 is \label{zcyl}
\[
  \CC_{\theta,c} = \{(u_1^\theta,u_2^\theta): |u_1^\theta| \leq \Phi^{-1}(c),  |u_2^\theta| \leq \Phi^{-1}(c)\}.
\]
We call this the \emph{0--cylinder of} $E_{\theta,r,c}$; it has one \emph{inside end} in the hyperplane $H_{\theta,\Phi^{-1}(c)}$, and an \emph{outside end} in $H_{\theta,-\Phi^{-1}(c)}$.  We thus call $E_{\theta,r,c}$ a \emph{tube-and-cylinders region}. By monotonicity of $\Phi$ and $\Xi$, $\{u:D_{\theta,r}(u)=c\}$ is the boundary of the tube-and-cylinders region.  

By \eqref{minwhich}, on $\{u: 0<u_1^\theta=|u_2^\theta|\}$ (which is a cone boundary), $\Phi(|u|_{\theta,\infty})$ is the ``min'' in \eqref{Ddef}; this uses the fact that $u_1^\theta=|u|_{\theta,\infty}$ on that cone boundary. This means that the boundary of the 0--cylinder meets the cylindrical shell in the inside end of the 0--cylinder, the intersection being the $(d-2)$-sphere of radius $c^{1/2}\Xi(\Phi^{-1}(c))$ around $\Phi^{-1}(c)y_\theta$ in $H_{\theta,\Phi^{-1}(c)}$; we call this $(d-2)$-sphere $\mathbb{S}_\theta(c)$.  In Figure \ref{fig2-4}, $\mathbb{S}_\theta(c)$ is just the two points where the shell meets the 0--cylinder.

The bound \eqref{yzangle} on the angle between $y_\theta$ and $z_\theta$ also gives information about the 0--cylinder of $E_{\theta,r,c}$.  If $u$ lies in either end of $\CC_{\theta,c}$ then $|u| \geq |y_\theta|\Phi^{-1}(c)/\sqrt{d}$, while if $u$ lies in the side of $\CC_{\theta,c}$ then $|u| \geq \Phi^{-1}(c)/\sqrt{d}$.  Thus
\begin{equation}\label{cylbdry}
 u \in \partial \CC_{\theta,c} \implies \frac{|y_\theta| \wedge 1}{\sqrt{d}} \Phi^{-1}(c) \leq |u| \leq (|y_\theta|+1)\Phi^{-1}(c).
\end{equation}

The proof of the next proposition is based on the fact that if a path $\gamma$ from 0 to some site $r\yt$ contains a site $u$ with $D_{\theta,r}(u) \geq t$, then there are necessarily 3 sites in $\gamma$ (one of which is an endpoint, 0 or $r\yt$) which form a $\delta$--fat triangle for some appropriate $\delta$, and Lemma \ref{gtri} can be used to help show that the probablity of this is small.  Analogous results based on the same general principle appear in \cite{BHS18} for an integrable last passage percolation model in $d=2$, and in \cite{Ga19} for FPP in $d=2$ under hypotheses similar to ours here.

The next two results are proved in Appendix \ref{basic}.  The first underlies all our results on regularity of geodesics.

\begin{proposition}\label{transfluct2}
Suppose A1, A2', A3 hold for some $\theta_0$ and $\ep_0>0$.  There exist constants $C_i$ as follows. For all $r,\theta$ with $\psi_{\theta\theta_0}<\ep_0$ and $0\neq r\yt\in\ZZ^d$, and all $t>0$, 
\begin{equation}\label{transfluct1}
  P\left( \max_{u\in\Gamma_{0,ry_\theta}} D_{\theta,r}(u) \geq t \right) 
    \leq C_{26}e^{-C_{27}t\log t}.
\end{equation}
\end{proposition}

We next consider the transverse increments of $T(0,u)$, that is, we bound $|T(0,u) - T(0,v)|$ when $|g(u)-g(v)|$ is small.  (We can't require it be exactly 0 since $u,v$ are lattice points.) 
Heuristically, assuming $|u-v| \ll \Delta(|u|)$, $\Delta^{-1}(|u-v|)$ may be viewed as the typical distance traveled by $\Gamma_{u0}$ and $\Gamma_{v0}$ before they can get close enough to coalesce.  Then $\sigma(\Delta^{-1}(|u-v|))$ becomes the scale of ``fluctuations before coalescing,'' and we show that $|T(0,u) - T(0,v)|$ is unlikely to be much larger than this scale.  A variant of the following proposition, for $d=2$, appears in \cite{Ga19}.

\begin{proposition}\label{transTincr}
Suppose A1, A2', A3 hold for some $\theta_0,\ep_0$, and let $\ep_7$ be as in Lemma \ref{gdiff1}.  There exist constants $C_i$ as follows. For all $u,v\in \ZZ^d$ with 
\begin{equation}\label{uvassump}
  |u|\geq C_{30}, \quad \left|\theta_0 - \frac{u}{|u|}\right|<\ep_7, \quad |g(u)-g(v)| \leq 4\mu d, 
    \quad\text{and}\quad |u-v|\leq C_{31}\Delta(|u|),
\end{equation}
and all $\lambda\geq C_{32}$, we have
\begin{equation}\label{Tchange}
  P\Big(T(v,0) - T(u,0) \geq \lambda\sigma(\Delta^{-1}(|u-v|))\log |u-v| \Big) \leq C_{33}e^{-C_{34}\lambda\log |u-v|}.
\end{equation}
\end{proposition}

\section{Existence and transverse fluctuations of $\theta$--rays}\label{rays}

For sites $u,v$ in a path $\gamma$, let $\gamma[u,v]$ \label{gab} denote the segment of $\gamma$ between $u$ and $v$.
For a geodesic ray $\Gamma = (v_0,v_1,\dots)$ (as a sequence of sites), we say a sequence $\{\Gamma_n\}$ of geodesics or geodesic rays from $v_0$ \emph{converges to} $\Gamma$ if for each $j\geq 1$, for all sufficiently large $n$, $\Gamma[v_0,v_j]$ is an initial segment of $\Gamma_n$. If $\{\Gamma_n\}$ is a sequence of geodesics or geodesic rays from a fixed $v_0$ with length $|\Gamma_n|\to\infty$, then $\{\Gamma_n\}$ has a converging subsequence.

In Proposition \ref{transfluct2}, the bound on the probability is uniform in $r$.  This enables us to turn that lemma (or more precisely, its proof) into a result about $\theta$--rays, which is part (ii) of the next proposition.  The proposition also includes parts (i), (iv)(a), and (iv)(b) of Theorem \ref{alldim}.  

\begin{proposition}\label{rayfluct} Suppose A1, A2', A3 hold for some $\theta_0,\ep_0$.

(i)
\begin{align}\label{rayexist}
  P\big(&\text{for all $v\in \ZZ^d$ and $\theta\in S^{d-1}$ with $\psi_{\theta\theta_0}<\ep_0$, a $\theta$-ray from $v$ exists}) = 1.
\end{align}
If also A3' holds then this is true without the condition $\psi_{\theta\theta_0}<\ep_0$.

(ii) There exist constants $C_i$ as follows.  For $t>1$,
\begin{align}\label{rayfluct1}
  P\Big(&\text{for some $\theta\in S^{d-1}$ with $\psi_{\theta\theta_0}<\ep_0$, there exists a $\theta$-ray $\Gamma$ from 0 with }  \notag\\
  &\qquad \sup_{u\in\Gamma} D_\theta(u) > t \Big) \leq C_{35}e^{-C_{36}t\log t}.
\end{align}
If also A3' holds then this is true without the condition $\psi_{\theta\theta_0}<\ep_0$.

(iii) 
\begin{align}\label{onedirec}
  P\big(&\text{for every $v\in \ZZ^d$ and $\theta\in S^{d-1}$ with $\psi_{\theta\theta_0}<\ep_0$, every subsequential $\theta$-ray} \notag\\
  &\qquad \text{ from $v$ is a $\theta$--ray}) = 1.
\end{align}
If also A3' holds then
\begin{align}\label{nodirec}
  P\big(&\text{there exists a geodesic ray with no asymptotic direction} \big) = 0.
\end{align}
\end{proposition}

\begin{proof}
Observe that event in \eqref{rayfluct1} is contained in the event
\begin{align*}
  A_t: &\text{ there exist $\theta\in S^{d-1}$ with $\psi_{\theta\theta_0}<\ep_0$ and $z_n \in \ZZ^d,$ with } \\
  &\qquad |z_n|\to\infty, \quad \theta_n = z_n/|z_n| \to \theta, \text{ and } \sup_{u\in\Gamma_{0,z_n}} D_{\theta_n}(u) > t.
\end{align*}
This can be seen by fixing $\Gamma= (v_0,v_1,\dots)$ (as a sequence of sites) and $\theta$ as in \eqref{rayfluct1}, and a site $u=v_m\in \Gamma$ with $D_\theta(u) > t$, taking $z_n=v_{m+n}$ for all $n$, and noting that $D_{\theta_n}(u) \to D_\theta(u)$ as $n\to\infty$.

Further, if $\Gamma= (v_0,v_1,\dots)$ is a subsequential $\theta$--ray which is not a $\theta$--ray, then there are subsequences $v_{n(k)}/|v_{n(k)}| \to \theta$ and $\zeta(j)=v_{n'(j)}/|v_{n'(j)}| \to \zeta$ for some $\theta\neq\zeta$.  But this means that given $t>0$, fixing $k$ sufficiently large we have $D_{\zeta(j)}(v_{n(k)}) > t$ for all sufficiently large $j$ (depending on $k$), which in turn means that $A_t$ occurs with $z_j = v_{n'(j)}$.  It follows that the complement of the event in \eqref{onedirec} is contained in $\cap_{t>0} A_t$.

Therefore to prove both \eqref{rayfluct1} and \eqref{onedirec} it is enough to show  
\begin{equation}\label{Atbound}
  P(A_t) \leq C_{35}e^{-C_{36}t\log t}.
\end{equation}

Note that for fixed $z_n$ we can use Proposition \ref{transfluct2}, but we cannot sum over possible $z_n$ as the entropy is too large.
However, in the proof of that lemma, all that we use is (after converting the notation for our present context) the existence of a $\delta$-fat triangle of sites $0,u,w$ in $\Gamma_{x,z_n}$ for sufficiently large $\delta$ (depending on $|u|,|w|$.) The bounds in that proof do not involve $z_n$, so in the 4 cases there it is only necessary to sum over ranges of possible values of $|u|$ or $|w|$.  Thus the proof of Proposition \ref{transfluct2} also proves \eqref{Atbound}.

The last sentence in (ii) follows from \eqref{rayfluct1} and the compactness of $S^{d-1}$, and similarly in (iii).

Turning to (i), it is enough to consider $v=0$. Given $\theta\in S^{d-1}$ with $\psi_{\theta\theta_0}<\ep_0$, let $z_n\in\ZZ^d$ with $|z_n|\to\infty$ and $\theta_n=z_n/|z_n|\to\theta$. Then
some subsequence of $\{\Gamma_{0z_n}\}$ converges to a geodesic ray $\Gamma_\infty=(0=w_0,w_1,\dots)$ from 0.  If $\Gamma_\infty$ is not a $\theta$--ray, it is a subsequential $\zeta$--ray for some $\zeta\in S^{d-1}, \zeta\neq\theta$.  It follows readily that 
\[
  \limsup_{j\to\infty} \limsup_{n\to\infty} D_{\theta_n}(w_j) = \infty,
\]
which means $A_t$ occurs for all $t>0$.  Thus $\Gamma_\infty$ is a $\theta$--ray a.s.  Equation \eqref{rayexist} then follows from \eqref{Atbound}, and the last sentence of (i) again follows from compactness of $S^{d-1}$.
\end{proof}

\section{Special types of bad geodesic behavior}\label{densely}

In this section we use the general ``bad geodesic behavior'' bounds of section \ref{cost} to control special types of bad geodesic behavior. Most proofs are deferred to Appendix \ref{spec}.

We call a path or geodesic from a set $A\subset \RR^d$ to $B\subset \RR^d$ \emph{nonreturning} if only
the first bond $\langle x_0x_1 \rangle$ (viewed as a line segment) intersects $A$ and only the last bond $\langle x_{m-1}x_m \rangle$ intersects $B$.  A nonreturning path from $H_{\theta,s_1}^-$ to $H_{\theta,s_2}^+$ for some $s_1<s_2$ is called a $\theta$--\emph{slab path}, and a $\theta$--\emph{slab geodesic} is a geodesic which is also a $\theta$--slab path.

In Theorem \ref{alldim}(ii), one can view the density of $H_{\theta_0,R}^+$--entry points as bounding the maximum possible longitudinal density of any set of halfspace $\theta_0$--rays with distinct $H_{\theta_0,R}^+$--entry points.  Given an upper bound for this density larger than its heuristically-suggested order of $\Delta(R)^{-(d-1)}$, we may express the bound as $n^{1+(d-1)\beta_0}/\Delta(R)^{d-1}$ for some $\beta_0>0$ and $n\geq 1$, which we may in turn read as $n\ \theta$--rays crossing per volume $(n^{-\beta_0}\Delta(R))^{d-1}$ in $H_{\theta,0}$.  Thus to bound the mean $H_{\theta_0,R}$--crossing density our main task is roughly to bound 
\begin{align}\label{keyprob}
  P&\Big( \text{there exist $n$ halfspace $\theta_0$--rays with distinct $H_{\theta_0,R}^+$--entry points} \notag\\
  &\qquad \text{originating from }\{0\}\times[-n^{-\beta_0} \Delta(R),n^{-\beta_0} \Delta(R)]^{d-1} \Big),
\end{align}
where the product is in $\theta_0$--coordinates. 
Obtaining such a bound with $n=(\pi(R)^2\log R)^K$ for some (large) $K$--as we will do in Proposition \ref{jammed1}---is similar, modulo a few technicalities, to bounding the mean $H_{\theta_0,R}$-crossing density by $(\pi(R)^2\log R)^{K'}/\Delta(R)^{d-1}$, with $K' = (1+(d-1)\beta_0)K$.  
Similarly, we can bound the mean combined $H_{\theta_0,R}$--crossing density of $\theta$--rays over $\theta\in J(\theta_0,\ep)$ for some $\ep$, mainly by bounding
\begin{align}\label{keyprob2}
  P&\Big( \text{there exist $n$ $\theta_0$--slab geodesics from $H_{\theta_0,0}^-$ to $H_{\theta_0,2R}^+$ with distinct 
    $H_{\theta_0,R}^+$--entry points } \notag\\
  &\qquad \text{originating from }\{0\}\times[-n^{-\beta_0} \Delta(R),n^{-\beta_0} \Delta(R)]^{d-1} 
    \text{ with initial orientation in } J(\theta_0,2\ep) \Big),
\end{align}
where by the initial orientation of a geodesic we mean its direction from its starting point to its $H_{\theta_0,R}^+$--entry point.

To bound something like \eqref{keyprob} or \eqref{keyprob2}, we need control of the probabilities of a number of special types of bad geodesic behavior, which we obtain in the rest of this section.

Given a geodesic $\Gamma$ from $H_{\theta,0}^-$ to $H_{\theta,R}^+$ for some $R$ and given $0\leq s\leq R$, let $x_{\theta,s}''(\Gamma)$ \label{xpri} denote the $H_{\theta,s}^+$--entry point of $\Gamma$, and let $x_{\theta,s}'(\Gamma)$ be the last site in $\Gamma$ before $x_{\theta,s}''(\Gamma)$. An $\ell$-\emph{segment} of $\Gamma$ is the segment $S_{\theta,i}(\Gamma) := \Gamma[x_{\theta,(i-1)\ell}''(\Gamma),x_{\theta,i\ell}''(\Gamma)]$ \label{seg} for some $i\geq 1$.

We first show that with high probability, a geodesic contains at least a certain minimum number of fast $\ell$--segments.  

\begin{lemma}\label{fastseg}
Suppose A1, A2', A3 hold for some $\theta_0,\ep_0$, and let $\eta\in(0,1)$.  There exist constants $C_i$ as follows.  Let 
\begin{equation}\label{Rell}
  R\geq C_{37}, \quad 0<\lambda\leq 1-\chi_2, \quad C_{38}(\log R)^{1/\lambda} \leq k \leq C_{39}R^{1/2}, \quad \ell=R/k,
\end{equation}
with $\chi_2$ from \eqref{powerlike} and $k$ an integer.  For $\psi_{\theta\theta_0}<\ep_0$ and $\Gamma$ a geodesic from $H_{\theta,0}^-$ to $H_{\theta,R}^+$, let \label{nth}
\[
  N_\theta(\Gamma) = \left|\left\{1\leq i\leq k: T(x_{\theta,(i-1)\ell}''(\Gamma),x_{\theta,i\ell}''(\Gamma)) 
    \leq \frac 1k ET(0,R \yt) + \frac{\eta\sigma(\ell)}{8} \right\}\right|.
\]
Then for all $v\in H_{\theta,0}^{\rm rfat}$ and all $w \in H_{\theta,R}^{\rm fat}$ with 
\begin{equation}\label{wclose}
  \max(|v|,|w-R\yt|) \leq C_{40}\eta^{1/2} k^{(1-\chi_2)/2}\Delta(R),
\end{equation}
we have
\begin{equation}\label{fastseg2}
  P\left(N_\theta(\Gamma_{vw}) \leq \frac{\eta}{32} k^{1-\lambda} \right) \leq C_{41}e^{-C_{42}k^\lambda\eta}.
\end{equation}
\end{lemma}

Note that \eqref{wclose} allows the angle $\psi_{\theta,w-v}$ to be nonzero but not too large.  Such a bound is necessary for \eqref{curvgap} in the proof.

\begin{remark}\label{betapowers} The results to follow involve several different length scales, and other quantities, expressed using $R$ and small (less than 1) powers of the number $n$ of geodesics we are dealing with.  We summarize them here for ready reference, with precise definitions to follow: \label{betas}
\begin{itemize}
\item[(i)] a length scale $n^{-\beta_0}\Delta(R)$ for cubic blocks in each hyperplane $H_{\theta,s}$, with geodesics considered close when they start and end with separation $n^{-\beta_0}\Delta(R)$ or less; 
\item[(ii)] a width $n^{\beta_1}\Delta(R)$ for the ``target box'' to which length--$R$ geodesics are confined, with high probability;
\item[(iii)] a length scale $\ell = n^{-\beta_2}R$ for segments of length-$R$ geodesics;
\item[(iv)] for a transition (i.e.~ choice of starting and ending blocks) made by a length--$\ell$ segment of a geodesic, a maximum number $n^{\beta_3}$ of other geodesics which can make the same transition, for that transition to be considered ``sparse'';
\item[(v)] a maximum overlap length $n^{-\beta_4}\ell$ for two length--$\ell$ geodesic segments to be considered ``low overlap''; this length also serves as the minimum significant backtrack in a geodesic.
\end{itemize}
Other quantities are defined in terms of these powers, such as a minimum number $n^{\beta_3-\beta_4}$ of geodesics passing through a site for it to qualify as what we will call a ``popular site.''
\end{remark}

We start with some formal definitions related to the items in Remark \ref{betapowers}.  

The following definitions are for fixed $R,n,\theta$, which don't necessarily appear in the notation.  For $0\leq s\leq R$, the \emph{home $\theta$--block} in a hyperplane $H_{\theta,s}$ is $\{s\}\times [-n^{-\beta_0}\Delta(R),n^{-\beta_0}\Delta(R)]^{d-1}$ (in $\theta$--coordinates.) \label{hblk} The home $\theta$--block in $H_{\theta,0}$ is denoted $B_{\theta,{\rm home}}$.
A $\theta$--\emph{block} in $H_{\theta,s}$ is a translate of the home $\theta$--block in each $\BB_{\theta}$--coordinate direction $i$ ($2\leq i\leq d$) by an integer multiple of $2n^{-\beta_0}\Delta(R)$ (so such $\theta$--blocks tile $H_{\theta,s}$.)  The center point of any $\theta$--block is called a $\theta$--\emph{block center}. An \emph{enlarged $\theta$--block} has the same center as a block but larger linear dimensions by a factor $2\sqrt{d}$; more precisely it has the form (in $\theta$--coordinates)
\[
  y + \left( \{s\}\times [-2\sqrt{d}n^{-\beta_0}\Delta(R),2\sqrt{d}n^{-\beta_0}\Delta(R)]^{d-1} \right),
\]
with $y$ a $\theta$--block center in $H_{\theta,s}$, and the $+$ denoting translation.  The factor $2\sqrt{d}$ ensures that if $u$ lies in a $\theta$--block with center $a$, and $v$ lies outside an enlarged $\theta$--block with center $b$, then $|v-b|\geq 2|u-a|$.

For $\Gamma$ a geodesic which starts in $H_{\theta,0}^-$ and crosses $H_{\theta,R}$, the \emph{pre-$H_{\theta,R}$ segment} of $\Gamma$ is the segment of $\Gamma$ from its starting point to $x_R''(\Gamma)$.  In what is to follow, as in Lemma \ref{fastseg}, we will divide the pre-$H_{\theta,R}$ segment of a $\theta$--slab geodesic  into sub-segments of some length $\ell$.  In this context, a geodesic $\Phi$ is an $(\ell,\theta)$-\emph{interval geodesic} if for some $i$, the initial site of $\Phi$ is in $H_{\theta,(i-1)\ell}^{\rm fat}$, and the last bond of $\Phi$ is the first bond of $\Phi$ to cross $H_{\theta,i\ell}$ (so the last site of $\Phi$ must lie in $H_{\theta,i\ell}^{\rm fat}$.) An $\ell$-segment of a $\theta$--ray (defined before Lemma \ref{fastseg}) is thus one example of an $(\ell,\theta)$--interval geodesic.  

The \emph{target $\theta$--box} is a tube around $L_\theta$: \label{tgt}
\[
  Q_{R,n,\theta} = \RR \times [-2n^{\beta_1}\Delta(R),2n^{\beta_1}\Delta(R)]^{d-1},
\]
in $\theta$--coordinates.  
A geodesic is $\theta$--\emph{target-directed} if it is contained in the target $\theta$--box.
A geodesic from $H_{\theta,0}^-$ to $H_{\theta,s}^+$ for some $s>0$ is $\theta$--\emph{target-directed up to $s$} if its pre-$H_{\theta,s}$ segment is target-directed.
A \emph{target $\theta$--block} is a $\theta$--block which intersects $Q_{R,n,\theta}$.

For $\theta\in S^{d-1}$, $\Gamma$ a geodesic with a designated direction, and sites $x$ preceding $y$ in $\Gamma$, we say $\Gamma$ has a $\theta$--\emph{backtrack of size} $r$ from $x$ to $y$ if $x_1^\theta - y_1^\theta \geq r$.

We may omit the parameters in the preceding terminology when it is clearly understood, e.g.~referring to a block rather than a $\theta$--block.

Lemma \ref{fastseg} controls one type bad geodesic behavior (not enough fast segments); we now proceed to other types.  The next lemmas involve the quantities $n^{\pm\beta_i}$ as in Remark \ref{betapowers}, so we now specify the relations that we assume to hold among these exponents.\\

\noindent{\bf A4. Exponent relations, parameters, and definitions.}  $\beta_i\in (0,1)$ $(i=0,1,2,3,4)$ satisfy
\begin{align}
  &\frac{2}{1+\chi_2}\beta_0 + \frac{3+\chi_2}{1+\chi_2}\beta_1 < \beta_1+\beta_2+\beta_4 < \frac{1-\chi_2}{(1+\chi_2)(d-1)}, \quad 
    \chi_1\beta_4 > \min(\chi_2,1-\chi_2)\beta_2 > 8\beta_1, \label{betas1} \\
  &(1-\chi_2)\beta_4 > (2+\chi_2)\beta_2,\quad \frac{2(1+\chi_2)}{\chi_1} \beta_1 + (1+\chi_2)\beta_2 < \beta_0
    <  \frac{\chi_1}{(1+\chi_2)(d-1)} \label{betas2} \\ 
  &3\beta_3 + 3\beta_2 + 7(d-1)(\beta_0+\beta_1) < 1, \label{betas3}
\end{align}
where $\chi_1<\chi<\chi_2$ are from \eqref{powerlike}. We assume tacitly throughout that $\chi_1,\chi_2$ have been chosen sufficiently close to $\chi$. Recalling the subpolynomial function $\pi$ from \eqref{zetadef},
$R,n,\ell$ satisfy 
\begin{equation}\label{Rncond}
  R\geq C_{43}, \quad C_{44}(\pi(R)^2\log R)^{C_{45}}\leq n\leq C_{46}\Delta(R)^{d-1}, \quad\text{and}\quad \ell = n^{-\beta_2}R, \text{ with } 
    C_{45}\geq \frac{2}{\beta_1}.
\end{equation}
Let $\ep_{\min}=\min\{\ep_0,\ep_1,\ep_5,\ep_6,\ep_7\}$, with $\ep_7$ from Lemma \ref{gdiff1} and the other $\ep_j$ from Section \ref{intro}.
$\ol y$ denotes a $\theta_0$--block center in $H_{\theta_0,R}$ for which $\theta = \ol y/|\ol y|$ satisfies $\psi_{\theta\theta_0}<\ep_{\min}/2$, and 
$B_{\theta_0,{\rm cross}}$ \label{blkc} is the $\theta_0$--block in $H_{\theta_0,R}$ with center $\ol y$.  $[s,r]$ with $0\leq s\leq r$ is the largest interval for which
\begin{equation}\label{rsdef}
  Q_{R,n,\theta} \cap \Omega_\theta(s,r) \subset Q_{R,n,\theta} \cap \Omega_{\theta_0}(0,R);
\end{equation}
see Figure \ref{fig4-10A}. $B_{r,\theta,{\rm home},+}$ and $B_{s,\theta,{\rm home},+}$ \label{blkr} are the enlarged home $\theta$--blocks in $H_{\theta,r}$ and $H_{\theta,s}$, respectively.  \\

The conditions \eqref{betas1}--\eqref{betas3} can be satisfied by choosing $\beta_3<1$, then temporarily setting $\beta_0=\beta_1=\beta_2=0$ and choosing $\beta_4>0$ to satisfy the inequalities where $\beta_4$ appears, then temporarily setting $\beta_1=\beta_2=0$ and choosing $\beta_0>0$ to satisfy the inequalities where $\beta_0$ appears, and finally choosing $\beta_2$ then $\beta_1$ similarly.

\begin{remark}
The conditions \eqref{Rncond} yield certain relative scales which we summarize here.  Assuming $R$ is large we have using \eqref{Rncond} and the last inequality in \eqref{betas2} that
\begin{equation}\label{nvsR}
  R \geq n^{2/(1+\chi_2)(d-1)}, \quad \frac{R}{\Delta(R)} = \left( \frac{R}{\sigma(R)} \right)^{1/2} 
    \geq R^{(1-\chi_2)/2} \geq n^{(1-\chi_2)/(1+\chi_2)(d-1)}, \quad \sigma(R) \geq R^{\chi_1} \geq n^{2\beta_0}.
\end{equation}
Putting together \eqref{nvsR} and the first inequality in \eqref{betas1} we obtain
\begin{equation}\label{nvsR2}
  R^{(1-\chi_2)/2} \geq n^{\beta_1+\beta_2+\beta_4} \geq n^{2\beta_1}.
\end{equation}
From \eqref{nvsR} and the second inequality in \eqref{betas1} we also get
\begin{equation}\label{nRell}
  n^{\beta_1+\beta_2+\beta_4} \leq \frac{R}{\Delta(R)} \quad\text{so}\quad \frac{n^{\beta_1}\Delta(R)}{\ell} \leq n ^{-\beta_4}.
\end{equation}
By \eqref{powerlike},
\begin{equation}\label{Deltaratio}
  \frac{\Delta(R)}{\Delta(\ell)} \leq C_{47} n^{\beta_2(1+\chi_2)/2}.
\end{equation}
From \eqref{betas1} and the second inequality in \eqref{betas2} we have
\[
  \beta_4 > \frac{8}{\chi_1}\beta_1 > \frac{2-\chi_1}{\chi_1}\beta_1 \quad\text{so}\quad 
    \frac{2\chi_1(1-\chi_2)}{(1+\chi_2)(d-1)} > 2\chi_1(\beta_1 + \beta_2 + \beta_4) > 4\beta_1 + \chi_1(1-\chi_2)\beta_2
\]
which with \eqref{nvsR} yields
\begin{equation}\label{ell-lower}
  \ell^{\chi_1(1-\chi_2)} = R^{\chi_1(1-\chi_2)}n^{-\chi_1(1-\chi_2)\beta_2} 
    \geq n^{2\chi_1(1-\chi_2)/(1+\chi_2)(d-1)} n^{-\chi_1(1-\chi_2)\beta_2} \geq n^{4\beta_1}.
\end{equation}
Also from \eqref{powerlike}, \eqref{nvsR}, and the second inequality in \eqref{betas1},
\begin{equation}\label{ellsigma}
  \frac{\ell}{\sigma(\ell)} \geq C_2n^{-\beta_2(1-\chi_1)} \frac{R}{\sigma(R)} = C_2n^{-\beta_2(1-\chi_1)} \frac{R^2}{\Delta(R)^2}
    \geq C_2n^{-\beta_2(1-\chi_1)} n^{2(1-\chi_2)/(1+\chi_2)(d-1)} \geq n^{2(\beta_1+\beta_4)}.
\end{equation}

Considering now $r,s,\ep_{\min}$ and $\ol y$ from A4 (see Figure \ref{fig4-10A}), observe that by \eqref{Vgap}, $|\ol y - R\ytz| \leq C_{22}\ep_{\min}R|\ytz|$. Further, by \eqref{zangle}, the angle between $H_{\theta_0,0}$ and $H_{\theta,s}$ is at most $C_{10}\ep_{\min}$; it follows using \eqref{linegap2} that $s\yt=L_\theta\cap H_{\theta,s}$ satisfies $|s\yt| \leq C_{48}\ep_{\min}n^{\beta_1}\Delta(R)$, and similarly $|\ol y - r\yt| \leq C_{48}\ep_{\min}n^{\beta_1}\Delta(R)$.  Then using \eqref{zangle},
\begin{align}\label{rslower}
  |r-s||\yt| &\geq |\ol y| - 2C_{48}\ep_{\min}n^{\beta_1}\Delta(R) 
    \geq (1-C_{22}\ep_{\min})R|\ytz| - 2C_{48}\ep_{\min}n^{\beta_1}\Delta(R) \notag\\
  &\text{and hence}\ \  r-s \geq (1 - C_{49}\ep_{\min})R.
\end{align}
\end{remark}

Write $\Pi_{ab}^\infty$ for the infinite line through $a$ and $b$. \label{piline}

\begin{lemma}\label{badbehavG5}
There exist constants $C_i>0$ as follows.  Suppose A1, A2', A3 hold for some $\theta_0,\ep_0$, and A4 holds for some $R,n,\theta,\ell,\ep_{\min}$.  For the event
\begin{align*}
  G_1: &\text{ ({\it large transverse fluctuation in a geodesic}) there exists a $\theta_0$--slab geodesic $\Gamma$ from 
     $B_{\theta_0,{\rm home}}^{\rm rfat}$ } \\
   &\qquad \text{to $H_{\theta_0,2R}^+$ with $H_{\theta_0,R}^+$--entry point in $B_{\theta_0,{\rm cross}}^{\rm fat}$ and with } 
     \Gamma \not\subset Q_{R,n,\theta}, 
\end{align*}
we have
\begin{equation}\label{allbound}
  P(G_1) \leq \exp\left( -C_{50}\frac{n^{2\beta_1}}{\log R} \right).
\end{equation}
\end{lemma}

The next lemma is illustrated by Figure \ref{fig4-6}.

\begin{lemma}\label{badbehavG15}
There exist constants $C_i>0$ as follows.  Suppose A1, A2', A3 hold for some $\theta_0,\ep_0$, and A4 holds for some $R,n,\theta,\ell,\ep_{\min}$.  For the event
\begin{align*}
   G_2: &\text{ ({\it geodesic evading enlarged $\theta$--block}) There exists a geodesic $\Gamma$ from 
     $B_{\theta_0,{\rm home}}^{\rm rfat}$ }\\
  &\qquad \text{to $B_{\theta_0,{\rm cross}}^{\rm rfat}$ which contains a site in }  
    (H_{\theta,s}^{\rm rfat} \bs B_{s,\theta,{\rm home},+}^{\rm rfat})
    \cup (H_{\theta,r}^{\rm fat} \bs B_{r,\theta,{\rm home},+}^{\rm fat}),
    \end{align*}
we have
\begin{equation}\label{G15bound}
  P(G_2) \leq \exp\left( -C_{50}\frac{n^{2\beta_1}}{\log R} \right).
\end{equation}
\end{lemma}

For the event $G_2$, the idea is that $B_{s,\theta,{\rm home},+}^{\rm rfat}$ lies in front of, and nearly parallel to, the much smaller block $B_{\theta_0,{\rm home}}^{\rm rfat}$, so any geodesic from $B_{\theta_0,{\rm home}}^{\rm rfat}$ to $B_{\theta_0,{\rm cross}}$ will very likely cross $H_{\theta,s}^{\rm rfat}$ by passing through $B_{s,\theta,{\rm home},+}^{\rm rfat}$, whereas $G_2$ says to the contrary. The picture with the larger $B_{r,\theta,{\rm home},+}^{\rm fat}$ just behind the smaller $B_{\theta_0,{\rm cross}}^{\rm rfat}$ is analogous.

\begin{lemma}\label{badbehavG2G7}
There exist constants $C_i>0$ as follows.  Suppose A1, A2, A3 hold for some $\theta_0,\ep_0$, and A4 holds with $\theta=\theta_0$ for some $R,n,\ell,\ep_{\min}$, and let $\eta$ be as in \eqref{minfluct}.  For the events
\begin{align*}
  G_3:  &\text{ (backtrack) some $\theta_0$--target-directed $(\ell,\theta_0)$-interval geodesic contains a $\theta_0$--backtrack of } 
    \frac 12 n^{-\beta_4}\ell,\\
  G_4: &\text{ ({\it quick sidestep in a direction--$\theta_0$ geodesic}) There exists an $(\ell,\theta_0)$--interval geodesic 
    $\Gamma\subset Q_{R,n,\theta_0}$} \\
  &\qquad \text{and sites $u,v\in \Gamma$  with $|(u-v)_1^{\theta_0}| \leq n^{-\beta_4}\ell$ and } 
    h(u-v) - h(|(u-v)_1^{\theta_0}|\yt) \geq \frac \eta 8 \sigma(\ell),
\end{align*}
we have 
\begin{equation}\label{allbound2}
  P(G_3 \cup G_4) \leq \exp\left( -C_{50}\frac{n^{2\beta_1}}{\log R} \right).
\end{equation}
\end{lemma}

The last inequality in the definition of $G4$ is a way of quantifying that the orientation of the segment $\Gamma[u,v]$ is different from $\theta_0$.  

Lemma \ref{badbehavG2G7} is actually valid for arbitrary $\eta>0$, if we allow the constants $C_i$ in \eqref{Rncond} to depend on $\eta$, and then we may assume A2' in place of A2.  But such dependence in \eqref{Rncond} adds unnecessary complication, so we assume A2 and state the lemma only for the particular $\eta$ of interest, which exists under A2. The same consideration applies to the rest of the lemmas in this section.

The next lemma may be viewed as an extension of Proposition \ref{transTincr}. The idea, illustrated in Figure \ref{fig4-8}, is similar: if two geodesics have close--together endpoints at both ends, then the top and bottom geodesics there are effectively choosing a geodesic from the same set of paths, except near the endpoints, so their passage times will likely be very similar.  Here ``close--together'' means ``much less separation than the typical transverse wandering of the geodesics.''

\begin{lemma}\label{badbehavG9}
There exist constants $C_i>0$ as follows.  Suppose A1, A2, A3 hold for some $\theta_0,\ep_0$, and A4 holds with $\theta=\theta_0$ for some $R,n,\ell,\ep_{\min}$.  Let $\eta$ be as in \eqref{minfluct}. For the event
\begin{align*}
  G_5: &\text{ ({\it there are close $(\ell,\theta)$--interval geodesics with dissimilar passage times}) There exist $i\leq R/\ell$ } \\
  &\qquad \text{and $u,w\in H_{\theta,(i-1)\ell}^{\rm fat} \cap Q_{R,n,\theta}\cap\ZZ^d$ and } 
    \text{ $v,x \in H_{\theta,i\ell}^{\rm fat} \cap Q_{R,n,\theta}\cap\ZZ^d$ with } \\
  &\qquad |u-w| \leq 2\sqrt{d}n^{-\beta_0}\Delta(R), \quad |v-x| \leq 2\sqrt{d}n^{-\beta_0}\Delta(R), \quad |T(u,v) - T(w,x)| \geq \frac \eta 8 \sigma(\ell), \end{align*}
we have
\begin{equation}\label{allbound3}
  P(G_5) \leq \exp\left( -C_{50}\frac{n^{2\beta_1}}{\log R} \right).
\end{equation}
\end{lemma}

\begin{lemma}\label{badbehavG11}
There exist constants $C_i>0$ as follows.  Suppose A1, A2, A3 hold for some $\theta_0,\ep_0$, and A4 holds with $\theta=\theta_0$ for some $R,n,\ell,\ep_{\min}$.  Let $\eta$ be as in \eqref{minfluct}. For the event
\begin{align*}
  G_6: &\text{ ({\it unusual--speed short segment}) There exist $u,v\in Q_{R,n,\theta}\cap\ZZ^d$ with } \\
  &\qquad |u-v|\leq 3n^{-\beta_4}\ell, \quad |T(u,v) - h(v-u)| \geq \frac \eta 8 \sigma(\ell |\yt|/2), \\
\end{align*}
we have
\begin{equation}\label{allbound4}
  P(G_6) \leq \exp\left( -C_{50}\frac{n^{2\beta_1}}{\log R} \right).
\end{equation}
\end{lemma}

\begin{proof}
When $|u-v|\leq 3n^{-\beta_4}\ell$ we have using \eqref{powerlike} and the last inequality in \eqref{betas1} that
\[
  \sigma(|u-v|) \leq 3^{\chi_1}C_2^{-1}n^{-\beta_4\chi_1} \sigma(\ell) \leq \frac \eta 8 n^{-2\beta_1} \sigma(\ell |\yt|/2),
\]
so \eqref{expbd} says
\[
  P\left( |T(u,v) - h(|u-v|)| \geq \frac \eta 8 \sigma(\ell |\yt|/2) \right) \leq 4\exp\left( -n^{2\beta_1} \right).
\]
As in \eqref{pg5a}, summing over the $O(R^{2d})$ possible values of $(u,v)$ gives
\begin{equation}\label{PG11}
  P(G_6) \leq c_1\exp\left( -\frac 12 n^{2\beta_1} \right).
\end{equation}
\end{proof}

\section{Crowded geodesics} \label{crowd}
We are now ready for the core of our main proof, given by the next proposition. The core ideas were described in section \ref{outl}. Recall A4 containing \eqref{betas1}---\eqref{rsdef}.

\begin{proposition}\label{jammed1}
Suppose A1, A2, A3 hold for some $\theta_0,\ep_0$.  There exist constants $\beta_j\in (0,1)$ and $C_i$ as follows.  Let $R,n,\ep_{\min}$ be as in A4, and 
let $B_{\theta_0,{\rm cross}}$ be a $\theta_0$--block in $H_{\theta_0,R}$ with center point $\ol y$. \label{yb} Let $\theta=\ol y/|\ol y|$, and suppose $\psi_{\theta\theta_0}<\ep_{\min}/2$.
Then
\begin{align}\label{nrays}
  P\Big( &\text{there exist $n$ $\theta_0$--slab geodesics from 
    $B_{\theta_0,{\rm home}}^{\rm rfat}$ to $H_{\theta_0,2R}^+$ } \notag\\
  &\qquad \text{with distinct $H_{\theta_0,R}^+$--entry points in $B_{\theta_0,{\rm cross}}^{\rm fat}$} \Big) 
     \leq \exp\left( -C_{51}\frac{n^{2\beta_1}}{\log R} \right).
\end{align}
\end{proposition}

The upper bound on $n$ in \eqref{Rncond} can be written
\[
  \frac{n}{(n^{-\beta_0}\Delta(R))^{d-1}} \leq C_{46}n^{\beta_0(d-1)},
\]
which is not really a restriction at all, since the density of $H_{\theta_0,R}^+$--entry points is bounded.  It is only there for technical use in Lemmas \ref{badbehavG5}---\ref{badbehavG11}.

\begin{proof}[Proof of Proposition \ref{jammed1}]
Let $s,r$, and $\beta_j\in (0,1)$ be as in A4.  For reference we recall the following events:
\begin{itemize}
\item[(i)] $G_1$, the transverse--fluctuation event of Lemma \ref{badbehavG5},
\item[(ii)] $G_2$, the block--evasion event of Lemma \ref{badbehavG15},
\item[(iii)] $G_3$, the backtrack event of Lemma \ref{badbehavG2G7},
\item[(iv)] $G_4$, the sidestep event of Lemma \ref{badbehavG2G7},
\item[(v)] $G_5$, the event of close geodesics with dissimilar passage times, from Lemma \ref{badbehavG9},
\item[(vi)] $G_6$, the ``unusual--speed short segment'' event of Lemma \ref{badbehavG11}
\end{itemize}

As a shorthand, a $\theta_0$--slab geodesic from $B_{\theta_0,{\rm home}}^{\rm rfat}$ to $H_{\theta_0,2R}^+$ will be called a $2R$--\emph{geodesic}; $H_{\theta_0,R}$ and $H_{\theta_0,2R}$ are the \emph{midway} and \emph{ending hyperplanes}. Given a $2R$--geodesic $\Gamma$, we can decompose the pre-$H_{\theta_0,R}$ segment of $\Gamma$ into an initial bond and $n^{\beta_2}\ \ell$-segments; every such $\ell$--segment is an $(\ell,\theta_0)$--interval geodesic. An $(\ell,\theta)$-interval geodesic $\Psi$ is \emph{good} if 
\begin{itemize}
\item[(1)] $\Psi$ is contained in $Q_{R,n,\theta}$,
\item[(2)] $\Psi$ contains no backtrack of $\frac 12 n^{-\beta_4}\ell$.
\end{itemize}

Let $G_7$ be the event in \eqref{nrays}, and define the event
\begin{align*}
  G_8: &\text{ there exist $n$ $\theta$--target-directed $2R$--geodesics with distinct $H_{\theta_0,R}^+$--entry points 
     in $B_{\theta_0,{\rm cross}}$},
\end{align*}
so
\begin{equation}\label{G015}
  G_7 \subset G_8\cup G_1.
\end{equation}
Note that $\tau\in G_8\bs G_3$ says that every $\ell$-segment of every $\theta$--target-directed $2R$--geodesic is a good $(\ell,\theta)$-interval geodesic.

Our main task is to bound $P(G_8)$.  There are at most $c_1n^{(\beta_0+\beta_1)(d-1)}$ $\theta_0$--blocks intersecting $H_{\theta_0,2R}\cap Q_{R,n,\theta}$, and we denote the $j$th one (in some arbitrary order) as $B_{2R,j}$. \label{jblk}
When $\tau\in G_8$, there exists a subcollection $\mathfrak{G}$ of size at least
\[
  g_n = c_2 n^{1-(\beta_0+\beta_1)(d-1)}
\]
out of the $n\ 2R$--geodesics, which for some $m$ all have $H_{\theta_0,2R}^+$--entry point in block $B_{2R,m}^{\rm fat}$.  We call such a subcollection $\mathfrak{G}$ a \emph{crowded set} \label{crs} (via $B_{\rm cross}$ and $B_{2R,m}$), and fix such a $\mathfrak{G}$ and $m$.  Let $y^*$ \label{yst} be the center of $B_{2R,m}$, let $\theta^* = (y^*-\ol y)/|y^*-\ol y|$, and let 
\[
  Q_{R,n,\theta^*}^* = \ol y + \left( \RR \times [-4\sqrt{d-1}n^{\beta_1}\Delta(R),4\sqrt{d-1}n^{\beta_1}\Delta(R)]^{d-1} \right),
    \quad\text{in $\theta^*$--coordinates}.
\]
This is a tube around $L_{\theta^*}(\ol y) = \Pi_{\ol y,y^*}^\infty$  with cross section larger than that of $Q_{R,n,\theta}$ in each dimension by a factor $2\sqrt{d-1}$; see Figure \ref{fig4-10A}. It is straightforward to show that due to this larger cross section we have \label{qst}
\[
  Q_{R,n,\theta^*}^* \cap \Omega_{\theta_0}(R,2R) \supset Q_{R,n,\theta} \cap \Omega_{\theta_0}(R,2R),
\] 
that is, between the midway and ending hyperplanes, $Q_{R,n,\theta^*}^*$ contains $Q_{R,n,\theta}$.
Analogously to $[s,r]$, there is a largest interval $[s^*,r^*]$ for which
\[
  Q_{R,n,\theta} \cap \Omega_{\theta^*}(s^*,r^*) \subset Q_{R,n,\theta} \cap \Omega_{\theta_0}(R,2R);
\]
see again Figure \ref{fig4-10A}.
Since the $\theta_0$--block with center $y^*$ intersects $Q_{R,n,\theta}$ it is easily checked that the $\theta$--ratio of $y^*-\ol y$ is bounded by $c_3n^{\beta_1}\Delta(R)/R$, and hence 
\begin{equation}\label{turnangle}
  \psi_{\theta\theta^*} \leq c_4\frac{n^{\beta_1}\Delta(R)}{R} \leq \frac{\ep_{\min}}{4} 
    \quad\text{and hence}\quad \psi_{\theta^*\theta_0} \leq \frac{\ep_{\min}}{2}.
\end{equation} 

\begin{figure}
\includegraphics[width=15cm]{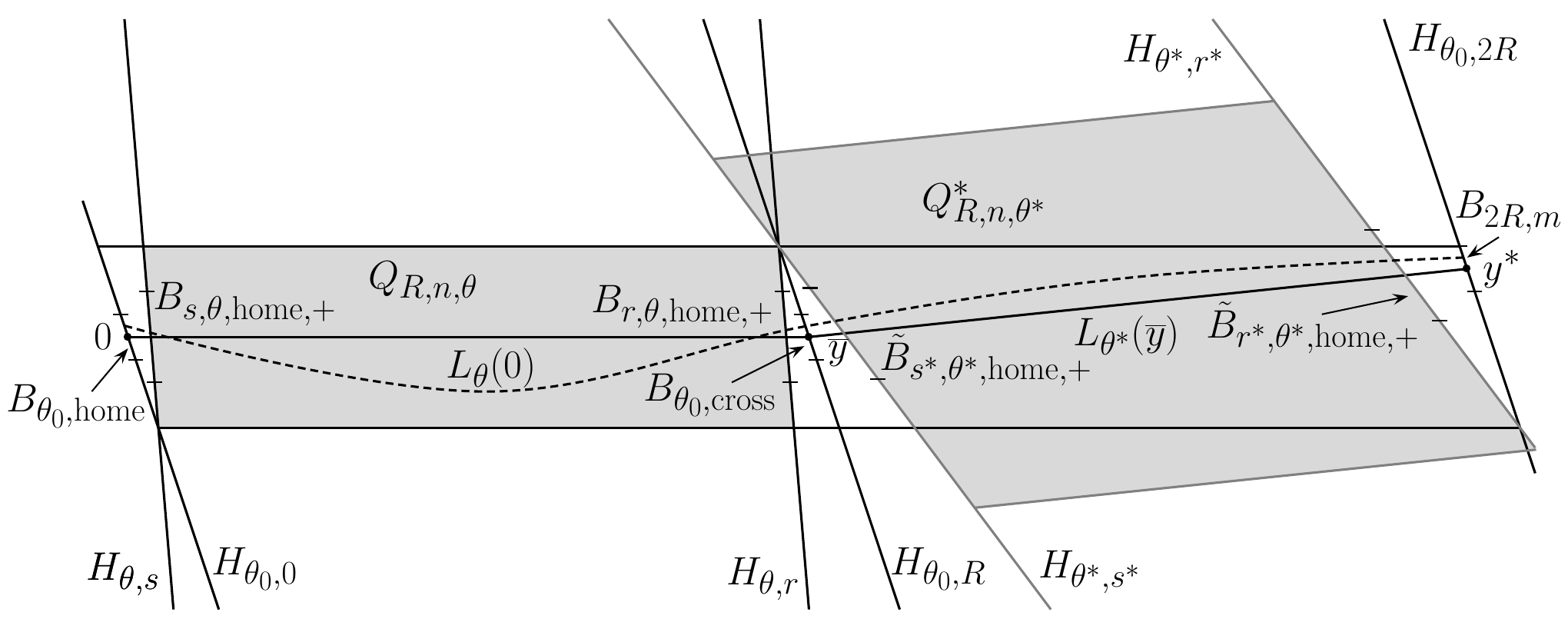}
\caption{ Illustration for the proof of Proposition \ref{jammed1}. The dashed line is a typical geodesic from the crowded set $\mkG$. The primary $\theta$--slab geodesic crosses the left gray box, which is part of the square tube $Q_{R,n,\theta}$ surrounding $L_\theta$.  The secondary $\theta^*$--slab geodesic crosses the right gray box, which is part of the square tube $Q_{R,n,\theta^*}^*$ surrounding $L_{\theta^*}(\ol y)$. The hash marks bound the named blocks (such as $B_{\theta_0,{\rm home}}$) in the hyperplanes as shown. }
\label{fig4-10A}
\end{figure}

There is a ``nuisance possibility'' we must deal with here: by assumption the $g_n$ (or more) $2R$--geodesics in the crowded set $\mathfrak{G}$ have distinct $H_{\theta_0,R}^+$--entry points, but this does not guarantee they have distinct entry points for the hyperplanes $H_{\theta,r}$ (behind $H_{\theta_0,R}$) and $H_{\theta^*,s^*}$ (ahead of $H_{\theta_0,R}$.)  See Figure \ref{fig4-10A}.  However, any set of geodesics in $\mkG$ which all have the same $H_{\theta,r}$--entry point must all have different $H_{\theta^*,s^*}$--entry points.
It follows that at least one of the following options must be true: \label{Gi}
\begin{itemize}
\item[(I)] there is a subset $\mkG_1\subset \mkG$ with $|\mkG_1|\geq g_n^{1/3}$ which all have the same $H_{\theta,r}^+$--entry point,
\item[(II)] there is a subset $\mkG_2\subset \mkG$ with $|\mkG_2|\geq g_n^{1/3}$ which all have the same $H_{\theta^*,s^*}^+$--entry point,
\item[(III)] there is a subset $\mkG_3\subset \mkG$ with $|\mkG_3|\geq g_n^{1/3}$ which all have distinct $H_{\theta,r}^+$--entry points, and which all have distinct $H_{\theta^*,s^*}^+$--entry points.
\end{itemize}
If (I) occurs in some $\tau$, then since the geodesics in $\mkG_1$ have distinct $H_{\theta_0,R}^+$--entry points, uniqueness of finite geodesics means these geodesics must be ``disjoint from $H_{\theta_0,R}$ to $H_{\theta_0,2R}^+$'', or more precisely, the geodesics $\Gamma[x_R'(\Gamma),x_{2R}''(\Gamma)], \Gamma \in \mkG_1,$ are disjoint, except for possibly sharing a starting point $x_R'(\Gamma)$. Similarly, if (II) occurs then the geodesics in $\mkG_2$ must all have disjoint pre--$H_{\theta_0,R}$ segments, except for possibly sharing the same $x_R''(\Gamma)$. In (III) we cannot conclude any such disjointness.

In $H_{\theta,s}$ and $H_{\theta,r}$ we have the enlarged home $\theta$--blocks $B_{s,\theta,{\rm home},+}$ and $B_{r,\theta,{\rm home},+}$, respectively, centered on the line $L_\theta$.  In $H_{\theta^*,s^*}$ and $H_{\theta^*,r*}$ we can define analogous shifted enlarged home $\theta^*$--blocks by translating an enlarged $\theta^*$--block within the hyperplane so that it is centered on $L_{\theta^*}(\ol y)$; we denote these shifted enlarged home $\theta^*$--blocks by $\tilde B_{s^*,\theta^*,{\rm home},+}$ and $\tilde B_{r^*,\theta^*,{\rm home},+}$, \label{tilb} respectively. The analog of the block--evasion event $G_2$ for the ``second half'' geodesic segments from $H_{\theta_0,R}$ to $H_{\theta_0,2R}$ is the following.
\begin{align*}
   G_9: &\text{ ({\it geodesic after $H_{\theta_0,R}$ evading enlarged $\theta$--blocks}) 
     There exists a geodesic $\Gamma$ from}\\
  &\qquad \text{$B_{\theta_0,{\rm cross}}^{\rm rfat}$ to $B_{2R,m}^{\rm fat}$ which contains a site in }  
    (H_{\theta^*,s^*}^{\rm rfat} \bs \tilde B_{s^*,\theta^*,{\rm home},+}^{\rm rfat})
    \cup (H_{\theta,r}^{\rm fat} \bs \tilde B_{r^*,\theta^*,{\rm home},+}^{\rm fat}).
\end{align*}   
In view of \eqref{turnangle} we essentially can apply Lemma \ref{badbehavG15} to bound $P(G_9)$, the only thing being different for $G_9$ is that the tube $Q_{R,n,\theta^*}^*$ is fatter by a constant factor $2\sqrt{d-1}$.  This makes no material difference so we have
\begin{equation}\label{PG16}
  P(G_9) \leq \exp\left( -C_{50}\frac{n^{2\beta_1}}{\log R} \right).
\end{equation}

Consider now $\tau \in G_8\bs (G_3\cup G_2\cup G_9)$ and suppose option (III) occurs. We have the following situation: each $\Gamma\in\mkG_3$ contains a unique $\theta$--slab geodesic from $B_{s,\theta,{\rm home},+}^{\rm rfat}$ to $B_{r,\theta,{\rm home},+}^{\rm fat}$, which we call the \emph{primary $\theta$--slab geodesic of $\Gamma$} and denote $\Gamma^{\rm pri}$. \label{gpr} $\Gamma$ also contains a unique $\theta^*$--slab geodesic from $\tilde B_{s^*,\theta^*,{\rm home},+}^{\rm rfat}$ to $\tilde B_{r^*,\theta^*,{\rm home},+}^{\rm fat}$, which we call the \emph{secondary $\theta^*$--slab geodesic of $\Gamma$} and denote $\Gamma^{\rm sec}$.  A site $x\in H_{\theta,r}^-$ is called \emph{popular for} $\mkG_3$ (in $\tau$) if there exist $n^{\beta_3-\beta_4}/12$ $2R$--geodesics $\Gamma\in\mkG_3$ for which $x$ lies in the primary $\theta$--slab geodesic of $\Gamma$.  A key observation is that if such a popular site $x$ exists, then since $\{\Gamma\in \mkG_3: x \in\Gamma\}$ have distinct $H_{\theta,r}^+$--entry points, the $n^{\beta_3-\beta_4}/12$ (or more) geodesics $\{\Gamma^{\rm sec}: \Gamma\in \mkG_3,x \in\Gamma\}$ must be disjoint.  Based on the preceding discussion we can now restate the options as follows, with option (III) split into two suboptions.
\begin{itemize}
\item[(I)] there is a subset $\mkG_1\subset \mkG$ with $|\mkG_1|\geq g_n^{1/3}$ for which $\{\Gamma^{\rm sec}: \Gamma\in \mkG_1\}$ are disjoint,
\item[(II)] there is a subset $\mkG_2\subset \mkG$ with $|\mkG_2|\geq g_n^{1/3}$ for which $\{\Gamma^{\rm pri}: \Gamma\in \mkG_1\}$ are disjoint,
\item[(IIIa)] there is a subset $\mkG_{3a}\subset \mkG$ with $|\mkG_{3a}|\geq n^{\beta_3-\beta_4}$ for which $\{\Gamma^{\rm sec}: \Gamma\in \mkG_{3a}\}$ are disjoint,
\item[(IIIb)] there is a subset $\mkG_{3b}\subset \mkG$ with $|\mkG_{3b}|\geq g_n^{1/3}$ for which no popular site exists for $\mkG_{3b}\}$.
\end{itemize}

In bounding the probabilities for options (I)--(IIIb) we only use the geodesics $\Gamma^{\rm pri}$ or $\Gamma^{\rm sec}$.  This means the original angle $\theta_0$ is no longer involved, except that we have effectively replaced $R$ with $r-s$ (for geodesics $\Gamma^{\rm pri}$) or $r^*-s^*$ (for geodesics $\Gamma^{\rm sec}$.) In view of \eqref{rslower} this has negligible effect. In the interest of expositional and notational clarity, we can therefore henceforth assume $\theta=\theta_0$ and $[r,s] = [0,R]$ for options (II) and (IIIb), where we deal with geodesics $\Gamma^{\rm pri}$.

The most difficult of the options to control is (IIIb) where we must deal with the lack of disjointness; in fact our proof of a probability bound for that case will essentially subsume the simpler proofs for the other 3 cases. Hence we consider the events
\begin{align*}
   G_{\rm bad} &= G_1\cup G_2\cup G_3\cup G_4\cup G_5\cup G_6\cup G_9\\
   G_{10}: &\ \tau \in G_8\bs G_{\rm bad} \text{ and option (IIIb) occurs},
\end{align*}
and we call $\mkG_{3b}$ a \emph{crowded subset}.

Recall that $S_{\theta,i}(\Gamma)$ denotes the $i$th $\ell$-segment of $\Gamma$.  We arbitrarily number the target $\theta$--blocks 1 through $(2\sqrt{d-1})^{d-1}n^{(\beta_0+\beta_1)(d-1)}$ in each $H_{i\ell}, i\leq n^{\beta_2}$.  For $\Gamma\in\mkG_{3b}$ we say $S_{\theta,i}(\Gamma)$ (or just $\Gamma$) makes a \emph{transition from $j$ to} $k$ if $x_{(i-1)\ell}''(\Gamma)$ is in fattened target $\theta$--block $j$ in $H_{(i-1)\ell}^{\rm fat}$ and $x_{i\ell}''(\Gamma)$ is in fattened target $\theta_0$--block $k$ in $H_{i\ell}^{\rm fat}$; we call this transition the $(i,j,k)$ \emph{transition} and write $S_{\theta,i}(\Gamma) \in \mT_i(j,k)$. \label{Tijk} For $\Gamma^{(1)},\Gamma^{(2)} \in \mkG_{3b}$ and $i\leq n^{\beta_2}$, $S_{\theta,i}(\Gamma^{(1)})$ and $S_{\theta,i}(\Gamma^{(2)})$ are called \emph{neighboring} if they make the same transition.    A given transition $(i,j,k)$ is called \emph{sparse} if the number of $2R$--geodesics in $\mkG_{3b}$ making that transition is at most $n^{\beta_3}$. 

The definitions of ``transition'' and ``neighboring,'' among others here, also make sense for general $(\ell,\theta)$-interval geodesics that are not part of a $2R$--geodesic.

When the $i$th $\ell$-segments of some $\Gamma,\hat\Gamma \in \mkG_{3b}$ intersect, $S_{\theta,i}(\Gamma) \cap S_{\theta,i}(\hat\Gamma)$ is necessarily a sub-segment of both $S_{\theta,i}(\Gamma)$ and $S_{\theta,i}(\hat\Gamma)$, with some endpoints $v=(v_1^{\theta},v_2^{\theta})_{\theta}$ and $w=(w_1^{\theta},w_2^{\theta})_{\theta}$, labeled so $v$ precedes $w$ in $\Gamma$.  We call $\Gamma[v,w] = \hat\Gamma[v,w]$ the \emph{overlap segment}.  The projection of the overlap segment onto the first $\theta$--coordinate is an interval in $\RR$ containing $v_1^\theta,w_1^\theta$ which we call the \emph{projected overlap interval}; we denote it $\mO(S_{\theta,i}(\Gamma),S_{\theta,i}(\hat\Gamma))$. \label{oss} The \emph{$\theta$--overlap} of $S_{\theta,i}(\Gamma)$ and $S_{\theta,i}(\hat\Gamma)$ is $|\mO(S_{\theta,i}(\Gamma),S_{\theta,i}(\hat\Gamma))|$.
If two neighboring $i$th $\ell$-segments have $\theta$--overlap at most $n^{-\beta_4}\ell$, we say they are \emph{low-overlap neighbors}; see Figure \ref{fig4-10B}.  We claim the following.

{\it Claim 1.} If $\tau\in G_{10}$ with crowded subset $\mkG_{3b}$, then every $\ell$-segment of every $\Gamma\in\mkG_{3b}$ making a non-sparse transition has a low-overlap neighbor.  In other words, the absence of a popular site forces the existence of low--overlap neighbors, except in sparse transitions.

\begin{figure}
\includegraphics[width=12cm]{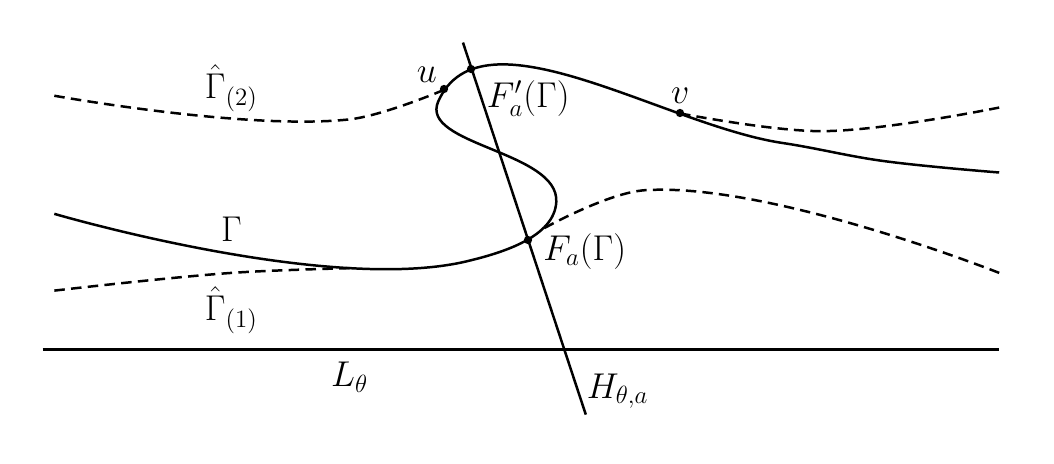}
\caption{ Diagram for the proof of Claim 1. Geodesic $\Gamma$ (solid curve) and neighbors $\hat\Gamma_{(1)},\hat\Gamma_{(2)}$. The projected overlap intervals $\mO(S_{\theta,i}(\Gamma),S_{\theta,i}(\hat\Gamma_{(1)}))$ and $\mO(S_{\theta,i}(\Gamma),S_{\theta,i}(\hat\Gamma_{(2)}))$ intersect, but the corresponding overlap segments do not. Each overlap segment whose projection contains $a$, for example $\Gamma[u,v]$, must contain $F_a(\Gamma)$ or $F_a'(\Gamma)$, because the overlap segment is longer, in the $\theta$ direction, than the maximum backtrack for $\tau\in G_3^c$. }
\label{fig5-1A}
\end{figure}

To prove Claim 1, fix $\tau\in G_{10}$. For fixed $i,j,k$ and $\Gamma\in\mathfrak{G}$ with $S_{\theta,i}(\Gamma)$ making a non-sparse $(i,j,k)$ transition, suppose $S_{\theta,i}(\Gamma)$ has no low-overlap neighbors. Then
\begin{equation}\label{oversum}
  \sum_{\hat\Gamma\in\mT_i(j,k),\hat\Gamma\neq\Gamma} \big| \mO(S_{\theta,i}(\hat\Gamma),S_{\theta,i}(\Gamma)) \big|
    \geq (|\mT_i(j,k)|-1)n^{-\beta_4}\ell \geq \frac 12 n^{\beta_3-\beta_4}\ell.
\end{equation}
We must deal with the technical complication that there may be small backtracks, meaning that 
\begin{itemize}
\item[(i)] not all projected overlap intervals, for $S_{\theta,i}(\Gamma)$ and some $S_{\theta,i}(\hat\Gamma)$, are necessarily contained in $[(i-1)\ell,i\ell]$, and \item[(ii)] having two projected overlap intervals (for different $\hat\Gamma$'s) intersect in an interval of positive length need not mean that the corresponding overlap segments have nonempty intersection.  
\end{itemize}
Issue (i) is readily dealt with: since we are assuming $\tau\in G_3^c$, so backtracks are small, every $\mO(S_{\theta,i}(\hat\Gamma),S_{\theta,i}(\Gamma))$ in \eqref{oversum} is contained in $[(i-2)\ell,i\ell]$. It then follows from \eqref{oversum} that some point $a\in [(i-2)\ell,i\ell]$ must be in at least $\frac 16 n^{\beta_3-\beta_4}$ of the intervals $\mO(S_{\theta,i}(\hat\Gamma),S_{\theta,i}(\Gamma))$.
For issue (ii), let $F_a(\Gamma)$ and $F_a'(\Gamma)$ be the first and last points, respectively, of $\Gamma$ in $H_{\theta,a}$.  See Figure \ref{fig5-1A}. Because we are assuming $\tau\in G_3^c$ and $\Gamma$ has no low-overlap neighbors, if $a$ lies in some $\mO(S_{\theta,i}(\hat\Gamma),S_{\theta,i}(\Gamma))$, then the corresponding overlap segment $S_{\theta,i}(\hat\Gamma) \cap S_{\theta,i}(\Gamma)$ must contain either $F_a(\Gamma)$ or $F_a'(\Gamma)$, because no segment of $\Gamma$ lying entirely between $F_a(\Gamma)$ and $F_a'(\Gamma)$ can have projected length more than $n^{-\beta_4}\ell/2$.  It follows that either $F_a(\Gamma)$ is contained in $S_{\theta,i}(\Gamma) \cap S_{\theta,i}(\hat\Gamma)$ for at least $\frac{1}{12} n^{\beta_3-\beta_4}$ of the neighbors $\hat\Gamma$ in \eqref{oversum}, or the same is true for $F_a'(\Gamma)$.  But this makes $F_a(\Gamma)$ or $F_a'(\Gamma)$ a popular site for $\mkG_{3b}$, a contradiction to $\tau\in G_{10}$.  This proves Claim 1.

{\it Claim 2.}  If $\tau\in G_{10}$ with crowded subset $\mkG_{3b}$, then there exists $\Gamma\in\mkG_{3b}$ such that every $\ell$--segment of $\Gamma$ has a low-overlap neighbor.  

To prove Claim 2, note that for each $i$, the number of possible transitions by the $i$th $\ell$--segment of a target-directed $2R$--geodesic is at most the square of the number of target $\theta$--blocks in each $H_{\theta,i\ell}$, so the number of $(i,j,k)$ such that $\ell$--segment $i$ can transition from block $j$ to block $k$ is at most $n^{\beta_2+2(\beta_0+\beta_1)(d-1)}$.  It follows that the total number of sparse transitions made by all $\Gamma\in\mkG_{3b}$ (summed over $\mkG_{3b}$) is at most $n^{\beta_3+\beta_2+2(\beta_0+\beta_1)(d-1)}$, which is less than $g_n^{1/3}$ by \eqref{betas3}. Hence there exists some $\Gamma\in\mkG_{3b}$ making no sparse transitions, and Claim 2 follows from Claim 1.

Our definition of low-overlap neighbor requires that such a neighbor be an $\ell$-segment of another $2R$--goedesic in our specified $\mkG_{3b}$.  We now loosen this restriction, and say for $\Gamma$ a $2R$--geodesic, any $(\ell,\theta)$--interval geodesic $\Psi$ is a \emph{low-overlap partner} of the $i$th $\ell$--segment of $\Gamma$ if the following hold:
\begin{itemize}
\item[(a)] $\Psi$ is a good $(\ell,\theta)$-interval geodesic, 
\item[(b)] $\Psi$ and $S_{\theta,i}(\Gamma)$ make the same transition, and
\item[(c)] the $\theta$--overlap of $\Psi$ and $S_{\theta,i}(\Gamma)$ is at most $n^{-\beta_4}\ell$.
\end{itemize}
It follows that every low-overlap neighbor is a low-overlap partner, for $\tau\in G_{10}$.
Define the event
\begin{align*}
  G_{11}: &\text{ there exist a target-directed $2R$--geodesic from $B_{\theta,{\rm home}}^{\rm rfat}$ to $H_{\theta,2R}^+$} \\
    &\qquad \text{with $H_{\theta,R}^+$--entry point in $B_{\theta,{\rm cross}}^{\rm fat}$ for which every $\ell$-segment in the pre-$H_{\theta,R}$ }\\ 
  &\qquad \text{segment has a low-overlap partner; } \tau\notin G_{\rm bad}. \\
\end{align*}
We then conclude from Claim 2 that
\begin{equation}\label{G1234}
  G_{10}\subset G_{11} \cup G_{\rm bad}.
\end{equation}

\begin{figure}
\includegraphics[width=14cm]{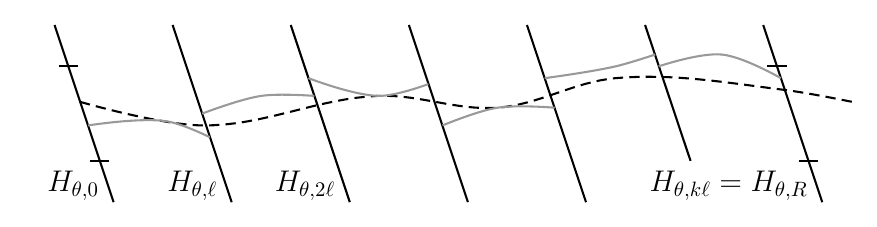}
\caption{ Geodesic $\Gamma$ (dashed curve) for which every $\ell$--segment of the pre--$H_{\theta,R}$ segment has a low-overlap neighbor (gray curves.) The low-overlap neighbors are $\ell$--segments of other geodesics in the crowded subset $\mkG_{3b}$ of the crowded group $\mkG$. Later in the proof we allow similar but more general ``low-overlap partners,'' which are still geodesics but which do not have to be $\ell$--segments of geodesics in the crowded subset. Of primary interest are the fast $\ell$--segments of $\Gamma$, for which the corresponding partners are (modulo a small-probability event) ``forced'' by Lemma \ref{badbehavG9} to be disjointly semi-fast. }
\label{fig4-10B}
\end{figure}

We now bound $P(G_{11})$. Define \label{nlnr}
\[
  N_L = |B_{\theta,{\rm home}}^{\rm rfat} \cap \ZZ^d|,\quad N_R = |B_{\theta,{\rm cross}}^{\rm fat} \cap \ZZ^d|,
    \quad \Omega_i = \Omega_{\theta}((i-1/6)\ell,(i+7/6)\ell).
\]
Let $\gamma_0$ be a $\theta$--slab path from $B_{\theta,{\rm home}}^{\rm rfat}$ to $B_{\theta,{\rm cross}}^{\rm fat}$.  Given $x,y\in \Omega_i\cap V$ we can define the set of low-overlap constrained paths 
\begin{align}\label{lowover}
  \mathfrak{P}_i&(x,y,\gamma_0) = \Big\{ \gamma: \gamma 
    \text{ is a path from $x$ to $y$ in $\ZZ^d$ with $\gamma\subset \Omega_i$, and either 
    $\gamma\cap\gamma_0=\emptyset$ or for some } \notag\\
  &\qquad\qquad\qquad \text{sites $u$ preceding $v$ in $\gamma_0$ with $|v_1^{\theta} - u_1^{\theta}| \leq n^{-\beta_4}\ell$ we have } 
    \gamma\cap\gamma_0 = \gamma_0[u,v] \Big\}
\end{align}
and for $\gamma \in \mathfrak{P}_i(x,y,\gamma_0)$ with $\gamma\cap\gamma_0 = \gamma_0[u,v]$, define \label{Td}
\[
  T^{\dis,i}(\gamma,\gamma_0) = T(\gamma) - T(\gamma\cap\gamma_0) + h((v_1^{\theta} - u_1^{\theta})\yt).
\]
Note that part of the passage time $T(\gamma)$ comes from the overlap segment $\gamma[u,v]$; in defining $T^{\dis,i}(\gamma,\gamma_0)$ we replace this part of the passage time with the approximation $h((v_1^{\theta} - u_1^{\theta})\yt)$, so that it does not depend on the passage times of bonds in the overlap segment.
Next define the disjoint passage times
\begin{align}\label{Tdis}
  T^{\dis,i}(x,y,\gamma_0) := \inf\{ T^{\dis,i}(\gamma,\gamma_0): \gamma \in \mathfrak{P}_i(x,y,\gamma_0) \},
\end{align}
so that $T^{\dis,i}(x,y,\gamma_0)$ is not affected by the passage times of the bonds in $\gamma_0$.  

The case of interest is the following: given $\tau\notin G_{\rm bad}$, if $\Gamma$ is a $\theta$--slab geodesic from $B_{\theta,{\rm home}}^{\rm rfat}$ to $B_{\theta,{\rm cross}}^{\rm fat}$, and $\Psi=\Psi[x,y]$ is a low-overlap partner of $S_{\theta,i}(\Gamma)$, then, due to the bound on backtracks in low-overlap partners we have $\Psi\subset\Omega_i$, so $\Psi\in \mathfrak{P}_i(x,y,\Gamma)$.  Suppose $\gamma \in \mathfrak{P}_i(x,y,\Gamma)$ with $\gamma\cap\Gamma = \Gamma[u,v]$ for some $u,v$. Since $v$ lies in the hyperplane $H_{\theta,v_1^\theta}$ tangent to $u +  |(v-u)_1^\theta|\mkB_g$ at $u+(v-u)_1^\theta\yt$, since $|v_1^{\theta} - u_1^{\theta}| \leq n^{-\beta_4}\ell$ we have using Lemma \ref{hmu}, \eqref{powerlike}, and \eqref{Rncond} (taking $C_{45}$ there sufficiently large) that
\begin{align*}
  h(v-u) &\geq g(v-u) \geq g((v-u)_1^\theta)\yt) \geq h((v-u)_1^\theta\yt) - C_{16}\sigma(|(v-u)_1^\theta\yt|) \log |(v-u)_1^\theta)\yt| \\
  &\geq h((v-u)_1^\theta\yt) - c_5n^{-\chi_1\beta_4}\sigma(\ell |\yt|/2) \log \ell \geq h((v-u)_1^\theta\yt) - \frac \eta 8 \sigma(\ell |\yt|/2).
\end{align*}
Since $\tau\notin G_6$, this yields that for all $\gamma \in \mathfrak{P}_i(x,y,\Gamma)$
\begin{equation}\label{discompare}
  T(\gamma) - T^{\dis,i}(\gamma,\Gamma) = T(u,v) - h(|u_1^{\theta} - v_1^{\theta}|\yt) 
    \geq T(u,v) - h(v-u) - \frac \eta 8 \sigma(\ell |\yt|/2) \geq  -\frac \eta 4 \sigma(\ell |\yt|/2).
\end{equation}
Since also $\tau\notin G_5$,
\begin{equation}\label{closetime}
  |T(\Psi) - T(S_{\theta,i}(\Gamma))| < \frac \eta 8 \sigma(\ell).
\end{equation}
For $\Gamma$ a $\theta$--slab geodesic from $B_{\theta,{\rm home}}^{\rm rfat}$ to $B_{\theta,{\rm cross}}^{\rm fat}$, we say $\Gamma$ is \emph{clean} if $\Gamma$ contains no $\theta$--backtrack of $\frac 12 n^{-\beta_4}\ell$, and no pair $u,v$ as in the event $G_4$; a $2R$--geodesic is called clean if its pre-$H_{\theta,R}$ segment is clean.
For $2\leq i\leq n^{\beta_2}-1$, as in Lemma \ref{fastseg} let us call $S_{\theta,i}(\Gamma)$ a \emph{fast segment} if 
\[
  T\big( S_{\theta,i}(\Gamma) \big) \leq \frac 1k ET(0,R\yt) + \frac \eta 8 \sigma(\ell |\yt|/2),
\]
so $N_{\theta}(\Gamma)$ (from Lemma \ref{fastseg}) is the number of fast $\ell$--segments in $\Gamma$.
We say an $(\ell,\theta)$--interval geodesic $\Psi=\Psi[x,y]$ 
is \emph{semi-fast} if its passage time satisfies
\[
  T(\Psi) < \frac 1k ET(0,R\yt) + \frac 34 \eta \sigma(\ell |\yt|/2),
\]
and \emph{disjointly semi-fast} if 
\[
  T^{\dis,i}(\Psi,\Gamma) < \frac 1k ET(0,R\yt) + \frac \eta 2 \sigma(\ell |\yt|/2).
\]
Thus for $\tau\in G_{11}$ we have the following:  there exists a clean $\theta$--target-directed geodesic $\Gamma$ from $B_{\theta,{\rm home}}^{\rm rfat}$ to $B_{\theta,{\rm cross}}^{\rm fat}$ for which every $S_{\theta,i}(\Gamma)$ has a low-overlap partner $\Psi_i = \Psi_i[y_i,z_i]$, and for each fast $\ell$--segment $S_{\theta,i}(\Gamma)$,
\[
   T^{\dis,i}(\Psi_i,\Gamma) \leq T(\Psi_i) + \frac \eta 4 \sigma(\ell |\yt|/2) \leq T(S_{\theta,i}(\Gamma)) +  \frac{3\eta}{8} \sigma(\ell |\yt|/2)
     \leq \frac 1k ET(0,R\yt) + \frac \eta 2 \sigma(\ell |\yt|/2),
\]
meaning $\Psi_i$ is disjointly semi-fast.  Here the first two inequalities use \eqref{discompare} and \eqref{closetime}. Thus for the event
\begin{align*}
  G_{12}: &\text{ there exists a clean $\theta$--target-directed $\theta$--slab geodesic $\Gamma$ from $B_{\theta,{\rm home}}^{\rm rfat}$ 
    to $B_{\theta,{\rm cross}}^{\rm fat}$} \\
  &\qquad \text{for which every fast $\ell$--segment $S_{\theta,i}(\Gamma)$ has a disjointly semi-fast low-overlap partner},
\end{align*}
we have
\begin{equation}\label{G3G12}
  G_{11} \subset G_{12},
\end{equation}
so we now want to bound $P(G_{12})$.  Let $a_1,\dots,a_{N_L}$ and $b_1,\dots,b_{N_R}$ \label{ajb} be the sites of $\ZZ^d$ in $B_{\theta,{\rm home}}^{\rm rfat}$ and $B_{\theta,{\rm cross}}^{\rm fat}$ respectively, and let
\[
  n_0=\eta n^{\chi_2\beta_2}/8.
\]
Defining events \label{Hjk}
\begin{align*}
  H_{jk}:\  &\Gamma_{a_jb_k} \text{ is a clean $\theta$--target-directed $\theta$--slab geodesic, $N_{\theta}(\Gamma_{a_jb_k}) \geq n_0$, 
    and every fast }  \notag\\
  &\quad \text{$\ell$--segment $S_{\theta,i}(\Gamma_{a_jb_k})$ with } 2\leq i\leq n^{\beta_2}-1 \text{ has a disjointly semi-fast low-overlap partner},
\end{align*}
we thus have
\begin{align}\label{closetimes}
  P(G_{12}) \leq \sum_{j=1}^{N_L} \sum_{k=1}^{N_R} 
    \Big( P(N_{\theta}(\Gamma_{a_jb_k}) < n_0) + P(H_{jk}) \Big).
\end{align}
(We note here that when $\tau\in H_{jk}$, the corresponding $\Gamma_{a_jb_k}$ can serve as the special geodesic in the heuristic in section \ref{outl}.)
Using the last  inequality in \eqref{betas1}, Lemma \ref{fastseg} yields that 
\begin{equation}\label{Nbound}
  P\left( N_{\theta}(\Gamma_{a_jb_k})<n_0 \right) \leq c_6R^{2(d-1)} \exp\left(-c_7 n^{(1-\chi_2)\beta_2} \right) 
    \leq c_8 \exp\left( -c_9 n^{2\beta_1} \right).
\end{equation}

We next bound $P(H_{jk})$. Suppose we fix both of the following: 
\begin{equation}\label{fixed}
  \text{a clean $\theta$--target-directed $\theta$--slab path $\gamma$ from $a_j$ to $b_k$,\ and the times } 
    \tau_\gamma = \{\tau_e: e\in\gamma\}. 
\end{equation}
These determine the set \label{Iset}
\[
  \mI(\gamma,\tau_\gamma) = \{ 2\leq i \leq n^{\beta_2}-1: S_{\theta,i}(\gamma) \text{ is a fast $\ell$--segment} \}.
\]
For each such $\gamma$ we have the event \label{Hev}
\[
  H_\gamma: \text{ every fast $\ell$--segment $S_{\theta,i}(\gamma)$, $2\leq i\leq n^{\beta_2}-1$, has a disjointly semi-fast low-overlap partner}.
\]
Conditionally on the fixed objects in \eqref{fixed}, we may view the events $H_\gamma$, as well as the event $\{\Gamma_{u_jv_k}=\gamma\}$, as determined by the unconditioned passage times $\{\tau_e:e\notin\gamma\}$.  In this context, $\{\Gamma_{u_jv_k}=\gamma\}$ is an increasing event (that is, its indicator is an increasing function of $\{\tau_e:e\notin\gamma\}$), whereas $H_\gamma$
is a decreasing event.  
It follows from the FKG inequality that
\begin{align}\label{mainrel}
  P(H_{jk} \mid \tau_\gamma, \Gamma_{u_jv_k}=\gamma) 
    &= P(H_\gamma \mid \tau_\gamma, \Gamma_{u_jv_k}=\gamma) \notag\\
  &\leq P(H_\gamma \mid  \tau_\gamma) \notag\\
  &= P(H_\gamma \mid \mI(\gamma,\tau_\gamma)).
\end{align}
For $I\subset \{1,\dots,n^{\beta_2}-2\}$, the events
\[
  H_{\gamma,I}: \text{ for all $i\in I$, $S_{\theta,i}(\gamma)$ has a disjointly semi-fast low-overlap partner}
\]
satisfy
\begin{equation}\label{HgI}
  P(H_\gamma \mid \mI(\gamma,\tau_\gamma)=I) 
    = P(H_{\gamma,I} \mid \mI(\gamma,\tau_\gamma)=I) = P(H_{\gamma,I}),
\end{equation}
since the event $H_{\gamma,I}$ is independent of $\{\tau_e,e\in\gamma\}$.  
For a clean $\theta$--target-directed $\theta$--slab path $\gamma$ from $B_{\theta,{\rm home}}^{\rm rfat}$ 
    to $B_{\theta,{\rm cross}}^{\rm fat}$, let \label{Mga}
\[
  \mM(\gamma) = \{(i,j,k): \gamma \text{ makes transition } (i,j,k)\}.
\]
and let $\mathfrak{M}$ be the set of all possible values of $\mM(\gamma)$ as $\gamma$ varies over all such paths. For $M\in\mathfrak{M}$ and $I\subset \{2,\dots,n^{\beta_2}-1\}$ define events \label{fmi}
\begin{align*}
  F_{M,I}: &\text{ for every $(i,j,k)\in M$ with $i\in I$, there exists a semi-fast good $(\ell,\theta)$--interval geodesic } \\
  &\qquad\text{making transition } (i,j,k).
\end{align*}

We claim that
\begin{equation}\label{}
  H_{\gamma,I} \bs (G_3\cup G_5\cup G_6) \subset F_{\mM(\gamma),I}\bs (G_3\cup G_5).
\end{equation}
In fact, suppose $\tau\in H_{\gamma,I} \bs (G_3\cup G_5\cup G_6),\, i\in I$, and $(i,j,k)\in \mM(\gamma)$. Then there are sites $u$ preceding $v$ in $S_{\theta,i}(\gamma)$ and a disjointly semi-fast $(\ell,\theta)$--interval geodesic $\Psi\subset Q_{R,n,\theta}$ for which
\[
  \Psi \text{ makes transition } (i,j,k), \quad \Psi\cap\gamma = \gamma[u,v], \quad |u_1^{\theta} - v_1^{\theta}| \leq n^{-\beta_4}\ell.
\]
Since $\tau\notin G_4\cup G_6$ and $\gamma$ is clean,
\begin{align}
  T(\Psi) &= T^{\dis,i}(\Psi,\gamma) + T(\Psi[u,v]) - h(|u_1^{\theta} - v_1^{\theta}|\yt) \notag\\
  &\leq T^{\dis,i}(\Psi,\gamma) + |T(\Psi[u,v]) - h(u-v)| + h(u-v) - h(|u_1^{\theta} - v_1^{\theta}|\yt) \notag\\
  &\leq \frac 1k ET(0,R\yt) + \frac 34 \eta \sigma(\ell |\yt|/2).
\end{align}
Thus $\Psi$ is semi-fast. Since $\tau\notin G_3, \Psi$ is also good, 
so $\tau \in F_{\mM(\gamma),I}$, proving the claim. This shows that
\begin{equation}\label{HgammaM}
  P(H_{\gamma,I}) \leq P( F_{\mM(\gamma),I}\bs (G_3\cup G_5)) + P(G_3\cup G_5\cup G_6).
\end{equation}

Let $\ol u_{ij}$ denote the center point of the $j$th block in $H_{i\ell}$, and define the event
\begin{align*}
  \ol F_{M,I}: &\text{ for every $(i,j,k)\in M$ with $i\in I$, we have } T(\ol u_{(i-1),j}, \ol u_{ik}) \leq \frac 1k ET(0,R\yt) + \frac 78 \eta \sigma(\ell |\yt|/2).
\end{align*}
Then $F_{\mM(\gamma),I}\bs G_5 \subset \ol F_{\mM(\gamma),I}$ so by \eqref{HgammaM},
\[
  P(H_{\gamma,I}) \leq P( \ol F_{\mM(\gamma),I}\bs (G_3\cup G_5)) + P(G_3\cup G_5\cup G_6).
\]
With \eqref{mainrel} and \eqref{HgI} we obtain from this that for the functions
\[
  f_0(\gamma,I) = P(H_{\gamma,I}), \quad f_1(M,I) = P( \ol F_{M,I}\bs (G_3\cup G_5)) + P(G_3\cup G_5\cup G_6),
\]
we have
\begin{align}\label{Hjkf1}
  P(H_{jk}) 
    &= E\bigg( P\Big(H_{jk}\ \Big|\  \{\tau_e,e\in\Gamma_{u_jv_k}\}, \Gamma_{u_jv_k}\Big) \bigg) \notag\\
    &\leq E\Big( f_0\Big(\Gamma_{u_jv_k},\mI(\Gamma_{u_jv_k},\{\tau_e:e\in \Gamma_{a_jb_k}\}) \Big) 
      1_{\{|\mI(\Gamma_{u_jv_k},\{\tau_e:e\in \Gamma_{a_jb_k}\})|\geq n_0-2\}} \Big) \notag\\
    &\leq E\Big( f_1\Big(\mM(\Gamma_{u_jv_k}),\mI(\Gamma_{u_jv_k},\{\tau_e:e\in \Gamma_{a_jb_k}\}) \Big) 
      1_{\{|\mI(\Gamma_{u_jv_k},\{\tau_e:e\in \Gamma_{a_jb_k}\})|\geq n_0-2\}} \Big).
\end{align}
As a comment on the last two expressions, we can view $\gamma,I$ as parameters in the probability $P(H_{\gamma,I})$ expressed by the function $f_0$, with this probability for a given $(\gamma,I)$ calculated for a random configuration $\tau$.  When we calculate the second expectation in \eqref{Hjkf1}, we can view it as choosing the parameters $\gamma,I$ randomly using a completely separate independent passage time configuration $\tau'$. Our ability to separate the choice of $\tau$ from the choice of parameters (functions of $\tau'$) is a consequence of the FKG property in \eqref{mainrel} and of \eqref{HgI}.  When we next replace $f_0$ with $f_1$ in the last line, the parametrization no longer uses the full path $\gamma$ but rather only the transitions of $\gamma$.  This formulation means that to bound $P(H_{jk})$, instead of summing over $M$ and $I$ in $f_1(M,I)$ (which involves too much entropy), we are averaging over $\Gamma_{u_jv_k}$.

We split $I$ into odd and even values, $I=I_{\rm odd}\cup I_{\rm even}$ and define for $u,v\in\Omega_i$:
\[
  T_{\Omega_i}(u,v) = \inf \{T(\gamma): \gamma \text{ is a path from $u$ to $v$}, \gamma\subset\Omega_i\}.
\]
Fix a possible value $M$ of $\mM(\gamma)$; for each $i\leq n^{\beta_2}$ there exist unique $j_i,k_i$ for which $(i,j_i,k_i)\in M$.
The events \label{Fst}
\[
  F_{M,i}^*: T_{\Omega_i}(\ol u_{(i-1),j_i}, \ol u_{ik_i}) \leq \frac 1k ET(0,R \yt) + \frac 78 \eta \sigma(\ell |\yt|/2) 
\]
with $i$ odd are independent, since the slabs $\Omega_i$ are disjoint.  Therefore
\begin{align}\label{prodbound}
  P( \ol F_{M,I}\bs G_3) &\leq P\left( \cap_{i\in I_{\rm odd}} F_{M,i}^* \right) \notag\\
  &= \prod_{i\in I_{\rm odd}} P(F_{M,i}^*) \notag\\
  &\leq \prod_{i\in I_{\rm odd}} P\left( T(\ol u_{(i-1),j_i}, \ol u_{i,k_i}) \leq \frac 1k ET(0,R\yt) + \frac 78 \eta \sigma(\ell |\yt|/2) \right).
\end{align}
With \eqref{powerlike} and Proposition \ref{hmu}, using $\ol u_{i,k_i} - \ol u_{(i-1),j_i} \in H_{\theta,\ell}$, we obtain
\begin{align}\label{uuvsRk}
  \frac 1k ET(0,R\yt) 
  &\leq g(\ell\yt) + \frac{C_{16}}{k}\sigma(R |\yt|) \log(R |\yt|) \notag\\
  &< g(\ol u_{i,k_i} - \ol u_{(i-1),j_i}) + \frac{1}{8} \eta \sigma(\ell |\yt|/2) \notag\\
  &\leq h(\ol u_{i,k_i} - \ol u_{(i-1),j_i}) + \frac{1}{8} \eta \sigma(\ell |\yt|/2).
\end{align}
Since $\ol u_{i,k_i},\ol u_{(i-1),j_i}$ lie in $Q_{R,n,\theta}$, it follows readily from \eqref{nRell} that 
\[
  |\ol u_{i,k_i} - \ol u_{(i-1),j_i}| \geq |(\ol u_{i,k_i} - \ol u_{(i-1),j_i})_1^{\theta}| |\yt| - |(\ol u_{i,k_i} - \ol u_{(i-1),j_i})_2^{\theta}|
    \geq \frac{\ell |\yt|}{2}.
\]

Combining this with \eqref{minfluct} and \eqref{uuvsRk} yields
\begin{align}\label{FMIbound}
  P&\left( T(\ol u_{(i-1),j_i}, \ol u_{i,k_i}) \leq \frac 1k ET(0,R\yt) + \frac 78 \eta \sigma(\ell |\yt|/2) \right) \notag\\
  &\qquad \leq P\Big( T(\ol u_{(i-1),j_i}, \ol u_{i,k_i}) < h(\ol u_{i,k_i} - \ol u_{(i-1),j_i}) 
    + \eta \sigma(|\ol u_{i,k_i} - \ol u_{(i-1),j_i}|) \Big) \notag\\
  &\qquad \leq 1-\frac{c_{10}}{\pi(\ell)^2\log\ell}.
\end{align}
We can take $\chi_1,\chi_2$ satisfying $2\chi_1-\chi_2 \geq \chi_2/2$, and then using \eqref{powerlike}, \eqref{Rncond}, and the last inequality in \eqref{betas1} we have 
\[
  \pi(\ell)\leq c_{11}n^{(\chi_2-\chi_1)\beta_2}\pi(R) \quad\text{so}\quad \pi(\ell)^2\log\ell \leq c_{12}n^{2(\chi_2-\chi_1)\beta_2}\pi(R)^2\log  R
    \leq c_{13}n^{2(\chi_2-\chi_1)\beta_2+\beta_1/2}.
\]
Since \eqref{prodbound} also holds for $I_{\rm even}$, and $\max(I_{\rm odd},I_{\rm even}) \geq |I|/2$, this together with the last inequality in \eqref{betas1} shows that for $|I| \geq n_0-2$,
\begin{align}
  P( \ol F_{M,I}\bs G_3) &\leq \left(1-\frac{1}{\pi(\ell)^2\log\pi(\ell)} \right)^{|I|/2} 
    \leq \left(1-\frac{1}{\pi(\ell)^2\log\pi(\ell)} \right)^{n_0/4} \notag \\
  &\leq \exp\left( -c_{14}n^{\beta_2\chi_2-2(\chi_2-\chi_1)\beta_2-\beta_1/2} \right) \leq \exp\left( -c_{15}n^{\beta_2\chi_2/2 - \beta_1/2} \right)
    \leq \exp\left(  -c_{16}n^{2\beta_1} \right).
\end{align}
Combining this with Lemmas \ref{badbehavG2G7}, \ref{badbehavG9}, and \ref{badbehavG11} we see that for $|I|\geq n_0-2$ and all $M$,
\[
  f_1(M,I) \leq \exp\left( -c_{16}n^{2\beta_1} \right) + \exp\left( -C_{50}\frac{n^{2\beta_1}}{\log R} \right).
\]
With \eqref{G1234}, \eqref{G3G12}, \eqref{closetimes}, \eqref{Nbound}, and \eqref{Hjkf1} this shows
\begin{align}\label{combine}
  P(\tau \in &G_8\bs G_{\rm bad} \text{ and option (IIIb) occurs}) \notag\\
  &=P(G_{10}) \notag\\
  &\leq P(G_{12}) + P(G_{\rm bad}) \notag\\ 
  &\leq N_L N_R \left( c_8 \exp\left( -c_9 n^{2\beta_1} \right) 
    + \exp\left( -c_{16}n^{2\beta_1} \right) + \exp\left( -C_{50}\frac{n^{2\beta_1}}{\log R} \right) \right) + P(G_{\rm bad}) \notag\\
  &\leq c_{17}R^{2(d-1)}\exp\left( -C_{50}\frac{n^{2\beta_1}}{\log R} \right) + P(G_{\rm bad}) \notag\\
  &\leq \exp\left( -c_{18}\frac{n^{2\beta_1}}{\log R} \right) + P(G_{\rm bad}),
\end{align}
where we have used $n^{\beta_1} \geq (\log R)^2$, from \eqref{Rncond}.

The preceding proof of \eqref{combine} applies also to options (I) and (II); the proof could even be substantially simplified for these cases since there is no overlap in the relevant geodesics.  To deal with option (IIIa) where the number of geodesics in the crowded subset is smaller, we can simply define $m$ by $n^{\beta_3-\beta_4} = g_m^{1/3}$, so $m=c_{19}n^\lambda$ with $\lambda<1$, and then repeat the entire proof with $n$ replaced by $m$.  (This has the effect of replacing each $\beta_j$ with $\lambda\beta_j$.) We thus conclude using \eqref{G015} and Lemmas \ref{badbehavG5}---\ref{badbehavG11} that
\[
  P(G_7) \leq P(G_8\bs G_{\rm bad}) + P(G_{\rm bad}) \leq 4\exp\left( -c_{18}\frac{n^{2\beta_1}}{\log R} \right) + 5P(G_{\rm bad})
    \leq \exp\left( -c_{20}\frac{n^{2\beta_1}}{\log R} \right),
\]
which completes the proof of Proposition \ref{jammed1}.
\end{proof}

\section{Crossing density bound}\label{cpdens}

In this section we prove Theorem \ref{alldim}(ii) and Theorem \ref{hitpoint}.

\subsection{Proof of Theorem \ref{alldim}(ii).} \label{alldimpf}
The \emph{orientation} of a finite geodesic from $x$ to $y$ is $(y-x)/|y-x|$.  With $\theta_0,R$ fixed, for $\Gamma$ a $\theta_0$--slab geodesic from $H_{\theta_0,-R}^-$ to $H_{\theta_0,R}^+$, the \emph{initial orientation} of $\Gamma$ is the orientation of its pre-$H_{\theta_0,0}$ segment.  The term ``initial orientation'' also applies to $\theta$--rays from $H_{\theta_0,-R}^-$ which cross $H_{\theta_0,0}$. 
Let $\ep_1$ be as in Remark \ref{spheres} and $\ep_{\min}$ as in A4, let $\ep(R) = (\log R)^2\Delta(R)/R$, and for $\theta\in J(\theta_0,\ep_{\min})$ define
\begin{align*}
  W_R := W_R(\theta_0,\theta,\tau) &:= \Big\{ x \in H_{\theta_0,0}^{\rm fat}\cap\ZZ^d: \text{ there exists a $\theta_0$--slab geodesic from 
    $H_{\theta_0,-R}^-$ to $H_{\theta_0,R}^+$} \\
  &\qquad\qquad \text{ with $H_{\theta_0,0}^+$--entry point $x$ and initial orientation in } J(\theta,\ep(R)) \Big\},\\
  Y_R := Y_R(\theta_0,\theta,\tau) &:= \Big\{ x \in H_{\theta_0,0}^{\rm fat}\cap\ZZ^d: 
    \text{ for some $\theta\in J(\theta_0,\ep_{\min})$ there exists a $\theta$--ray from $H_{\theta_0,-R}^-$ } \\
  &\qquad\qquad \text{ with $H_{\theta_0,0}^+$--entry point $x$ and initial orientation not in } J(\theta,\ep(R)) \Big\},\\
  Z_R := Z_R(\theta_0,\theta,\tau) &:= \Big\{ x \in H_{\theta_0,0}^{\rm fat}\cap\ZZ^d: \text{ there exists a 
    $\theta$--ray from $H_{\theta_0,-R}^-$ with} \\
  &\qquad\qquad \text{ $H_{\theta_0,0}^+$--entry point $x$} \Big\},
\end{align*}
so
\begin{equation}\label{union}
  Z_R \subset W_R \cup Y_R.
\end{equation}
We may think of $W_R$ and $Y_R$ as the sets of entry points of ``normal'' and ``crooked'' $\theta$--rays, respectively, though formally $W_R$ is defined in terms of finite geodesics.

Let us first bound the density of $Y_R$. Suppose $x\in Y_R,\theta\in J(\theta_0,\ep_{\min})$, and there exists a $\theta$--ray from $H_{\theta_0,-R}^-$ with starting point $a\in H_{\theta_0,-R}^{\rm rfat}$ and $H_{\theta_0,0}^+$--entry point $x$ for which the initial orientation $\alpha=(x-a)/|x-a| \notin J(\theta,\ep(R))$.  This means $\psi_{\alpha\theta} \geq \ep(R)$, and it is easily checked that this implies 
\[
  D_\theta(x-a) \geq c_1\left( \frac{\ep(R)^2|x-a|^2}{\Xi(|x-a|)^2} \wedge \Phi(|x-a|) \right)
    \geq c_2\ep(R)^2 \frac{|x-a|}{\sigma(|x-a|)\log |x-a|}.
\]
But then, defining events
\begin{align*}
  F_{x,j}: &\text{ for some $a\in \ZZ^d$ with $2^{j-1}<|x-a|\leq 2^j$ and some $\theta\in J(\theta_0,\ep_{\min})$ with } \\
  &\qquad\text{ $D_\theta(x-a)\geq c_2\ep(R)^2 \frac{|x-a|}{\sigma(|x-a|)\log |x-a|}$, there is a $\theta$--ray from $a$ containing $x$,  }
\end{align*}
we get from Proposition \ref{rayfluct}(ii) that 
\begin{equation}\label{nobend}
  P(x\in Y_R) \leq \sum_{j\geq \log_2 R} P(F_{x,j}) \leq 
    \sum_{j\geq \log_2 R} c_32^{dj} e^{-c_4\ep(R)^2 2^j/j\sigma(2^j)} 
    \leq c_5e^{-c_6(\log R)^2}.
\end{equation}

\begin{figure}
\includegraphics[width=14cm]{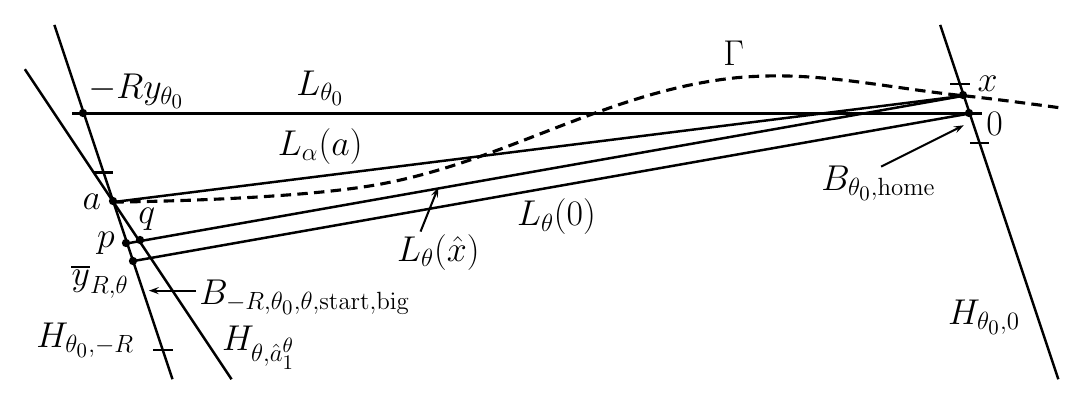}
\caption{ Illustration for the proof of Theorem \ref{alldim}(ii), showing the event $x\in W_R$.  The corresponding $\theta_0$--slab geodesic $\Gamma$ starts at some $a\in B_{-R,\theta_0,\theta,{\rm start,big}}^{\rm rfat}$ (the block centered at $\ol y_{R,\theta}$, delimited by hash marks in the figure), and enters $H_{\theta_0,0}^+$ at $x$.  The event $x\in Y_R$ is similar but $\Gamma$ is a $\theta$-ray and $a$ is far from $\ol y_{R,\theta}$, meaning the initial orientation of $\Gamma$ is not close to $\theta$. }
\label{fig1-5ii}
\end{figure}

We now turn to our main task, which is bounding the density of $W_R$; see Figure \ref{fig1-5ii}. We keep $\theta_0,\theta$ fixed throughout, with $\psi_{\theta\theta_0}<\ep_{\min}$. Let $R>0$ and $n=(\pi(R)\log R)^{K_1}$, with $K_1\geq C_{45}$ to be specified (see \eqref{Rncond}.) Define 
\[
  \ol y_{R,\theta} = L_\theta \cap H_{\theta_0,-R},
\]
the point in $H_{\theta_0,-R}$ from which a $\theta$--ray ``aims at 0.'' Let $B_{\theta_0,{\rm home}}\subset H_{\theta_0,0}$ be the home $\theta_0$--block as in Section \ref{densely}, and define the ``big block'' centered at $\ol y_{R,\theta}$:
\[
    B_{-R,\theta_0,\theta,{\rm start,big}} = \ol y_{R,\theta} + \left( \{-R\} \times \big[-8\sqrt{d}|\ytz|\ep(R)R,8\sqrt{d}|\ytz|\ep(R)R\big]^{d-1} \right)
\]
expressed in $\theta_0$--coordinates. As after Remark \ref{betapowers}, we can define fattened versions $B_{\theta_0,{\rm home}}^{\rm fat}$ and $B_{-R,\theta_0,\theta,{\rm start,big}}^{\rm rfat}$.  We want to show that the $\theta_0$--slab geodesic in the definition of $W_R$ must have starting point in $B_{-R,\theta_0,\theta,{\rm start,big}}^{\rm rfat}$.  To that end we take $a\in H_{\theta_0,-R}^{\rm rfat}\bs B_{-R,\theta_0,\theta,{\rm start,big}}^{\rm rfat}$ and $x\in  B_{\theta_0,{\rm home}}^{\rm fat}$ and let $\alpha = (x-a)/|x-a|$. We will show that $\psi_{\alpha\theta}>\ep(R)$, so that indeed $a$ is not the starting point in question.

Let $\hat a,\hat x$ be the $\theta_0$--projections of $a,x$ into $H_{\theta_0,-R}$ and $H_{\theta_0,0}$, respectively, and let $\hat \alpha = (\hat x - \hat a)/|\hat x-\hat a|$, so $\psi_{\alpha\hat\alpha}\leq c_7/R$. Define intersection points
\[
  p = L_\theta(\hat x)\cap H_{\theta_0,-R}, \quad q = L_\theta(\hat x)\cap H_{\theta,\hat a_1^\theta} \quad\text{so}\quad 
    p-\ol y_{R,\theta} = \hat x,
\]
and note that both the hyperplanes here contain $\hat a$.  By \eqref{linegap2} we have $|p-q|\leq C_{22}\psi_{\theta\theta_0}|\hat a-p|\leq C_{31}\psi_{\theta\theta_0}(|\hat a-\ol y_{R,\theta}| + |y_{R,\theta} - p|)$. Hence the $\theta$--ratio of $\hat x-\hat a$ satisfies
\begin{align}\label{tratio}
  \frac{|(\hat x-\hat a)_2^\theta|}{|\yt| |(\hat x-\hat a)_1^\theta|} &= \frac{|q-\hat a|}{|\hat x-q|} 
    \geq \frac{|\hat a-\ol y_{R,\theta}| - |\ol y_{R,\theta}-p| - |p-q|}{|\hat x-p|+|p-q|} \notag\\
  &\geq \frac{(1-C_{22}\psi_{\theta\theta_0})|\hat a-\ol y_{R,\theta}| - (1+C_{22}\psi_{\theta\theta_0})|\hat x|}{|y_{R,\theta}| +
    C_{22}\psi_{\theta\theta_0}(|\hat a-\ol y_{R,\theta}| + |\hat x|)}.
\end{align}
Since $a\notin B_{-R,\theta_0,\theta,{\rm start,big}}^{\rm rfat}$ we have
\[
  |\hat a-\ol y_{R,\theta}| \geq 8\sqrt{d}|\ytz|\ep(R)R,\quad |\hat x| \leq \sqrt{d-1}n^{-\beta_0}\Delta(R)
\]
and from \eqref{Vgap}
\[
  |\ol y_{R,\theta}| \leq |-R\ytz| + |\ol y_{R,\theta} - (-R\ytz)| \leq R|\ytz|(1 + C_{22}\psi_{\theta\theta_0}).
\]
With \eqref{tratio} these yield
\[
  \frac{|(\hat x-\hat a)_2^\theta|}{|\yt| |(\hat x-\hat a)_1^\theta|} \geq 7\sqrt{d}\ep(R).
\]
so from \eqref{tanvsratio2} (after reducing $\ep_{\min}$ if necessary),
\[
 \tan \psi_{\hat\alpha\theta} \geq \frac 72 \ep(R) \quad\text{and hence}\quad 
   \psi_{\alpha\theta}>\psi_{\hat\alpha\theta} - \psi_{\alpha\hat\alpha} > \ep(R),
\]
as desired, showing the starting point corresponding to $x\in  B_{\theta_0,{\rm home}}^{\rm fat}$ cannot be $a\notin B_{-R,\theta_0,\theta,{\rm start,big}}^{\rm rfat}$.

In the other direction, if $y^*$ is the center of one of the blocks $B_{-R,m}$ intersecting $B_{-R,\theta_0,\theta,{\rm start,big}}$, then
from \eqref{tanvsratio}, the $\theta_0$--ratio of $y^*$ satisfies
\begin{equation}\label{OKangle}
  \frac{|(y^*)_2^{\theta_0}|}{(y^*)_1^{\theta_0}} \leq \frac{8d|\ytz|\ep(R)R + \sqrt{d-1}n^{-\beta_0}\Delta(R)}{R}
    \leq 9d|\ytz|\ep(R) \quad\text{so}\quad \psi_{\theta_0y^*} \leq 36d\ep(R) < \frac{\ep_{\min}}{2},
\end{equation}
which will enable us later to apply Proposition \ref{jammed1}.

Now the number (which we call $M_R$) of $\theta_0$--blocks in $H_{\theta_0,-R}$ which intersect $B_{-R,\theta_0,\theta,{\rm start,big}}$ satisfies
\begin{equation}\label{MR}
  M_R \leq \left( \frac{ 9\sqrt{d}|\ytz|\ep(R)R }{ n^{-\beta_0}\Delta(R) } \right)^{d-1} \leq c_8n^{\beta_0(d-1)}(\log R)^{2(d-1)}.
\end{equation}
We denote these $\theta_0$--blocks as $\{B_{-R,m},m\leq M_R\}$, and define events
\begin{align*}
  A_m: &\text{ there exist $n$ $\theta_0$--slab geodesics from $B_{-R,m}^{\rm rfat}$ to $H_{\theta_0,R}^+$ with}\\
  &\qquad \text{distinct $H_{\theta_0,0}^+$--entry points, all in $B_{\theta_0,{\rm home}}^{\rm fat}$}.
\end{align*}
Then from Proposition \ref{jammed1} and \eqref{MR}, assuming $K_1$ is large, using \eqref{Rncond},
\begin{align}\label{toomany}
  P\Big( |W_R(\theta_0,\theta,\ep_{\min}) &\cap B_{\theta_0,{\rm home}}^{\rm fat}| \geq nM_R\Big) \leq \sum_{m=1}^{M_R} P(A_m) \notag\\
  &\leq M_R \exp\left( -C_{51}\frac{n^{2\beta_1}}{\log R} \right) \leq \exp\left( -\frac{C_{51}}{2}\frac{n^{2\beta_1}}{\log R} \right)
\end{align}
so in view of \eqref{union} and \eqref{nobend},
\begin{align}\label{toomanyE}
  E\Big( |Z_R&(\theta_0,\theta,\tau) \cap B_{\theta_0,{\rm home}}^{\rm fat}| \Big) \notag\\
    &\leq nM_R + |\ZZ^d \cap B_{\theta_0,{\rm home}}^{\rm fat}| 
      \left[ \exp\left( -\frac{C_{51}}{2}\frac{n^{2\beta_1}}{\log R} \right) 
      + \sup_{x \in H_{\theta_0,0}^{\rm fat}\cap\ZZ^d} P(x\in Y_R(\theta_0,\theta,\ep_1)) \right]\notag\\
  &\leq 2nM_R \leq 2c_8(\pi(R)\log R)^{(1+\beta_0(d-1))K_1+2(d-1)}.
\end{align}
Letting $\hat B_{\theta_0,0,t} = \{0\}\times [-t,t]^{d-1}$ (in $\theta_0$--coordinates), and noting that \eqref{toomanyE} is also valid for any translate of $B_{\theta_0,{\rm home}}^{\rm fat}$ in $H_{\theta_0,0}^{\rm fat}$, we can sum \eqref{toomanyE} over such translates to obtain
\begin{equation}\label{supE}
  \limsup_{t\to\infty} \frac{ E\big( |Z_R(\theta_0,\theta,\tau) \cap \hat B_{\theta_0,0,t}^{\rm fat}| \big) }{(2t)^{d-1}} \leq 
    \frac{ E( |Z_R(\theta_0,\theta,\ep_{\min}) \cap B_{\theta_0,{\rm home}}^{\rm fat}| ) }{ (2n^{-\beta_0}\Delta(R))^{d-1} } \leq
    \frac{ c_8(\pi(R)\log R)^{K_2} }{2^{d-2}\Delta(R)^{d-1}},
\end{equation}
where $K_2=(1+2\beta_0(d-1))K_1+2(d-1)$, proving \eqref{crossdens}.

To prove \eqref{combdens}, we proceed similarly but in the definition of $B_{-R,\theta_0,\theta,{\rm start,big}}$ we replace $\ep(R)$ with $\ep_{\min}/72d$, and in place of $W_R,Z_R$ we use
\begin{align*}
  \hat W_R := \hat W_R(\theta_0,\theta,\tau) &:= \Big\{ x \in H_{\theta_0,0}^{\rm fat}\cap\ZZ^d: 
    \text{ for some $\theta\in J(\theta_0,\ep_{\min}/72d)$ there exists a $\theta_0$--slab} \\
  &\qquad\qquad \text{ geodesic from $H_{\theta_0,-R}^-$ to $H_{\theta_0,R}^+$ with $H_{\theta_0,0}^+$--entry point $x$ and } \\
  &\qquad\qquad \text{ initial orientation in } J(\theta,\ep_{\min}/72d) \Big\},\\
  \hat Z_R := \hat Z_R(\theta_0,\theta,\tau) &:= \Big\{ x \in H_{\theta_0,0}^{\rm fat}\cap\ZZ^d: 
    \text{  for some $\theta\in J(\theta_0,\ep_{\min}/72d)$ there exists a 
    $\theta$--ray from} \\
  &\qquad\qquad \text{ $H_{\theta_0,-R}^-$ with $H_{\theta_0,0}^+$--entry point $x$} \Big\},
\end{align*}
which satisfy $\hat Z_R\subset \hat W_R\cup Y_R$.  Then we must also replace $\ep(R)$ with $\ep_{\min}/72d$ in \eqref{MR}.
The modified \eqref{OKangle} still gives $\psi_{\theta_0y^*} \leq \ep_{\min}/2$ so Proposition \ref{jammed1} still applies.
Hence in place of \eqref{toomanyE} we have
\begin{align}\label{MR2}
  E\Big( |\hat Z_R(\theta_0,\theta,\tau) \cap &B_{\theta_0,{\rm home}}^{\rm fat}| \Big) 
    \leq 2nM_R \leq c_9n^{1+\beta_0(d-1)}\left( \frac{R}{\Delta(R)} \right)^{d-1} \notag\\
  &\leq c_{10}(\pi(R)\log R)^{(1+\beta_0(d-1))K_1} \left( \frac{R}{\Delta(R)}\right)^{d-1},
\end{align}
and then in place of \eqref{supE},
\begin{equation}\label{supE2}
  \limsup_{t\to\infty} \frac{ E\big( |\hat Z_R(\theta_0,\theta,\tau) \cap \hat B_{\theta_0,0,t}^{\rm fat}| \big) }{(2t)^{d-1}} \leq 
    \frac{ E( |\hat Z_R(\theta_0,\theta,\tau) \cap B_{\theta_0,{\rm home}}^{\rm fat}| ) }{ (2n^{-\beta_0}\Delta(R))^{d-1} } \leq
    \frac{ c_8(\log R)^{(1+2\beta_0(d-1))K_1} }{2^{d-2}\sigma(R)^{d-1}},
\end{equation}
proving \eqref{combdens} and completing the proof of Theorem \ref{alldim}(ii).

\subsection{Proof of Theorem \ref{hitpoint}.} \label{hitpointpf}
Suppose A1, A2, A3 hold for some $\theta_0,\ep_0$. By changing $\theta_0$ slightly we may assume $\theta_0$ is rationally oriented.  This means there exist vectors $b_j\in H_{\theta_0,0}\cap\ZZ^d,j\leq d-1$, which form a basis for $H_{\theta_0,0}$.  The sites of form $z_{\bdn} =\sum_{j=1}^{d-1} n_jb_j$ with $\bdn=(n_1,\dots,n_{d-1})\in\ZZ^{d-1}$ form a lattice in $H_{\theta_0,0}$ which divides $H_{\theta_0,0}$ into parallelapideds.  Let $\Lambda_{\bdn}$ denote the parallelapided with opposite corners $z_{\bdn},z_{\bdm}$ with $\bdm=(n_1+1,\dots,n_{d-1}+1)$; these parallelapideds tile $H_{\theta_0,0}$.

For $a\neq b$ in $\RR^d$ we write $\alpha_{ab}$ for $(b-a)/|b-a|$.  
Let $u,v$ be as in the theorem statement.
If $d_{\theta_0}(0,\Pi_{uv}) \geq c_1\Delta(|u|)(\log |u|)^2$ then it is easily checked that for $r = |v-u|/|y_{\alpha_{uv}}|$ translating the picture by $u$ we have $D_{\alpha_{uv},r}(-u) \geq c_2\Delta(|u|)^2(\log |u|)^4/\Xi(|u|)^2 \geq c_3(\log |u|)^3$, so from Proposition \ref{transfluct2},
\begin{equation}\label{crookedcost}
  P(0\in\Gamma_{uv}) \leq C_{26}e^{-c_4(\log |u|)^3}, 
\end{equation}
and \eqref{hitpoint2} follows.  Further, if $|u|$ is large and $\psi_{\alpha_{uv},\alpha_{u0}} > \ep_{\min}/8$ then similarly to \eqref{minwhich} we have 
$D_{\alpha_{uv},r}(-u) \geq c_5 \Phi(|u|) \geq c_6(\log |u|)^3$ so again \eqref{crookedcost} holds.
Therefore we need only consider the case where passage of $\Gamma_{uv}$ through 0 involves a ``small transverse fluctuation'' in the sense that
\begin{equation}\label{closeline}
  d_{\theta_0}(0,\Pi_{uv}) < c_1(\log |u|)^2\Delta(|u|) \quad\text{and}\quad 
    \psi_{\alpha_{uv},\alpha_{u0}} \leq \frac{\ep_{\min}}{8}.
  \end{equation}
Let $R = -(u_1^{\theta_0} + \mu\sqrt{d})$, so $R,|u|$ are of the same order.
Similarly to \eqref{crookedcost} we have for $\chi_2>\chi$ and $|u|$ large
\begin{equation}\label{cylcost}
  P\big(y \in\Gamma_{uv} \text{ for some $y\in H_{\theta_0,R}^-$ with $|y-u| \geq c_7(\log |u|)^{2/(1-\chi_2)}$} \big)
    \leq C_{26}e^{-c_6(\log |u|)^2},
\end{equation}
since any such $y$ satisfies $D_{\alpha,r}(y-u)\geq c_9\Phi((\log |u|)^{2/(1-\chi_2)}) \geq c_{10}(\log |u|)^2$, with $\chi_2$ from A3.
Again similarly, we get
\begin{equation}\label{cylcost2}
  P\big(z \in\Gamma_{u0} \text{ for some $z\in H_{\theta_0,0}^+$ with $|z| \geq c_7(\log |u|)^{2/(1-\chi_2)}$} \big)
    \leq C_{26}e^{-c_6(\log |u|)^2}.
\end{equation}

We will bound $P(0\in\Gamma_{uv})$ essentially by considering the expected number of sites $x$ in certain fattened blocks for which the translated event $x\in\Gamma_{x+u,x+v}$ occurs, and relating this expected number to the expected number of certain entry points, which can be bounded using Proposition \ref{jammed1}.   
Note that $x+u \in H_{\theta_0,-R}^{\rm rfat}$ for all $x\in H_{\theta_0,0}^{\rm fat}$.  For $x,z\in H_{\theta_0,0}^{\rm fat}\cap\ZZ^d$ and $y\in H_{\theta_0,-R}^{\rm rfat}$ define events
\begin{align*}
  C_{x,y,z}:\ &x\in\Gamma_{x+u,x+v},\, y \text{ is the last site of $\Gamma_{x+u,x+v}$ in } H_{\theta_0,-R}^-, \\
  &\qquad z \text{ is the first site of $\Gamma_{y,x+v}$ in } H_{\theta_0,0}^+,\\
\end{align*}
\begin{align*}
  F_x = \bigcup &\bigg\{ C_{x,y,z}: y\in H_{\theta_0,-R}^{\rm rfat}\cap \ZZ^d,\, |y-(x+u)| \leq c_7(\log |u|)^{2/(1-\chi_2)},\\
  &\qquad z\in H_{\theta_0,0}^{\rm fat}\cap \ZZ^d, |z-x| \leq c_7(\log |u|)^{2/(1-\chi_2)} \bigg\}, 
\end{align*}
and let
\[
  X_{uv} = \{x\in H_{\theta_0,0}^{\rm fat}\cap\ZZ^d:\tau\in F_x\};
\]
see Figure \ref{fig-thm1-8}.
For a given configuration $\tau$ with $0\in\Gamma_{uv}$, we either have $\tau\in F_0$ (meaning the sites $y,z$ in the event $C_{0,y,z}$ are close to $u,0$ respectively) or one of the events in \eqref{cylcost} or \eqref{cylcost2} occurs.  Therefore
\begin{equation}\label{zeroin}
  P(0\in\Gamma_{uv}) \leq P(F_0) + 2C_{26}e^{-c_6(\log |u|)^2}.
\end{equation}
Define next the random set of entry points
\begin{align*}
  \tilde W_u &:= \Big\{ z\in H_{\theta_0,0}^{\rm fat}\cap \ZZ^d: \text{ for some $x\in H_{\theta_0,0}^{\rm fat}$ and $y\in 
    H_{\theta_0,-R}^{\rm rfat}$ with $|z-x|\leq c_7(\log |u|)^{2/(1-\chi_2)}$} \\
  &\qquad\qquad\text{and $|y-(x+u)|\leq c_7(\log |u|)^{2/(1-\chi_2)}$, there is a $\theta_0$--slab geodesic from $y$ to $H_{\theta_0,R}^+$}\\
  &\qquad\qquad\text{with $H_{\theta_0,0}^+$--entry point $z$} \Big\}.  
\end{align*}
(a variant of $W_R$ in the previous proof.)
If $\tau\in C_{x,y,z}$ for some $C_{x,y,z}\subset F_x$, then $z\in \tilde W_u$.  Let $n=(\pi(u)\log |u|)^K$ with $K$ to be specified, let $\beta_j$ be as in A4, and as in section \ref{densely}, subdivide $H_{\theta_0,-R}$ and $H_{\theta_0,0}$ each into blocks of side $2n^{-\beta_0}\Delta(R)$. Fix a block $\hat B$ of $H_{\theta_0,0}$ and let $\ol z$ be its center; then $\ol z+u\in H_{\theta_0,-R}^{\rm rfat}$. Let $\hat y$ be the $\theta_0$--projection of $\ol z+u$ into $H_{\theta_0,-R}$, let $Q_{\hat B}$ be the block of $H_{\theta_0,-R}$ containing $\hat y$, and let $\ol y$ be the center of $Q_{\hat B}$.  Then $\ol z+u\in Q_{\hat B}^{\rm rfat}$.

\begin{figure}
\includegraphics[width=14cm]{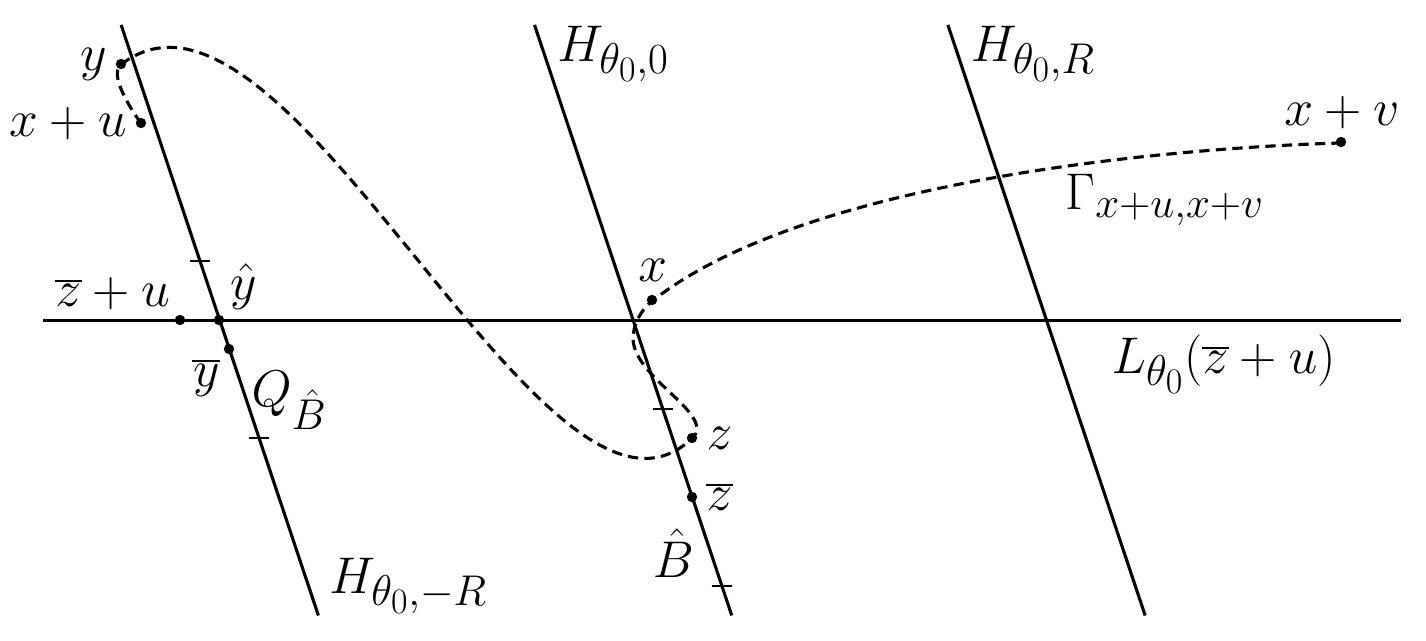}
\caption{ To bound the probability of $0\in\Gamma_{uv}$ we consider sites $x\in H_{\theta_0,0}^{\rm fat}$ for which $x\in\Gamma_{x+u,x+v}$. For such $x$, outside of a small-probability event, there exist $y$ near $x+u$ and $z$ near $x$ for which $z$ is the $H_{\theta_0,0}^+$--entry point of a slab geodesic from $y$ to $H_{\theta_0,R}^+$, as shown. In this case, for the block $\hat B$ with $z\in\hat B^{\rm fat}$, we can translate $\hat B$ by approximately $u$ to obtain the block $Q_{\hat B}$, and the slab geodesic start point
$y$ must lie in a fattened block $B_{-R,z,j}^{\rm rfat}$ close to $Q_{\hat B}$. If there are too many $x$ with $x\in\Gamma_{x+u,x+v}$, then for some block $B_{-R,z,j}^{\rm rfat}$ close to $Q_{\hat B}$, there are many such $H_{\theta_0,0}^+$--entry points $z$ which have some $y\in B_{-R,z,j}^{\rm rfat}$ as corresponding start point, which has small probability by Proposition \ref{jammed1}.}
\label{fig-thm1-8}
\end{figure}

From the definition of $\tilde W_u$, given $z\in \tilde W_u\cap\hat B^{\rm fat}$ there exist $x\in H_{\theta_0,0}^{\rm fat}$ and $y\in H_{\theta_0,-R}^{\rm rfat}$ with $|z-x|\leq c_7(\log |u|)^{2/(1-\chi_2)}$ and $|y-(x+u)|\leq c_7(\log |u|)^{2/(1-\chi_2)}$ such that $\Gamma_{yz}$ is a slab geodesic. Since by \eqref{closeproj} $d_{\theta_0}(z,H_{\theta_0,0}) \leq d$, we have $|z-\ol z| \leq d + \sqrt{d-1}n^{-\beta_0}\Delta(R)$. It follows that, provided $|u|$ (and hence $R$) is large,
\begin{align}\label{yloc}
  d(y,Q_{\hat B}^{\rm rfat}) &\leq |y-(\ol z+u)| \notag\\
  &\leq |y-(x+u)| + |(x+u)-(z+u)| + |(z+u)-(\ol z+u)| \notag \\
  &\leq 2c_7(\log |u|)^{2/(1-\chi_2)} + d + \sqrt{d-1}n^{-\beta_0}\Delta(R) \notag\\
  &\leq 2\sqrt{d-1}n^{-\beta_0}\Delta(R).
\end{align}
Let $\{B_{-R,z,j},j\leq N_d\}$ be those blocks $B$ of $H_{\theta_0,-R}$ satisfying $d(Q_{\hat B}^{\rm rfat},B)\leq 2\sqrt{d-1}n^{-\beta_0}\Delta(R)$; the number $N_d$ of such blocks depends only on $d$, and by \eqref{yloc}, $y$ must be in one of these blocks (backwards-fattened.)  Suppose now that $|\tilde W_u\cap\hat B^{\rm fat}| \geq N_dn$; then some $B_{-R,z,j}^{\rm rfat}$ with $j\leq N_d$ contains the starting points $y$ of at least $n$ of the corresponding $\theta_0$--slab geodesics from the definition of $\tilde W_u$. Let $\ol y_j$ be the center of $B_{-R,z,j}^{\rm rfat}$ and $\zeta_j = (\ol z-\ol y_j)/|\ol z-\ol y_j|$. To apply Proposition \ref{jammed1} we need to bound $\psi_{\theta_0\zeta_j}$; for this we will use
\begin{equation}\label{psis}
  \psi_{\theta_0\zeta_j} \leq \psi_{\theta_0,\alpha_{uv}} + \psi_{\alpha_{uv},\alpha_{u0}} + \psi_{\alpha_{u0},\alpha_{zy}}
    + \psi_{\alpha_{zy},\alpha_{\ol z\ol y_j}}
\end{equation}
The first two angles on the right are bounded by an assumption of the theorem and by \eqref{closeline}, so we will bound the last two.  From the bounds on $|z-x|$ and $|y-(x+u)|$ in the definition of $\tilde W_u$, provided $|u|$ is large we have $\psi_{\alpha_{u0},\alpha_{zy}} \leq c_{11}(\log |u|)^{2/(1-\chi_2)}/|u| \leq \ep_{\min}/8$. Since the pairs $y,\ol y_j$ and $z,\ol z$ each lie in the same (forwards or backwards) fattened block, we have
$\psi_{\alpha_{zy},\alpha_{\ol z\ol y_j}} \leq c_{12}n^{-\beta_0}\Delta(R)/R < \ep_{\min}/8$, again provided $|u|$ (and hence $R$) is large.  Provided we take $\ep_4\leq\ep_{\min}/8$ we thus obtain from \eqref{psis} that $\psi_{\theta_0\zeta_j} < \ep_4 + 3\ep_{\min}/8 < \ep_{\min}/2$, and then from Proposition \ref{jammed1} and \eqref{Rncond} that, provided $K$ is large enough,
\[
  P(|\tilde W_u\cap\hat B^{\rm fat}| \geq N_dn) \leq N_d\exp\left(-C_{51}\frac{n^{2\beta_1}}{\log R} \right)
    \leq e^{-c_{13}(\log |u|)^2},
\]
and therefore, recalling $R,|u|$ are of the same order,
\[
  E\Big( |\tilde W_u\cap\hat B^{\rm fat}| \Big) \leq N_dn + c_{14}(n^{-\beta_0}\Delta(R))^{d-1}e^{-c_{13}(\log |u|)^2} \leq 2N_dn
    =2N_d(\pi(u)\log |u|)^K.
\]
Now $x\in X_{uv}$ implies $z\in \tilde W_u$ for some $z$ with $|z-x| \leq c_7(\log |u|)^{2/(1-\chi_2)}$, so we have 
\begin{equation}\label{Xuvsize}
  E\Big( |X_{uv}\cap\hat B^{\rm fat}| \Big) \leq c_{15}(\log |u|)^{2(d-1)/(1-\chi_2)}E\Big( |\tilde W_u\cap\hat B^{\rm fat}| \Big)
   \leq c_{16}(\pi(u)\log |u|)^{K+2(d-1)/(1-\chi_2)}.
\end{equation}
Observe that $P(x\in X_{uv}) = P(F_x)$ has period $\Lambda_{\bdz}$, in the sense that it takes the same value at $x$ and $x+b_j$ for all $x$ and all $j\leq d-1$.  Therefore, letting $q=|\{\bdn: \Lambda_{\bdn} \subset \hat B\}|$, we have from \eqref{Xuvsize}
\begin{equation}\label{PFxbound}
  \sup_x P(F_x) \leq \sum_{x\in\Lambda_{\bdz}} P(F_x) = E\Big( |X_{uv}\cap\Lambda_{\bdz}^{\rm fat}| \Big) 
    \leq \frac 1q E\Big( |X_{uv}\cap\hat B^{\rm fat}| \Big) \leq \frac{c_{16}}{q} (\pi(u)\log |u|)^{K+2(d-1)/(1-\chi_2)}.
\end{equation}
Provided $|u|$ (and hence $\hat B$) is large, we have $q|\Lambda_{\bdz}| \geq |\hat B|/2$ so 
\[
  q \geq \frac{c_{17}}{|\Lambda_{\bdz}|} n^{-(d-1)\beta_0} \Delta(R)^{d-1} \geq c_{18} (\pi(u)\log |u|)^{-(d-1)\beta_0K}\Delta(|u|)^{d-1}
\]
which with \eqref{PFxbound} shows that, provided $K$ is large,
\begin{equation}\label{PFx2}
  \sup_x P(F_x) \leq \frac{c_{19}(\pi(u)\log |u|)^{dK}}{\Delta(|u|)^{d-1}}.
\end{equation}
This and \eqref{zeroin} complete the proof.

\section{Nonexistence of bigeodesics--proof}\label{nobigeo}

In this section we prove parts (iii) and (iv)(c) of Theorem \ref{alldim}. We begin with (iii); the main idea is that all bigeodesics as in (iii) are $\theta$--rays in one direction and $(-\theta)$--rays in the other, for some $\theta\in J(\theta_0,\ep_2)$, and the crossing density of such bigeodesics is bounded above by $\ol\rho_{J(\theta_0,\ep_2),R}$ for all $R$, up to a small error term.

We may assume $\theta_0$ is rationally oriented.  Suppose A2, A3 hold for some $\theta_0,\ep_0$, fix $\ep<\ep_0$ to be specified, and suppose 
\begin{equation}\label{biexist}
  P\left(\text{there exists a bigeodesic containing a subsequential $\theta$--ray for some $\theta\in J(\theta_0,\ep)$} \right) > 0.
\end{equation}
Let $\Gamma = (x_i,i\in \ZZ)$ be such a bigeodesic; 
we may assume the labeling is such that subsequential $\theta$--ray is $(x_i,i\geq 0)$.  By Proposition \ref{rayfluct}(iii), $(x_i,i\geq 0)$ is a $\theta$--ray.  We write $\Gamma[x_j,\infty)$ for the $\theta$--ray $(x_i,i\geq j)$ and $\Gamma(-\infty,x_j]$ for the geodesic ray $(x_i,i\leq j)$.

We claim $(x_{-i},i\geq 0)$ is a $(-\theta)$--ray for every such $\theta$ and $\Gamma$, a.s.  If not, some such $\theta$--ray is a subsequential $\varphi$--ray for some $\varphi\neq-\theta$, so there exists $i_k\to\infty$ for which $x_{-i_k}/|x_{-i_k}|\to\varphi$.  Letting $r_k = |x_0-x_{-i_k}|$, the distance from $x_0$ to the ray $\{x_{-i_k}+t\theta:t\geq 0\}$ is then of order $\psi_{\varphi,-\theta}r_k$, and
it is easily checked that we therefore have for all sufficiently large $k$
\begin{equation}\label{Dmin}
  \sup_{u\in \Gamma[x_{-i_k},\infty)} D_\theta(u-x_{-i_k}) \geq D_\theta(x_0-x_{-i_k}) 
    \geq c_1\frac{(\psi_{\varphi,-\theta}r_k)^2}{\Xi(r_k)^2} \geq c_2\psi_{\varphi,-\theta}^2\frac{r_k}{\sigma(r_k)\log r_k}.
\end{equation}
But for the events
\begin{align*}
  F_{\delta,j}: &\text{ for some $x\in \ZZ^d$ with $2^{j-1}<|x|\leq 2^j$ and some $\theta\in J(\theta_0,\ep)$ with } \\
  &\qquad\text{ $D_\theta(-x)\geq c_3\delta^2\frac{|x|}{\sigma(|x|)\log |x|}$, there is a $\theta$--ray from $x$ containing 0,  }
\end{align*}
we have from Proposition \ref{rayfluct} that
\begin{equation}\label{nobend2}
  P(F_{\delta,j}) \leq c_42^{dj} e^{-c_5\delta^22^j/j\sigma(2^j)} \quad\text{so}\quad P(F_{\delta,j} \text{ i.o.~in }j) = 0.
\end{equation}
The same is true if we translate the events by $x_0$, replacing 0 with $x_0$ and $|x|$ with $|x-x_0|$, so this proves the claim that $(x_{-i},i\geq 0)$ is a $(-\theta)$--ray.
$\Gamma$ is thus a $\theta$--bigeodesic, by which we mean a  bigeodesic which is a $\theta$--ray in one direction, and a $(-\theta)$--ray in the opposite direction.  We have shown that with probability 1, every bigeodesic $\Gamma$ containing a subsequential $\theta$--ray for some $\theta\in J(\theta_0,\ep)$ is actually a $\theta$--bigeodesic, so it has a well-defined entry point $x_{\theta_0,0}''(\Gamma)$ in $H_{\theta_0,0}^+$. Let 
\[
  \mC_{J(\theta_0,\ep),0}^{\rm bi}(A) = \Big\{x\in A: x = x_{\theta_0,0}''(\Gamma) \text{ for some $\theta\in J(\theta_0,\ep)$ and 
    $\theta$--bigeodesic } \Gamma \Big\}.
\]
We may consider the ``largest $\theta_0$--backtrack after $x_{\theta_0,0}''(\Gamma)$'', or more precisely, the value
\[
  R_0(\Gamma) = \min\{u_1^{\theta_0}: u\in \Gamma[x_{\theta_0,0}''(\Gamma),\infty)\} \wedge 0,
\]
which by the preceding is finite for every $\theta$--bigeodesic $\Gamma$ with $\theta \in J(\theta_0,\ep)$.  

Before continuing let us recall that the mean (combined) $H_{\theta_0,R}$--crossing density is defined in \eqref{rhobar} and \eqref{rhobarJ} by counting entry points in $H_{\theta_0,R}^+$ of geodesic rays having only their initial site in $H_{\theta_0,0}^-$. We could equally well consider a ``shift by $R$'': entry points in $H_{\theta_0,0}^+$ of geodesic rays having only their initial site in $H_{\theta_0,-R}^-$. It is enough to consider $R$ for which $H_{\theta_0,-R}$ (and therefore also $H_{\theta_0,R}$) contains a lattice point; by stationarity, for such $R$ the shift by $R$ does not alter the crossing density.
For $R>0$ we split $\mC_{J(\theta_0,\ep),0}^{\rm bi}(A)$ into a ``large--backtrack'' set
\[
  \mC_{J(\theta_0,\ep),0}^{{\rm bi},R+}(A) = \Big\{x\in A: x = x_{\theta_0,0}''(\Gamma) \text{ for some $\theta\in J(\theta_0,\ep)$ and 
    $\theta$--bigeodesic $\Gamma$ with $R_0(\Gamma)\geq R$} \Big\}
\]
and a small--backtrack set $\mC_{J(\theta_0,\ep),0}^{{\rm bi},R-}(A) = \mC_{J(\theta_0,\ep),0}^{\rm bi}(A) \bs \mC_{J(\theta_0,\ep),0}^{{\rm bi},R+}(A)$.
For $x\in\mC_{J(\theta_0,\ep),0}^{{\rm bi},R-}(A)$, there exists a $\theta$--bigeodesic $\Gamma$ with
$R>|R_0(\Gamma)|$ and with $H_{\theta_0,0}^+$--entry point $x$, and letting $y_{-R}(\Gamma)$ be the last point of $\Gamma$ in $H_{\theta_0,-R}^-$, we have that $\Gamma[y_{-R}(\Gamma),\infty)$ is a halfspace $\theta$--ray from $H_{\theta_0,-R}^-$ with the same $H_{\theta_0,0}^+$--entry point $x$. 
Now the mean density of bigeodesic entry points, given by
\[
  \ol\rho_{J(\theta_0,\ep),0}^{\rm bi} = \limsup_{r\to\infty} \frac{E\big(| \mC_{J(\theta_0,\ep),0}^{\rm bi}(H_{\theta_0,0}^{\rm fat} \cap B_r(0))|\big)} 
    {\Vol_{d-1}(H_{\theta_0,0} \cap B_r(0)) }
\]
satisfies 
\[
  \ol\rho_{J(\theta_0,\ep),0}^{\rm bi} \leq \limsup_{r\to\infty} 
    \frac{E\big(| \mC_{J(\theta_0,\ep),0}^{{\rm bi},R-}(H_{\theta_0,0}^{\rm fat} \cap B_r(0))|\big)} {\Vol_{d-1}(H_{\theta_0,0} \cap B_r(0)) }
    + \limsup_{r\to\infty} 
    \frac{E\big(| \mC_{J(\theta_0,\ep),0}^{{\rm bi},R+}(H_{\theta_0,0}^{\rm fat} \cap B_r(0))|\big)} {\Vol_{d-1}(H_{\theta_0,0} \cap B_r(0)) },
\]
for all $R>0$, and, by the preceding remark about $x\in\mC_{J(\theta_0,\ep),0}^{{\rm bi},R-}(A)$, the first lim sup on the right is bounded above by $\ol\rho_{J(\theta_0,\ep),R}$. The second lim sup is bounded by
\[
  c_6P\Big(\text{for some $\theta\in J(\theta_0,\ep)$ there exists a $\theta$--ray from 0 which intersects $H_{\theta_0,-R}$}\Big),
\]
which (cf.~Lemma \ref{badbehavG2G7}) is readily shown to approach 0 as $R\to\infty$. Since $R$ is arbitrary, provided $\ep\leq\ep_2$ it then follows from Theorem \ref{alldim}(ii) that $ \ol\rho_{J(\theta_0,\ep),0}^{\rm bi}=0$.  Since $\theta_0$ is rationally oriented, periodicity of $P(x \in \mC_{J(\theta_0,\ep),0}^{\rm bi}(H_{\theta_0,0}^{\rm fat}))$ then means that we have $P(x \in \mC_{J(\theta_0,\ep),0}^{\rm bi}(H_{\theta_0,0}^{\rm fat})) = 0$ for all sites $x \in H_{\theta_0,0}^{\rm fat}$, which proves Theorem \ref{alldim}(iii).  Then (iv)(c) follows from (iii) and compactness of $S^{d-1}$.

\section{Coalescence time bounds--proof}\label{coal}
In this section we prove Theorem \ref{coalesce}.  We start with the upper bound on $P( (U_{xy}^\theta)_1^\theta \geq r )$, as the lower bound is much simpler.

Let $\ep_3 = \min(\ep_0/2,\ep_2/2,\ep_6$), where $\ep_2$ is from Theorem \ref{alldim} and $\ep_6$ from Lemma \ref{coordchg}. Let $\theta\in J(\theta_0,\ep_3)$ and let $x,y$ be $\theta$--start sites with second coordinates $x_2<y_2$.  Now $x,y$ do not necessarily lie in the hyperplane $H_{\theta,0}$, and our first task is to effectively reduce to the case where they do so lie, for technical convenience.
Provided $|y-x|$ is large, there exists $\tth$ with $\psi_{\theta\tth}<\ep_3$ (and thus $\psi_{\theta_0\tth}<2\ep_3$) for which $y-x\in H_{\tth,0}$, so $y-x$ and 0 are both $\tilde\theta$--start sites.
We may assume $y-x$ makes an angle of at least $\pi/4$ with the horizontal axis; then there is exactly one $\tth$--start site $z_k$ at each integer height $k$.  We may also assume $(y-x)_2^\theta>0$. We have $x_1^\theta,y_1^\theta \in [-\mu\sqrt{d},0]$, and we observe that provided $\ep_3$ is small, if $(U_{0,y-x}^\theta)_1^{\theta} \geq r$ and $(U_{0,y-x}^\theta)_2^\theta < (U_{0,y-x}^\theta)_1^\theta$, then $(U_{0,y-x}^\theta)_1^{\tth} \geq r/2$.  Using \eqref{minwhich} and Proposition \ref{rayfluct}(ii) we obtain from this that for large $r$,
\begin{align}\label{tvstilde}
  P( (U_{xy}^\theta)_1^\theta \geq r ) &= P( (U_{0,y-x}^\theta)_1^\theta \geq r-x_1^\theta ) \notag\\
  &\leq P\left( (U_{0,y-x}^\theta)_1^{\tth} \geq \frac r2 \right) + 
    P\Big( (U_{0,y-x}^\theta)_2^\theta \geq (U_{0,y-x}^\theta)_1^\theta \geq r \Big) \notag\\
  &\leq P\left( (U_{0,y-x}^\theta)_1^{\tth} \geq \frac r2 \right) + 
    P\Big( \sup_{u\in \Gamma_0^\theta} D_\theta(u) \geq \Phi(r) \Big) \notag\\
  &\leq P\left( (U_{0,y-x}^\theta)_1^{\tth} \geq \frac r2 \right) + c_1e^{-c_2\Phi(r)}.
\end{align}
It follows that we may assume $x=0$ and consider $(U_{0y}^\theta)_1^{\tth}$ with $y\in H_{\tth,0},\,y_2>0$, and $C_{52}\leq |y|\leq \Delta(r)=r^\xi\nu_\Delta(r)$.

For points $u,v\in H_{\tth,0}$ we use notation $[u,v]$ for the interval from $u$ to $v$ in $H_{\tth,0}$; we call such an interval an $H$--\emph{interval}.
For each $\tth$--start site $z$ let $\bar z$ be its projection horizontally into $H_{\tth,0}$. 
(Throughout this section, ``projection'' will mean horizontal projection into $H_{\tth,0}$, unless stated otherwise.)

Recall $\theta$--sources as defined after \eqref{thdecomp}, and (since $d=2$ here) the unique $\theta$--ray $\Gamma_z^\theta$ from a site $z$. Let $\mZ_{\tth}$ denote the set of all $\tth$--start sites, and $\mS_{\tth,\theta} \subset \mZ_{\tth}$ the (random) subset which are $\theta$--sources. Let $r>1$, and for each $z\in \mZ_{\tth}$ let $V_z$ be the last site of $\Gamma_z^\theta$ in $H_{\tth,0}^-$ 
(so necessarily $V_z\in\mS_{\tth,\theta}$)
and let $W_z$ be the $H_{\tth,r}^+$--entry point of $\Gamma_{V_z}^\theta$. (See Figure \ref{fig1-7A}.)  A $(\tth,\theta,r)$--\emph{gap} is an open $H$--interval $I$ in $H_{\tth,0}$ with the properties that (i) $I$ contains no projected $\theta$--source $\ol V_z$, and (ii) the endpoints $\ol v,\ol w$ of $I$ are projections of sources $v,w\in \mS_{\tth,\theta}$ with $W_v\neq W_w$.  A $(\tth,\theta,r)$--\emph{entry interval} is the closed interval between two successive $(\tth,\theta,r)$--gaps. It then follows from planarity of $\ZZ^2$ that any two $\theta$--sources $v,w$ satisfy $W_v=W_w$ if and only if $\ol v, \ol w$ lie in the same $(\tth,\theta,r)$--entry interval; thus the gaps separate those groups of halfspace $\theta$--rays from $H_{\tth,0}^-$ which coalesce before crossing $H_{\tth,r}$.  Equivalently,  
\begin{equation}\label{gapmeaning}
  v,w \in \mS_{\tth,\theta},\ (U_{vw}^\theta)_1^{\tth} \geq r \implies \text{ there is a $(\tth,\theta,r)$--gap between $v$ and $w$}.
\end{equation}
(It should be emphasized that this is only true for $\theta$--sources, not general $\tth$--start sites.)

\begin{figure}
\includegraphics[width=14cm]{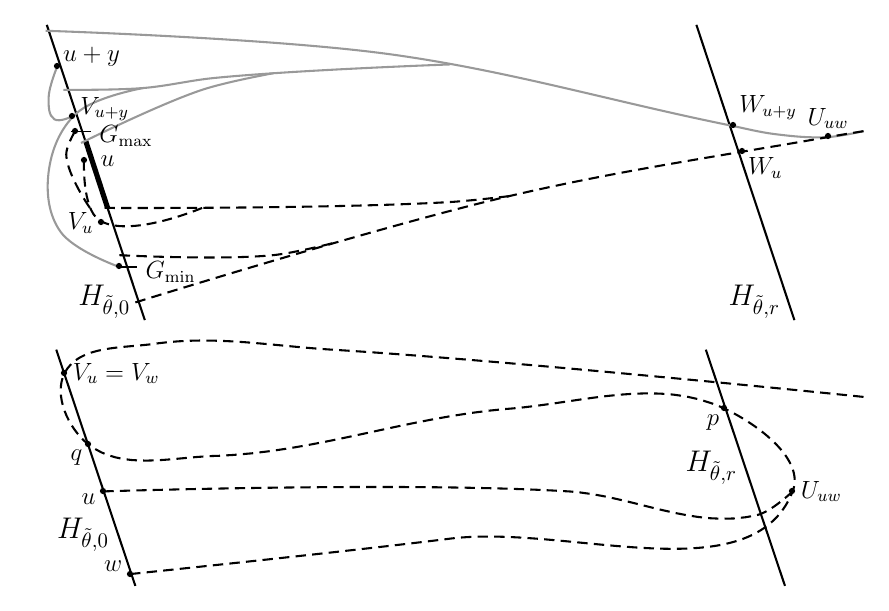}
\caption{ {\it Top}: The gray $\theta$--rays share $W_{u+y}$ as their common $H_{\tilde\theta,r}^+$--entry point; the dashed ones similarly share $W_u$. The thick segment in $H_{\tilde\theta,0}$ is the $(\tilde\theta,\theta,r)$--gap $G_{u,u+y}$, separating those gray and dashed $\theta$--rays which are halfspace $\theta$--rays. The hash marks on $H_{\tilde\theta,0}$ show the corresponding enlarged $(\tilde\theta,\theta,r)$--gap $(G_{\min},G_{\max})$; the lowest start point of a gray geodesic is at $G_{\min}$, and the highest start point of a dashed one is at $G_{\max}$.  {\it Bottom}: The event $A_{uw}$.  The $H_{\tilde\theta,r}^+$--entry points from $u$ and $w$ are different, but $V_u=V_w$. }
\label{fig1-7A}
\end{figure}

We note that $\tth$--start sites are periodic in the sense that $u$ is a $\tth$--start site if and only if $u+y$ is one. (This is why we arranged to have $y\in H_{\tth,0}$.)  We now consider translates of the events in \eqref{tvstilde} corresponding to $\tth$--start sites $u,u+y$ in place of $0,y$. We have for all $u\in \mZ_{\tth}$:
\[
  P\Big( (U_{u,u+y}^\theta)_1^{\tth} \geq r \Big) = P\Big( (U_{0y}^\theta)_1^{\tth} \geq r-u_1^{\tth} \Big)
    \geq P\Big( (U_{0y}^\theta)_1^{\tth} \geq r + \mu\sqrt{d} \Big),
\]
and therefore, averaging over a period,
\begin{equation}\label{peravg}
  P\Big( (U_{0y}^\theta)_1^{\tth} \geq r + \mu\sqrt{d} \Big) \leq \frac{1}{y_2}\sum_{k=0}^{y_2-1} 
    P\Big( (U_{z_k,z_k+y}^\theta)_1^{\tth} \geq r \Big).
\end{equation}

For $\tth$--start sites $u,w$, it is possible that $(U_{uw}^\theta)_1^{\tth} \geq r$ and the continuation $\Gamma_u^\theta\cap\Gamma_w^\theta = \Gamma_{U_{uw}}^\theta$ after coalescence backtracks to visit $H_{\tth,0}^-$, so that $V_u=V_w$; let us call this event $A_{uw}$.  See again Figure \ref{fig1-7A}. Given $\tau\in A_{uw}$ let $q$ be first point of $\Gamma_{U_{uw}}^\theta$ in $H_{\tth,0}^-$, and $p$ the last point of $\Gamma_{U_{uw}}^\theta$ before $q$ with $p\in H_{\tth,r}^+$.  Then $(p-u)_1^{\tth}\in [r,r+2\mu\sqrt{d}]$.  If $|(p-u)_2^{\tth}|\geq 2r$ then by Lemma \ref{coordchg} we also have $|(p-u)_2^\theta|\geq |(p-u)_1^\theta|\geq r/2$ so $D_\theta(p-u) \geq \Phi(r/2)$.  If instead $|(p-u)_2^{\tth}|< 2r$ then we use the readily-verified fact that for such $p$ we have $D_\theta(q-p) \geq \Phi(r/2)$ for all $q\in H_{\tth,0}^-$. 
Let $K_r = \lfloor \log_2 r \rfloor$ and define corresponding events for these two situations:
\[
  A_u^\prime: \text{ there exist $q,p$ with $|p-u|\leq (3+|y_{\tth}|)r,\ q\in\Gamma_p^\theta,$ and $D_\theta(q-p) \geq \Phi(r/2)$},
\]
\[
  A_u^{\prime\prime}: \text{ there exists $p\in\Gamma_u^\theta$ with $D_\theta(p-u) \geq \Phi(r/2)$},
\]
so that from Proposition \ref{rayfluct}(ii), taking $w=u+y$,
\begin{align}\label{bigback}
  P(A_{u,u+y}) &\leq P(A_u^\prime) +P(A_u^{\prime\prime}) \leq c_3r^{2d} e^{-C_{27}\Phi(r/2)} + C_{26}e^{-C_{27}\Phi(r/2)}
    \leq e^{-C_{27}\Phi(r/2)/2}.
\end{align}
It remains to consider the ``main'' event that $(U_{u,u+y}^\theta)_1^{\tth} \geq r$ but $\tau\notin A_{u,u+y}$.  In this case, for $\theta$--sources $V_u,V_{u+y}$, coalescence of $\Gamma_{V_u}^\theta,\Gamma_{V_{u+y}}^\theta$ occurs in $H_{\tth,r}^+$, so by \eqref{gapmeaning} there must be a $(\tth,\theta,r)$--gap between $\ol V_u$ and $\ol V_{u+y}$; let $G_{u,u+y}$ be longest such gap, breaking ties arbitrarily.  
There are 2  cases to consider:
\begin{itemize}
\item[(1)] $|\ol u - \ol V_u| \geq (\log r)^{c_4}$ or $|\ol{u+y} - \ol V_{u+y}| \geq (\log r)^{c_4}$
\item[(2)]  $\max(d(\ol u,G_{u,u+y}),d(\ol{u+y},G_{u,u+y})) \leq |y| + 
  2\max(|\ol u - \ol V_u|,|\ol{u+y} - \ol V_{u+y}|) \leq |y| + 2(\log r)^{c_4}$,
\end{itemize}
where $c_4$ is chosen so $\Phi((\log r)^{c_4}) \geq (\log r)^2$.
In case (1) if $|\ol u - \ol V_u| \geq (\log r)^{c_4}$ then we have $V_u\in\Gamma_u^\theta$ and $D_\theta(V_u-u)\geq \frac 12 \Phi((\log r)^{c_4})$, and similarly for $u+y$ in place of $u$. Therefore by Proposition \ref{rayfluct}(ii) we have
\begin{equation}\label{noback1}
  P\Big( (U_{u,u+y}^\theta)_1^{\tth} \geq r, \tau\notin A_{u,u+y}, \text{ and Case (1) holds}\Big) 
    \leq C_{26}e^{-C_{27}\Phi((\log r)^{c_4})} \leq C_{26}e^{-C_{27}(\log r)^2}.
\end{equation}
Case (2) is more complicated.  We have
\[
  |G_{u,u+y}| \leq |\ol V_u - \ol V_{u+y}| \leq |y| + 2(\log r)^{c_4};
\]
we call any $(\tth,\theta,r)$--gap $G$ \emph{short} if $|G|\leq |y| + 2(\log r)^{c_4}$.
Then
\begin{align}\label{noback2}
  P\Big( (U_{u,u+y}^\theta)_1^{\tth} &\geq r, \tau\notin A_{u,u+y}, \text{ and Case (2) holds}\Big) \notag\\
  &\leq P\Big( d(\ol u,G) \leq |y| + 2(\log r)^{c_4} \text{ for some short $(\tth,\theta,r)$--gap } G \Big).
\end{align}
We now take $u=z_k$ and consider the average as in \eqref{peravg}.
Define events
\[
  Q_k: (U_{z_k,z_k+y}^\theta)_1^{\tth} \geq r, \tau\notin A_{z_k,z_k+y},
\]
\[
  R_k: d(\ol z_k,G) \leq |y| + 2(\log r)^{c_4} \text{ for some short $(\tth,\theta,r)$--gap $G$ in $H_{\tth,0}$}.
\]
From \eqref{bigback}, \eqref{noback1}, and \eqref{noback2},
\begin{align}\label{average}
  P\Big( (U_{0y}^\theta)_1^{\tth} \geq r + \mu\sqrt{d} \Big) \leq \frac{1}{y_2}&\sum_{k=0}^{y_2-1} [P(Q_k) + P(A_{z_k,z_k+y})] \leq \frac{1}{y_2}\sum_{k=0}^{y_2-1} P(R_k) + 2C_{26}e^{-C_{27}(\log r)^2}.
\end{align}

We need to bound the average on the right in \eqref{average}. Let $J_t=\{y\in H_{\tth,0}: |y|\leq t\}$ and let $N_t$ be the number of short $(\tth,\theta,r)$--gaps $G$ intersecting $J_t$.  Each corresponding $(\tth,\theta,r)$--entry interval (between two such gaps) contains the projection of a $\theta$--source, and the halfspace $\theta$--rays from these sources have different $H_{\tth,r}^+$--entry points for each $(\tth,\theta,r)$--entry interval.  It then follows using Theorem \ref{alldim}(ii) that for $\rho_\theta(r)$ from \eqref{rho}, 
\begin{equation}\label{Ntrate}
  \limsup_{t\to\infty} \frac{N_t}{2t} \leq \rho_\theta(r) \leq \frac{\nu_2(r)}{r^\xi }.
\end{equation}
We now use periodicity as in the proof of Theorem \ref{hitpoint}.  By the multidimensional ergodic theorem (see \cite{Ge88}, Appendix 14.A), we have 
\begin{equation}\label{ergodic}
  \frac{1}{y_2}\sum_{k=0}^{y_2-1} P(R_k) = \lim_{m\to\infty} \frac{1}{2m+1}\sum_{k=-m}^m 1_{R_k},\ \  \text{a.s.}
\end{equation}
Now assuming $|y|$ is large,
\begin{align}
  \sum_{k=-m}^m 1_{R_k} &\leq 
    \left| \big\{ u\in \mZ_{\tth}: |u_2|\leq m, d(\ol u,G) \leq |y| + 2(\log r)^{c_4} 
    \text{ for some short $(\tth,\theta,r)$--gap } G \big\} \right| \notag\\
  &\leq N_{m|y|/y_2}(3|y| + 6(\log r)^{c_4}) \leq N_{m|y|/y_2}|y|(\log r)^{c_4} \notag
\end{align}
so
\[
  \frac{1}{2m+1}\sum_{k=-m}^m 1_{R_k} \leq \frac{|y|}{y_2}\, \frac{N_{m|y|/y_2}}{2m|y|/y_2}\, \frac{2m}{2m+1}
    |y| (\log r)^{c_4}.
\]
With \eqref{Ntrate} and \eqref{ergodic} this yields
\[
  \frac{1}{y_2}\sum_{k=0}^{y_2-1} P(R_k) 
    \leq c_5 \frac{\nu_2(r)(\log r)^{c_4}|y| } {r^\xi}.
\]
Combining this with \eqref{average} we obtain
\[
  P( (U_{0y}^{\tth})_1^\theta \geq r + \mu\sqrt{d} )\leq c_6 \frac{\nu_2(r)(\log r)^{c_4}|y| } {r^\xi},
\]
which establishes the upper bound in \eqref{coaltime}.

Turning to the lower bound, we need to consider the fact that $z$ and $V_z$ may lie on opposite sides of a gap.  Given a $(\tth,\theta,r)$--gap $G$, we let $G_{\max},G_{\min}$ denote respectively the highest (lowest) projected $\tth$--start site $\ol z$ for which $\ol V_z$ lies below (above) $G$. It is not necessarily true that $G_{\min}$ is below $G_{\max}$, but $G_{\max}$ is in or above $G$, $G_{\min}$ is in or below $G$, and $G_{\max}$ is at most one vertical unit below $G_{\min}$.  It is easily seen that for every $\tth$--start site $z$ we have 
\begin{equation}\label{zVz}
  D_\theta(V_z-z) \geq c_7\Phi(|\ol z - \ol V_z|);
\end{equation} 
it then follows readily from Proposition \ref{rayfluct}(ii) that $G_{\min}, G_{\max}$ always exist for all $G$, a.s.
We call $G^+ := G\cup (G_{\min},G_{\max})$ an \emph{enlarged $(\tth,\theta,r)$--gap}.  We say an enlarged $(\tth,\theta,r)$--gap $G^+$ is \emph{semi-short} if 
$|G^+|\leq |y| + (\log r)^{c_4}$, and \emph{very long} otherwise. A key observation is that 
\begin{equation}\label{jump}
  \text{one of $z= G_{\min},G_{\max}$ must satisfy $|\ol z - \ol V_z| \geq \frac 12 |G^+|-1$.}
\end{equation}

\begin{figure}
\includegraphics[width=14cm]{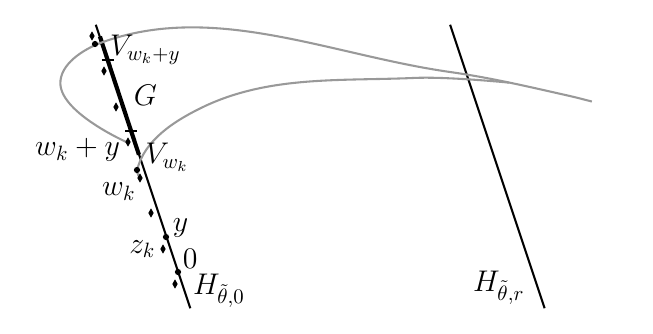}
\caption{ The gray curves are $\theta$--rays. The gap $G$ is marked by the hash marks on $H_{\tilde\theta,0}$; the enlarged gap $G^+$ is the thickened part of $H_{\tilde\theta,0}$. We fix a height $k\in[0,y_2)$; the row of diamonds to the left of $H_{\tilde\theta,0}$ are the sites $z_k+iy, i\in\ZZ$. $w_k$ is the lowest such site for which $V_{w_k}$ and $V_{w_k+y}$ lie on opposite sides of $G$. There is such a site $w_k$ for each $k\in[0,y_2)$, creating $y_2$ occurrences of events $F_j$.  Necessarily at least one of $\ol w_k,\ol{w_k+y}$ lies in $G^+$.  }
\label{fig1-7C}
\end{figure}

Let $a_r=\Delta(r)\log r$ and define overlapping intervals (much longer than a semi--short gap) in $H_{\tth,0}$:
\[
  I_j = \left[ ja_r, (j+2)a_r \right],
\]
and define events
\[
  Y_j: \text{ $I_j$ intersects a semi-short enlarged $(\tth,\theta,r)$--gap $G^+$}.
\]
Suppose $I_j$ intersects a semi-short enlarged $(\tth,\theta,r)$--gap $G^+$, for some gap $G$.  For each $0\leq k<y_2$, consider the points $\{z_k+iy:i\in\ZZ\}$; let $w_k$ be the lowest such point (i.e.~with smallest $i$) for which $V_{z_k+(i-1)y}$ and $V_{z_k+iy}$ are on opposite sides of $G$, noting that no point $V_{z_k+iy}$ can lie inside $G$. Then $w_k\in G^+$, and the points $\{w_k: 0\leq k<y_2\}$ are all distinct.  This shows that
\[
  \Big| \{u\in \mZ_{\tth}: \ol u \in G^+, V_u \text{ and $V_{u+y}$ are on opposite sides of $G$}\} \Big| \geq y_2;
\]
see Figure \ref{fig1-7C}. Now, any given semi-short enlarged $(\tth,\theta,r)$--gap $G^+$ intersects at least one and at most two $H$-intervals $I_j$. It follows that for the events
\[
  F_k: \text{ $V_{z_k}$ and $V_{z_k+y}$ are on opposite sides of some $(\tth,\theta,r)$--gap $G$ for which $G^+$ is semi-short}
\]
we have
\[
  \frac{y_2}{2} \sum_{j=-\ell}^\ell 1_{Y_j} \leq \sum_{k=-\ell\Delta(r)\log r}^{(\ell+2)\Delta(r)\log r} 1_{F_k},
\]
and hence by the ergodic theorem, dividing by $(2\ell+1)\Delta(r)\log r$ and letting $\ell\to\infty$ gives
\begin{align}\label{Y0Fk}
  \frac{y_2}{2\Delta(r)\log r} P(Y_0) &\leq \frac{1}{y_2} \sum_{k=0}^{y_2-1} P(F_k) \notag\\
  &\leq \frac{1}{y_2}\sum_{k=0}^{y_2-1} \left[ P\left( (U_{z_k,z_k+y}^\theta)_1^{\tth} \geq \frac r2 \right)
    + P\left( \Gamma_{W_{z_k}}^\theta \cap H_{\tth,r/2}^- \neq \emptyset \right) \right] \notag\\
  &\leq P\left( (U_{0y}^\theta)_1^{\tth} \geq \frac r2 \right) + c_8c^{-c_9\Phi(r)}.
\end{align}
Here the second inequality reflects that when $\tau\in F_k$, the $H_{\tth,r}^+$--entry points $W_{z_k} \neq W_{z_k+y}$, so either coalescence occurs in $H_{\tth,r/2}^+$ or both geodesics $\Gamma_{z_k}^\theta,\Gamma_{z_k+y}^\theta$ backtrack to $H_{\tth,r/2}^-$ after entering $H_{\tth,r}^+$.  The third inequality reflects that (i) the first probability on the second line is maximized over $k$ when $z_k,z_k+y$ lie in $H_{\tth,0}$, and (ii) the second probability on the second line can be bounded as in \eqref{bigback}.
(Also in \eqref{Y0Fk}, for technical convenience in applying the ergodic theorem, we have assumed $\Delta(r)\log r$ is an integer multiple of $y_2$, ensuring $P(Y_j)$ is the same for all $j$. The added technicalities without this assumption are tedious but straightforward, using our assumption $|y|\leq\Delta(r)$.)

\begin{figure}
\includegraphics[width=14cm]{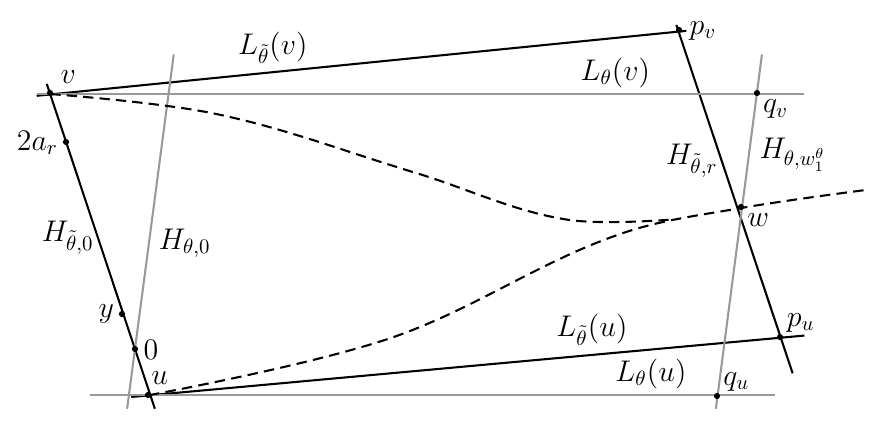}
\caption{ The event that $L_0=[0,2a_r]$ intersects no enlarged $(\tth,\theta,r)$--gap. Because $|v-u|\approx |p_v-p_u|$ is large, the common $H_{\tilde\theta,r}^+$--entry point $w$ from $u$ and $v$ must be far from either $p_u$ or $p_v$. This remains true when we shift slightly from $\tilde\theta$--coordinates (black lines) to $\theta$--coordinates (gray lines), so $\Gamma_u^\theta$ or $\Gamma_v^\theta$ must make a large transverse fluctuation to pass through $w$.  }
\label{fig1-7B}
\end{figure}

Next we show $P(Y_0)$ is near 1; with \eqref{Y0Fk} this will complete the proof of the lower bound in \eqref{coaltime}.  We have
\begin{align}\label{gaptype}
  P(Y_0^c) &\leq P(\text{$I_0$ intersects no enlarged $(\tth,\theta,r)$--gap}) \notag\\
  &\qquad + P(\text{$I_0$ intersects a very long enlarged $(\tth,\theta,r)$--gap}).
\end{align}
Let us consider the first probability on the right of \eqref{gaptype}.
Suppose $I_0$ intersects no enlarged $(\tth,\theta,r)$--gap; then $I_0$ is contained in some $(\tth,\theta,r)$--entry interval $[\ol u,\ol v]$.  See Figure \ref{fig1-7B}. This means $u,v$ are $\theta$--sources with $|\ol u-\ol v| \geq 2\Delta(r)\log r$, and  with $W_u=W_v=w$ for some $w$. Now the angle between $\tth$ and $H_{\tth,r}$ is at least $\pi/4$ by \eqref{yzangle}, and it follows by straightforward geometry that the points $p_u = L_{\tth}(u)\cap H_{\tth,r}$ and $p_v = L_{\tth}(v)\cap H_{\tth,r}$ satisfy
\[
  \Big| |p_u-p_v| - |\ol u-\ol v| \Big| \leq \sqrt{2}|u-\ol u| + \sqrt{2}|v-\ol v| \leq 2\sqrt{2}.
\]
Provided $\psi_{\theta\tth}$ is small, when we change the angle to $\theta$ and consider $q_u = L_{\theta}(u)\cap H_{\theta,w_1^\theta}$ and $q_v = L_{\theta}(v)\cap H_{\theta,w_1^\theta}$, we have via straightforward use of \eqref{linegap2} and \eqref{Vgap} that
\[
  w_1^\theta \geq \frac r2, \quad |q_u - q_v| \geq \frac 34 |p_u-p_v| \geq \frac 12 |\ol u - \ol v|.
\]
Since $w,q_u,q_v$ all lie in $H_{\theta,w_1^\theta}$ follows that 
\[
  \max((w-u)_2^\theta,(w-v)_2^\theta) = \max(|w-q_u|,|w-q_v|) \geq \frac 14 |\ol u - \ol v|. 
\]
We may assume the first entry in the maximum is the larger one.  Then using \eqref{minwhich}, for some $c_{10},c_{11}$,
\[
  D_\theta(w-u) \geq D_*(|\ol u - \ol v|) := \begin{cases} 
    \dfrac{c_{10}|\ol u - \ol v|}{\sigma^*(|\ol u-\ol v|)\log |\ol u - \ol v|} &\text{if } |\ol u - \ol v| \geq r,\\
    \hskip 0.4in \dfrac{c_{11}|\ol u-\ol v|^2}{r\sigma(r)\log r} &\text{if } |\ol u - \ol v|<r. \end{cases}
\]
Let $\hat K_r = \lfloor \log_2( 2\Delta(r)\log r ) \rfloor$, and define events
\begin{align}\label{nogap}
  E_k: &\text{ for some $u,v\in \mZ_{\tth}$ with $I_0\subset [\ol u, \ol v]$ and $2^k < |\ol u - \ol v| \leq 2^{k+1}$, } \sup_{w\in \Gamma_u^\theta} D_\theta(w-u) \geq D_*(2^k).
\end{align}
The preceding together with Proposition \ref{rayfluct}(ii) then shows that
\begin{align}\label{nogap2}
  P(&\text{$I_0$ intersects no enlarged $(\tth,\theta,r)$--gap}) \leq \sum_{k=\hat K_r}^\infty P(E_k) \notag\\
  &\hskip 0.5in \leq \sum_{k=\hat K_r}^\infty c_{12} 2^{2k} \exp\left( -C_{36}D_*(2^k)\log D_*(2^k) \right)
    \leq e^{-c_{13}(\log r)(\log\log r)}.
\end{align}

We now turn to the last probability in \eqref{gaptype}.  Suppose $I_0$ intersects a very long enlarged $(\tth,\theta,r)$--gap $(\ol f,\ol g)$. Then by \eqref{jump}, one of $z=f$ or $g$ satisfies $|\ol z - \ol V_z| \geq \frac 13 |\ol f - \ol g| \geq \frac 13 (|y| + (\log r)^{c_4})$
and $D_\theta(V_z-z) \geq c_{14}\Phi(|\ol z - \ol V_z|)$.  Letting $K(r,y)=\lfloor \log_2(\frac 13 (|y| + (\log r)^{c_4})) \rfloor$ and defining events
\[
  M_k: \text{ for some $\tth$--start site $z$ with $d(\ol z,I_0) \leq 2^k$ we have 
    $D_\theta(V_z-z) \geq c_{14}\Phi(2^{k-1}/3)$},
\]
we see that if $2^{k-1}<|\ol f - \ol g| \leq 2^k$ then $\tau\in M_k$.  It follows using Proposition \ref{rayfluct}(ii) that 
\begin{align}\label{sidestep}
  P(\text{$I_0$ intersects a very long enlarged $(\tth,\theta,r)$--gap}) &\leq \sum_{k=K(r,y)}^\infty P(M_k) \notag\\
  &\leq c_{15} \sum_{k=K(r,y)}^\infty (\Delta(r)\log r + 2^{k+1})e^{-c_{16}\Phi(2^k)\log \Phi(2^k)} \notag\\
  &\leq c_{17} \Delta(r) (\log r) e^{-c_{18}\Phi((\log r)^{c_4})\log \Phi((\log r)^{c_4})} \notag \\
  &\leq e^{-c_{18}(\log r)(\log\log r)/2}.
\end{align}
With \eqref{gaptype} and \eqref{nogap} this shows that $P(Y_0)\geq 1/2$; with \eqref{Y0Fk} this completes the proof of the lower bound in \eqref{coaltime}.

The general outline of this proof, and in particular the use of gaps, is analogous to (\cite{BSS19}, Section 6), with added complications due to the undirected nature of paths here, which means not all start sites are sources, and backtracking may occur after coalescence.

\begin{appendices} 
\section{Appendix. Proofs--basic bad geodesic behavior}\label{basic}
We prove Propositions \ref{transfluct2} and \ref{transTincr}.

\subsection{Proof of Proposition \ref{transfluct2}} \label{3.6pf}
It is enough to consider $t$ sufficiently large (not depending on $r,\theta$.)
Let $U=(U_1^\theta,U_2^\theta)$ be the site which maximizes $D_{\theta,r}(\cdot)$ over $\Gamma_{0,r\yt}$, with ties broken arbitrarily, and let $C=D_{\theta,r}(U)\geq t$ be the corresponding maximum value.  By monotonicity of $\Phi$, $U$ must lie on the boundary of $\{u:D_{\theta,r}(u)\leq C\}$.  From symmetry we may assume $U_1^\theta\leq r/2$.  
We show there exists $W\in\Gamma_{0,r\yt}$ such that $0,U,W$ form a $\delta$-fat triangle for appropriate $\delta$.  
Fix $\kappa$ large enough so that, for $C_3$ from \eqref{powerlike},
\begin{align}\label{kappa}
  \kappa \geq &\frac{2\sqrt{d}(1+|y_\theta|)}{|y_\theta|\wedge 1},
    \quad \frac{C_3}{\kappa^{(1-\chi_2)/2}} \frac{\log(2+\kappa s)}{\log(2+s)} \leq 1 \ \text{for all } s\geq 0, \notag\\
  &\quad\text{and}\quad \kappa^{-(1-\chi)/4} \leq \left( \frac{\ep_0}{4} \wedge \frac{1}{4\sqrt{d}(1+|\yt|)C_3^{1/2}} \right)|\yt|.
\end{align}
Note that from \eqref{sigmaineq}, \eqref{powerlike}, and \eqref{kappa}, given $\delta>0$, provided we take $\kappa$ sufficiently large,
\begin{equation}\label{Xikappa}
  \left( \frac{\Xi(\kappa s)}{\kappa} \right)^2 = 
    \frac{s}{\kappa} \sigma(\kappa s)\log(2+\kappa s) \leq \frac{C_3}{\kappa^{1-\chi_2}}s\sigma(s)\log(2+\kappa s)
    \leq \delta \Xi(s)^2.
\end{equation}

There are four cases; see Figure \ref{fig2-4}.

{\it Case 1:} Suppose that $U_1^\theta \leq \Phi^{-1}(C)\leq r/2\kappa$. Then $U$ lies on the boundary of the 0-cylinder $\{u: \Phi(|u|_{\theta,\infty})\leq C\}$.  Let $\ol W$ be the first point of $\Gamma_{0,r\yt}$ after $U$ with $\ol W_1^\theta = \kappa\Phi^{-1}(C)$, and let $W$ be the first site in $\Gamma_{0,r\yt}$ after $\ol W$.  Then $\ol W,W$ must lie in the ``tube'' portion of $E_{\theta,r,C}$, meaning $|W_2^\theta|\leq C^{1/2}\Xi(W_1^\theta)$. Provided $t$ (and hence $C$ and $W_1^\theta$) are large, from \eqref{XiPhi}, \eqref{XiPhi2}, and \eqref{kappa}, the $\theta$--ratio of $\ol W$ satisfies 
\begin{equation}\label{maxslope}
  \frac{|\ol W_2^\theta|}{\ol W_1^\theta} \leq \frac{C^{1/2}\Xi(\ol W_1^\theta)}{\ol W_1^\theta} \leq \left( \frac{C}{\Phi(\ol W_1^\theta)} \right)^{1/2} 
    = \left( \frac{C}{\Phi(\kappa\Phi^{-1}(c))} \right)^{1/2} \leq \kappa^{-(1-\chi)/4} \leq \frac{\ep_0|y_\theta|}{4},
\end{equation}
and hence by \eqref{tanvsratio} and \eqref{kappa},
\begin{equation}\label{maxslope2}
  \tan \psi_{\theta \ol W} \leq \frac{2}{|y_\theta|} \kappa^{-(1-\chi)/4} \leq \frac{\ep_0}{2},
\end{equation}
so $\psi_{\theta\ol W} < \ep_0$.
Now $\min\{d_\theta(u,\Pi_{0,r\yt}): u_1^\theta \leq \Phi^{-1}(C), D_{\theta,r}(u)=C\}$ is achieved on $\mathbb{S}_\theta(C)$ so is equal to $C^{1/2}\Xi(\Phi^{-1}(C))$. Hence by \eqref{dvsdth},
\[
  d(U,\Pi_{0,r\yt}) \geq \frac{1}{d^{1/2}} C^{1/2}\Xi(\Phi^{-1}(C)).
\]
On the other hand, letting $V$ denote the closest point to $U$ in $\Pi_{0\ol W}$, we have $|V|\leq |U|\leq (1+|y_\theta|)\Phi^{-1}(C)$, so from \eqref{XiPhi}, \eqref{kappa}, and \eqref{maxslope2}, 
\begin{align*}
  d(V,\Pi_{0,r\yt}) &\leq |V|\tan \psi_{\theta \ol W} \leq \frac{2(1+|\yt|)}{|y_\theta|} \kappa^{-(1-\chi)/4}\Phi^{-1}(C) \\
  &\leq \frac{2(1+|\yt|)C_3^{1/2}}{|y_\theta|} \kappa^{-(1-\chi)/4}C^{1/2}\Xi(\Phi^{-1}(C)) \leq \frac{1}{2d^{1/2}} C^{1/2}\Xi(\Phi^{-1}(C)),
\end{align*}
so using \eqref{XiPhi} again, we must have
\begin{equation}\label{dUPi}
  d(U,\Pi_{0\ol W}) = d(U,V) \geq \frac{1}{2d^{1/2}} C^{1/2}\Xi(\Phi^{-1}(C)) \geq \frac{1}{2d^{1/2}} \Phi^{-1}(C).
\end{equation}
Now in $\theta$--coordinates $\ol W/\kappa=(\Phi^{-1}(C),\kappa^{-1}W_2^\theta)_\theta$, and from \eqref{Xikappa} and \eqref{XiPhi},
\begin{equation}\label{lineheight}
  \frac{|\ol W_2^\theta|}{\kappa} \leq \frac{C^{1/2}\Xi(\kappa\Phi^{-1}(C))}{\kappa} \leq C^{1/2}\Xi(\Phi^{-1}(C)) \leq \Phi^{-1}(C),
\end{equation}
so $\ol W/\kappa$ lies in the boundary (inside end) of the 0-cylinder of $E_{\theta,r,C}$, and hence by \eqref{cylbdry},
\begin{equation}\label{Wsize}
  \frac{|y_\theta| \wedge 1}{\sqrt{d}} \Phi^{-1}(C) \leq \frac{|\ol W|}{\kappa} \leq (1+|y_\theta|)\Phi^{-1}(C).
  \end{equation}
\begin{figure}
\includegraphics[width=16cm]{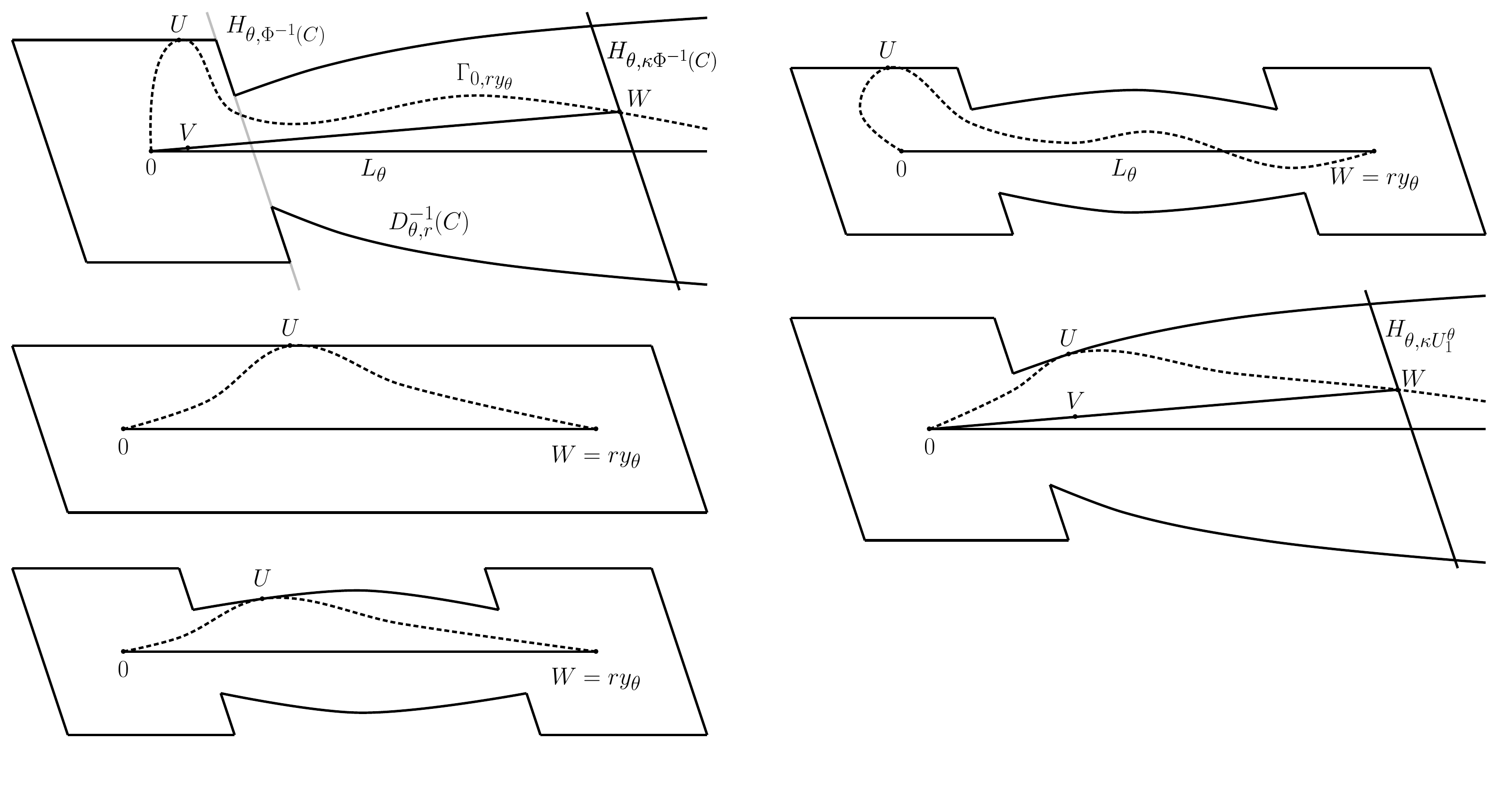}
\caption{ Illustrations for the proof of Proposition \ref{transfluct2}.  Top row: Cases 1 and 2a.  Second row: Cases 2b and 3.  Bottom row: Case 4. }
\label{fig2-4}
\end{figure}
From \eqref{dUPi} and the second inequality in \eqref{Wsize} we obtain
\[
  d(U,\Pi_{0\ol W}) \geq \frac{1}{2\kappa d^{1/2}(1+|y_\theta|)} |\ol W|.
\]
Since $|W-\ol W|\leq 1$ it is then straightforward that  
\[
  d(U,\Pi_{0,W}) \geq \frac{1}{4\kappa d^{1/2}(1+|y_\theta|)} |W|,
\]
meaning $0,U,W$ form a $\delta$--fat triangle for $\delta = (4\kappa d^{1/2}(1+|y_\theta|))^{-1}$.
Then from Proposition \ref{hmu} and Lemma \ref{gtri},
\begin{equation}\label{extradist2}
  h(U) + h(W-U) - h(W) \geq g(U) + g(W-U) - g(W) - C_{16}\sigma(|w|)\log |w| \geq c_1|W|.
\end{equation}
Let $K_{t,1}$ satisfy
\begin{equation}\label{Kt}
  2^{K_{t,1}-2} \leq \frac{\kappa (|y_\theta|\wedge 1)\Phi^{-1}(t)}{\sqrt{d}} 
    \leq \frac 32 \frac{\kappa (|y_\theta|\wedge 1)\Phi^{-1}(t)}{\sqrt{d}} \leq 2^{K_{t,1}},
\end{equation}
so from \eqref{Wsize}, since $|W-\ol W|\leq 1$,
\[
  2^{K_{t,1}-3} \leq  \frac{\kappa (|y_\theta|\wedge 1)\Phi^{-1}(t)}{2\sqrt{d}} 
    \leq \frac{\kappa (|y_\theta|\wedge 1)|\Phi^{-1}(C)}{2\sqrt{d}} \leq \frac{|\ol W|}{2} \leq |W|. 
\]
Since $U$ lies in the boundary of the 0-cylinder of $E_{\theta,r,C}$, as in \eqref{Wsize} we then have using \eqref{kappa}
\[
  |U| \leq (1+|y_\theta|)\Phi^{-1}(C) \leq \frac{2\sqrt{d}(1+|y_\theta|)}{\kappa (|y_\theta|\wedge 1)}|W| \leq |W|.
\]

Consider now the events
\begin{align*}
  A_k: &\text{ there exist $u,w\in \ZZ^d$ with } 2^{k-3}<|w|\leq 2^k, |u|\leq |w|, \notag\\
  &\qquad h(u) + h(w-u) - h(w) \geq c_1|w|, \text{ and } T(0,u)+T(u,w)=T(0,w).
\end{align*}
We have by \eqref{extradist2} that
\begin{equation}\label{sumkc1}
  P\left( \sup_{u\in\Gamma_{0,ry_\theta}} D_{\theta,r}(u) \geq t \text{ and Case 1 holds} \right) \leq 2\sum_{k=K_{t,1}}^\infty P(A_k).
\end{equation}
Here the factor of 2 accounts for the fact we assumed $U_1^\theta \leq r/2$.  Let $A_k(u,w)$ denote the event that one of the following holds:
\begin{align}\label{split3}
  h(u) - T(0,u) \geq \frac{c_1}{3}|w|,\quad h(w-u) - T(u,w)\geq \frac{c_1}{3}|w|,\quad T(0,w) - h(w) \geq \frac{c_1}{3}|w|.
\end{align}
For $u,w$ as in the event $A_k$, one of these inequalities must hold, so $A_k \subset \cup_{u,w} A_k(u,w)$, where the union is over $u,w$ as in the definition of $A_k$.  For each such $u,w$ we have from \eqref{expbd} that 
\[
  P(A_k(u,w)) \leq 12\exp\Big( -c_1|w|/3\sigma\big(\max(|u|,|w-u|,|w|)\big) \Big)
    \leq 12\exp\left( -c_1 2^{k-3}/3\sigma(2^{k+1}) \right).
\]
Summing the last bound over $u,w,k$ and using \eqref{Kt}, \eqref{sumkc1} yields
\begin{align}\label{Aksum}
  P\left( \sup_{u\in\Gamma_{0,ry_\theta}} D_{\theta,r}(u) \geq t \text{ and Case 1 holds} \right) &\leq 2\sum_{k=K_{t,1}}^\infty P(A_k) \notag\\
  &\leq \sum_{k=K_{t,1}}^\infty c_2 2^{2dk}\exp\left( -c_3 2^k/\sigma(2^k) \right) \notag\\
  &\leq c_4 \exp\Big( -c_5 \Phi^{-1}(t)/\sigma(\Phi^{-1}(t)) \Big) \notag\\
  &\leq c_4 e^{-c_6 t\log t}.
\end{align}

{\it Case 2:} Suppose $U_1^\theta\leq \Phi^{-1}(C)$ and $\Phi^{-1}(C) > r/2\kappa$. (The latter means ``the cylinder is not small relative to the tube.'') Then $U$ again lies on the boundary of the 0-cylinder in $E_{\theta,r,C}$, but this time we take $W=ry_\theta$.  We consider two subcases.

{\it Case 2a:} Suppose Case 2 holds with $\Phi^{-1}(C) \leq r/2$.  This means the tube is not completely contained inside the two cylinders; this can only occur when $\Phi^{-1}(t)\leq r/2$. Using \eqref{cylbdry} we then have
\[
  \frac{1\wedge |y_\theta|}{2\kappa\sqrt{d}} r < \frac{1\wedge |y_\theta|}{\sqrt{d}} \Phi^{-1}(C) \leq |U| 
    \leq (1 + |y_\theta|)\Phi^{-1}(C) \leq \frac{1+|y_\theta|}{2}r,
\]
and as in Case 1, using \eqref{XiPhi}, since here $\Pi_{0W}\subset L_\theta$,
\[
  d(U,\Pi_{0W}) \geq d(U,L_\theta) \geq \frac{1}{\sqrt{d}} C^{1/2}\Xi(\Phi^{-1}(C)) \geq \frac{1}{(C_3d)^{1/2}}\Phi^{-1}(C) 
    \geq c_7 |y_\theta| r = c_7 |W|.
\]
Thus $0,U,W$ form a $c_7 $-fat triangle.
Define the event
\begin{align*}
  B: &\text{ there exists $u\in \ZZ^d$ with } \frac{1\wedge |y_\theta|}{2\kappa\sqrt{d}} r < |u| \leq \frac{1+|y_\theta|}{2}r, \notag\\
  &\qquad h(u) + h(ry_\theta - u) - h(ry_\theta) \geq c_8 r, \text{ and } T(0,u)+T(u,ry_\theta)=T(0,ry_\theta).
\end{align*}
Similarly to Case 1 we have using $\Phi^{-1}(t)\leq r/2$ that
\begin{equation}\label{sumkc2a}
  P\left( \sup_{u\in\Gamma_{0,ry_\theta}} D_{\theta,r}(u) \geq t \text{ and Case 2a holds} \right) 
    \leq 2P(B) \leq c_9 r^d e^{-c_{10}r/\sigma(r)}
    \leq c_9 e^{-c_{11}t\log t}.
\end{equation}

{\it Case 2b:} Suppose Case 2 holds with $\Phi^{-1}(C) > r/2$.  This means the two cylinders contain the entire tube, and here we need only consider $\Phi^{-1}(t)\geq r/2$.  This is generally similar to Case 1, except that we do not know $|U|\leq |W|=r|y_\theta|$.  

Analogously to \eqref{cylbdry} we do have $d(U,\Pi_{0,ry_\theta}) \geq (|y_\theta|\wedge 1)\Phi^{-1}(t)/\sqrt{d} \geq c_{12}r$, so as in \eqref{extradist2},
\[
  h(U) + h(ry_\theta-U) - h(ry_\theta) \geq c_{13}|U|.
\]
We also know from \eqref{cylbdry} that $|U| \geq c_{14}\Phi^{-1}(t)$.
Instead of the events $A_k$ we use
\begin{align*}
  B_k: &\text{ there exists $u\in\ZZ^d$ with } 2^{k-1}<|u|\leq 2^k, \notag\\
  &\qquad h(u) + h(ry_\theta-u) - h(ry_\theta) \geq c_{13}|u|, \text{ and } T(0,u)+T(u,ry_\theta)=T(0,ry_\theta),
\end{align*}
and we define $K_{t,2}$ by
\[
  2^{K_{t,2}-1} < c_{14}\Phi^{-1}(t) \leq 2^{K_{t,2}},
\]
so that similarly to \eqref{Aksum},
\begin{align}\label{sumkc2b}
  P\left( \sup_{u\in\Gamma_{0,re_1}} D_r(u) \geq t \text{ and Case 2b holds} \right) &\leq 2\sum_{k=\hat K_{t,2}}^\infty P(B_k) \notag\\
  &\leq 2\sum_{k=\hat K_{t,2}}^\infty c_{15} 2^{dk} \exp\left( -c_{16}2^k/\sigma(2^k) \right) \notag\\
  &\leq c_{17}e^{-c_{18}t\log t}.
\end{align}

{\it Case 3:} Suppose $\Phi^{-1}(C)<U_1^\theta \leq r/2\kappa$, meaning that $U$ lies on the tube boundary $\{u:|u_2^\theta|=C^{1/2}\Xi(u_1^\theta)\}$ near the 0 end (but outside the 0-cylinder.)  Similarly to Case 1, let $\ol W$ be the first point of $\Gamma_{0,ry_\theta}$ after $U$ with $\ol W_1^\theta = \kappa U_1^\theta$, and let $W$ be the first site in $\Gamma_{0,ry_\theta}$ after $\ol W$.  Then using \eqref{XiPhi} and \eqref{Xikappa},
\begin{align}\label{Wmin}
  |\ol W| &\geq |\ol W_1^\theta y_\theta| - |\ol W_2^\theta| \geq \kappa |y_\theta| U_1^\theta - C^{1/2}\Xi(\kappa U_1^\theta) \notag\\
  &\geq \kappa |y_\theta| U_1^\theta - \frac{\kappa |y_\theta|}{2}C^{1/2}\Xi(U_1^\theta)
  \geq \frac{\kappa |y_\theta|}{2} U_1^\theta > \frac{\kappa |y_\theta|}{2} \Phi^{-1}(C),
\end{align}
and the $\theta$-ratio of $U$ satisfies
\begin{equation}\label{Uratio}
  \frac{|U_2^\theta|}{U_1^\theta} = \frac{C^{1/2}\Xi(U_1^\theta)}{U_1^\theta} \leq \frac{C^{1/2}}{\Phi(U_1^\theta)^{1/2}} <1
\end{equation}
so using \eqref{Wmin} without the last inequality, provided $\kappa$ is large,
\begin{equation}\label{UW}
  |U| \leq (1+|y_\theta|)U_1^\theta \leq \frac{|\ol W|}{2} \leq |W|.
\end{equation}
As in Case 1 let $V$ be the closest point to $U$ in $\Pi_{0\ol W}$, so $|V|\leq |U|$.
From \eqref{kappa} and \eqref{Uratio}, the $\theta$--ratio of $\ol W$ satisfies
\begin{equation}\label{Wratio}
  \frac{|\ol W_2^\theta|}{\ol W_1^\theta} \leq \frac{C^{1/2}\Xi(\kappa U_1^\theta)}{\kappa U_1^\theta} 
    \leq \kappa^{-(1-\chi)/4} \frac{C^{1/2}\Xi(U_1^\theta)}{U_1^\theta} 
    \leq \kappa^{-(1-\chi)/4} < \frac{|y_\theta|}{2},
\end{equation}
so by \eqref{tanvsratio}, \eqref{Xikappa}, \eqref{UW}, and the first inequality in \eqref{Wratio},
\begin{equation}\label{dVL}
  d(V,L_\theta) \leq |V| \tan \psi_{\theta\ol W} \leq |U| \tan \psi_{\theta\ol W} 
    \leq \frac{2}{|y_\theta|}\frac{C^{1/2}\Xi(\kappa U_1^\theta)|U|}{\kappa U_1^\theta}
     \leq \frac{1}{2\sqrt{d}} C^{1/2}\Xi(U_1^\theta).
\end{equation}
On the other hand, we have from \eqref{dvsdth} that
\[
  d_\theta(U,L_\theta) = C^{1/2}\Xi(U_1^\theta) \quad\text{so}\quad d(U,L_\theta) \geq \frac{1}{\sqrt{d}} C^{1/2}\Xi(U_1^\theta),
\]
which with \eqref{dVL} shows
\[
  d(U,\Pi_{0\ol W}) = |U-V| \geq \frac{1}{2\sqrt{d}} C^{1/2}\Xi(U_1^\theta).
\]

Thus $0,U,\ol W$ form a $\delta$-fat triangle with 
\[
  \delta = \frac{C^{1/2}\Xi(U_1^\theta)}{2\sqrt{d}|\ol W|}.
\]
From \eqref{Wratio} and \eqref{magsize} we have
\begin{equation}\label{WvsW1}
    \frac 12 |y_\theta|\kappa U_1^\theta \leq \frac 12 |y_\theta|W_1^\theta \leq |\ol W| \leq 2|y_\theta|W_1^\theta,
\end{equation}
so from \eqref{XiPhi}, provided $\kappa$ is large this $\delta$ satisfies
\[
  \delta \leq \frac{\Phi(U_1^\theta)^{1/2}\Xi(U_1^\theta)}{\kappa|y_\theta|\sqrt{d}U_1^\theta} \leq \frac{1}{\kappa|y_\theta|\sqrt{d}} < 1.
\]
Hence from Lemma \ref{gtri}, \eqref{Wratio}, and \eqref{XiPhi} we have
\begin{align*}
  g(U) + g(\ol W-U) - g(\ol W) &\geq C_{25}\frac{C\Xi(\ol W_1^\theta/\kappa)^2}{4d|\ol W|}  \notag\\
  &\geq \frac{C_{25}C}{4d\kappa} \frac{\ol W_1^\theta}{|\ol W|}\ \sigma\left( \frac{\ol W_1^\theta}{\kappa} \right) 
    \log\left( 2 + \frac{\ol W_1^\theta}{\kappa} \right) \\
   &\geq c_{19}C\sigma(|\ol W|) \log(2+|\ol W|),
\end{align*}
and since $|W-\ol W|\leq 1$, the same holds for $W$ in place of $\ol W$, with a smaller $c_{19}$.
Defining the events
\begin{align*}
  F_k: &\text{ there exist $u,w\in \ZZ^d$ with } 2^{k-1}<|w|\leq 2^k, |u|\leq |w|, \notag\\
  &\qquad g(u) + g(w-u) - g(w) \geq c_{19}t\sigma(|w|) \log(2+|w|), \text{ and } T(0,u)+T(u,w)=T(0,w),
\end{align*}
and defining $K_{t,3}$ by
\[
  2^{K_{t,3}-1} < \frac{\kappa |y_\theta|}{2}\Phi^{-1}(t) \leq 2^{K_{t,3}},
\]
in view of \eqref{Wmin} and \eqref{UW} we have similarly to \eqref{Aksum} and \eqref{sumkc2b}, provided $t$ is large:
\begin{align}\label{sumkc3}
  P\left( \sup_{u\in\Gamma_{0,re_1}} D_r(u) \geq t \text{ and Case 3 holds} \right) &\leq 2\sum_{k=K_{t,3}}^\infty P(F_k) \notag\\
  &\leq 2\sum_{k=K_{t,3}}^\infty c_{20} 2^{2dk} \exp\left( -c_{21}tk \right) \notag\\
  &\leq c_{22}e^{-c_{23}t\log t}.  
\end{align}

{\it Case 4:} Suppose $\max(\Phi^{-1}(C),r/2\kappa) <U_1^\theta \leq r/2$, meaning that $U$ lies on the tube boundary but not near an end.  This can only occur when $\Phi^{-1}(t)\leq r/2$, and this time we take $W=ry_\theta$. Since $\Pi_{0W}\subset L_\theta$, we have using \eqref{dvsdth} that
\[
  d(U,\Pi_{0W})  \geq d(U,L_\theta) \geq \frac{d_\theta(U,L_\theta)}{\sqrt{d}} = \frac{|U_2^\theta|}{\sqrt{d}} 
    = \frac{C^{1/2}\Xi(U_1^\theta)}{\sqrt{d}} \geq \frac{C^{1/2}\Xi(r/2\kappa)}{\sqrt{d}}
\]
so $0,U,ry_\theta$ form a $\delta$-fat triangle with 
\[
  \delta = \frac{C^{1/2}\Xi(r/2\kappa)}{\sqrt{d} |y_\theta| r}.
\]
From the definition \eqref{Ddef} we have
\[
  \delta^2 r \geq \frac{C\Xi(r/2\kappa)^2}{d|y_\theta|^2 r} 
    = \frac{C}{2\kappa d|y_\theta|^2} \sigma\left( \frac{r}{2\kappa} \right)\log\left(2 + \frac{r}{2\kappa} \right).
\]
Similarly to Case 3 we have $\delta<1$. From \eqref{magsize} and \eqref{XiPhi} we obtain
\[
  |U| \leq |y_\theta|U_1^\theta + C^{1/2}\Xi(U_1^\theta) \leq |y_\theta|U_1^\theta + \Phi(U_1^\theta)^{1/2}\Xi(U_1^\theta)
    \leq (1+|y_\theta|) U_1^\theta \leq (1+|y_\theta|)r.
\]
Defining the event 
\begin{align*}
  F: &\text{ there exists $u\in \ZZ^d$ with } |u| \leq (1+|y_\theta|)r, T(0,u) +T(u,ry_\theta)=T(0,ry_\theta), \notag\\
  &\qquad \text{ and }  g(u) + g(ry_\theta - u) - g(ry_\theta) \geq 
    C_{25}\frac{t}{2\kappa d|y_\theta|^2} \sigma\left( \frac{r}{2\kappa} \right)\log\left(2 + \frac{r}{2\kappa} \right),  
\end{align*}
the rest is similar to Case 2a, and we obtain
\begin{align}\label{sumkc4}
  P\left( \sup_{u\in\Gamma_{0,ry_\theta}} D_{\theta,r}(u) \geq t \text{ and Case 4 holds} \right) \leq 2P(F) \leq c_{24}e^{-c_{25}t\log t}.
\end{align} 
Putting the 4 cases together, \eqref{Aksum}, \eqref{sumkc2a}, \eqref{sumkc2b}, \eqref{sumkc3}, and \eqref{sumkc4} complete the proof.

\subsection{Proof of Proposition \ref{transTincr}} \label{3.7pf}
Let 
\[
  \theta = \frac{u}{|u|},\quad \ell=c_1\Delta^{-1}(|u-v|), \quad\text{and}\quad t=c_2\lambda\log |u-v|,
\]
with $c_1,c_2$ to be specified; note $\log \ell$ and $\log |u-v|$ are of the same order.  Provided $C_{31}$ is taken small (depending on $c_1$), we have $\ell \leq |u|/2|\yt| = u_1^\theta/2=g(u)/2$.  
We consider first the case of ``moderately large $\lambda$'':
\begin{equation}\label{reglam}
  \lambda\sigma(\Delta^{-1}(|u-v|))\log |u-v| \leq 2h(u-v).
\end{equation}
We intersect a tube-and-cyliners region with a fattened hyperplane to get
\[
  \Upsilon_\ell = H_{\theta,u_1^\theta - \ell}^{\rm rfat} \cap E_{\theta,u_1^\theta,t}
\]
so that every path crossing $H_{\theta,u_1^\theta - \ell}$ inside $E_{\theta,u_1^\theta,t}$ must include a site in $\Upsilon_\ell$.
Therefore either $\Gamma_{u0} \not\subset E_{\theta,u_1^\theta,t}$ or there exists a first site $X$ of $\Gamma_{u0}$ in $\Upsilon_\ell$, in which case
\[
  T(u,0) = T(u,X) + T(X,0) \quad\text{so}\quad T(v,0) - T(u,0) \leq T(v,X)- T(u,X). 
\]
It follows using Proposition \ref{transfluct2} that
\begin{align}\label{Upspoints}
  P\Big(T&(v,0) - T(u,0) \geq \lambda\sigma(\Delta^{-1}(|u-v|))\log |u-v|\Big) \notag\\
  &\leq P\left( \sup_{x\in \Gamma_{u0}} D_{\theta,u_1^\theta}(x) > t \right)
    + \sum_{x\in\Upsilon_\ell\cap\ZZ^d} P\Big( T(v,x) - T(u,x) \geq \lambda\sigma(\Delta^{-1}(|u-v|))\log |u-v| \Big) \notag\\
  &\leq C_{26}e^{-C_{27}t\log t} + c_3 \ell^{d-1} \max_{x\in\Upsilon_\ell\cap\ZZ^d} 
    P\Big( T(v,x) - T(u,x) \geq \lambda\sigma(\Delta^{-1}(|u-v|))\log |u-v| \Big).
\end{align}

\begin{figure}
\includegraphics[width=14cm]{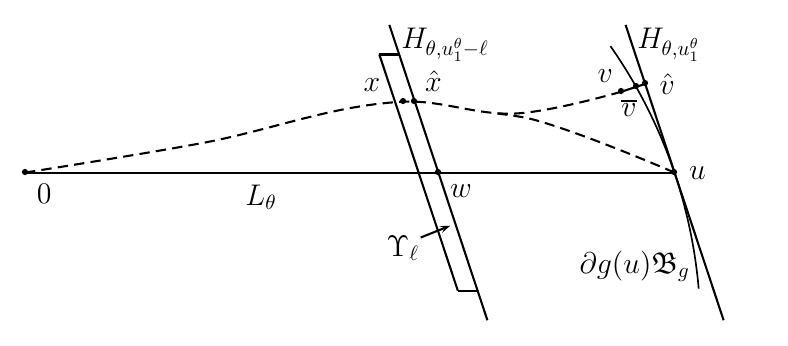}
\caption{ Illustration for the proof of Proposition \ref{transTincr}. $T(0,u)$ and $T(0,v)$ are likely to be close because $\Gamma_{0u}$ has the option to pass through the same site $x\in\Upsilon_\ell$ that $\Gamma_{0v}$ passes through. }
\label{fig2-6}
\end{figure}

Let $x\in\Upsilon_\ell\cap\ZZ^d$. We claim that
\begin{equation}\label{ggapclaim}
  g(v-x) - g(u-x) \leq c_4 (t\log |u-v|)^{1/2}\sigma(\Delta^{-1}(|u-v|)).
\end{equation}
To prove this, let $w=(u_1^\theta-\ell)\yt \in \Pi_{0u}$. We first consider $x$ replaced by its $\theta$--projection $\hat x= (w_1^\theta,x_2^\theta)$ into $H_{\theta,u_1^\theta-\ell}$. From the definition of $\Upsilon_\ell$,
\begin{equation}\label{xhatx}
  g(x-\hat x) = |x_1^\theta - w_1^\theta| \leq \mu\sqrt{d}.
\end{equation}
Let $\hat v$ be the closest point to $v$ in $H_{\theta,u_1^\theta}$.   
Since $|g(v)- g(u)| \leq 4\mu d$, provided $|u|$ is large (so by \eqref{uvassump} $\psi_{uv}$ is small) there exists a point $\ol v$ on the line through $v$ and $\hat v$ satisfying $g(\ol v) = g(u)$ and $g(v-\ol v)\leq 5\mu d$.
In order to bound $g(v-\hat v)$, we first observe that
\begin{equation}\label{uvratio}
  |u-\hat v| \leq |u-v| = \Delta(\ell/c_1) \quad\text{and hence}\quad
  \frac{|u-\hat v|}{g(u)} \leq \frac{\Delta(\ell/c_1)}{\ell},
\end{equation}
and the last fraction can be made small by taking $c_1$ large in the definition of $\ell$ (so $\ell$ itself is large), and it then follows using A3' and \eqref{gsize} that
\begin{align}\label{vhatv}
  g(v-\hat v) &\leq \mu\sqrt{d} |\ol v-\hat v| + g(v-\ol v) = \mu\sqrt{d}\, d(\ol v,H_{\theta,u_1^\theta}) + 5\mu d 
     \leq c_5 \frac{|u-\hat v|^2}{u_1^\theta} + 5\mu d \notag\\
  &\leq c_5 \frac{\Delta(\ell/c_1)^2}{\ell} + 5\mu d \leq c_6 \sigma(\Delta^{-1}(|u-v|)).
\end{align}
In view of \eqref{xhatx} and \eqref{vhatv}, to prove \eqref{ggapclaim} it remains to bound $g(\hat v - \hat x) - g(u-\hat x)$, for which we use Lemma \ref{gdiff1} with the origin shifted to $\hat x$.  First observe that provided $c_1$ (and hence $\ell$) is large, using \eqref{reglam} we have 
\begin{equation}\label{tbound}
  t = c_2\lambda\log |u-v| \leq c_7\frac{\Delta(\ell)}{\sigma(\ell)} = c_7 \left( \frac{\ell}{\sigma(\ell)} \right)^{1/2} < \Phi(\ell).
\end{equation}
Therefore $\ell>\Phi^{-1}(t)$, meaning that $w,x$ lie in the tube part of the tube-and-cylinders region $E_{\theta,u_1^\theta,t}$. Therefore
\begin{equation}\label{shortx}
  |w-\hat x| = |x_2^\theta| \leq t^{1/2}\Xi((x-u)_1^\theta) = t^{1/2}\Xi(\ell),
\end{equation}
and hence from the first inequality in \eqref{tbound}, provided $c_1$ (and therefore $\ell$) is large,
\[
  \frac{|w-\hat x|}{\ell} \leq \frac{t^{1/2}\Xi(\ell)}{\ell} \leq \left( \frac{2c_7\Delta(\ell) \log \ell}{\ell} \right)^{1/2} < \ep_7,
\]
for $\ep_7$ from Lemma  \ref{gdiff1}. From \eqref{uvratio} we also have
\[
  \frac{|u-\hat v|}{\ell} \leq \frac{\Delta(\ell/c_1)}{\ell} < \ep_7,
\]
so Lemma \ref{gdiff1} applies, giving 
\begin{equation}\label{shiftedg}
  g(\hat v - \hat x) - g(u-\hat x) \leq C_{29}\left( \frac{|\hat v - u|\, |w-\hat x|}{\ell} + \frac{|\hat v - u|^2}{\ell}\right).
\end{equation}
Using again $|\hat v - u| \leq \Delta(\ell/c_1)$ along with \eqref{shortx} and \eqref{shiftedg} lets us conclude
\[
  g(\hat v - \hat x) - g(u-\hat x) \leq c_8(t\log \ell)^{1/2}\sigma(\ell).
\]
With \eqref{xhatx} and \eqref{vhatv} this proves \eqref{ggapclaim}.

From \eqref{ggapclaim} and Proposition \ref{hmu}, provided $\lambda$ is large,
\begin{align}\label{hdiff}
  h(v-x) - h(u-x) &\leq \left(c_9(t\log\ell)^{1/2}+ c_{10}\log \ell\right)\sigma(\ell) \notag\\
  &\leq c_{11}\lambda^{1/2}\sigma(\Delta^{-1}(|u-v|))\log |u-v| \notag\\
  &\leq \frac \lambda 2 \sigma(\Delta^{-1}(|u-v|))\log |u-v|.
\end{align}
It follows using \eqref{expbd} that the probability in \eqref{Upspoints} satisfies
\begin{align}\label{Tgap}
  P\Big( &T(v,x) - T(u,x) \geq \lambda\sigma(\Delta^{-1}(|u-v|))\log |u-v| \Big) \notag\\
  &\leq P\left( T(v,x) - h(v-x) \geq \frac\lambda 4 \sigma(\Delta^{-1}(|u-v|))\log |u-v| \right) \notag\\
  &\qquad + P\left( T(u,x) - h(u-x) \leq -\frac\lambda 4 \sigma(\Delta^{-1}(|u-v|))\log |u-v| \right) \notag\\
  &\leq 4\exp\left( -\frac{\lambda}{4}\, \frac{\sigma(\Delta^{-1}(|u-v|))}{\sigma(|v-x|)} \log |u-v| \right) 
    +  4\exp\left( -\frac{\lambda}{4}\, \frac{\sigma(\Delta^{-1}(|u-v|))}{\sigma(|u-x|)} \log |u-v| \right).
\end{align} 
From \eqref{tbound} (without the last inequality) and \eqref{XiPhi} we get $t\Xi(\ell)^2/\ell^2 \leq c_{12}(\sigma(\ell)/\ell)^{1/2}\log \ell\leq (|\yt|/2)^2$.
Using this with \eqref{gsize}, \eqref{closeproj}, \eqref{xhatx}, and \eqref{shortx} we get
\begin{align}\label{vx}
  |v-x| &\leq |v-u| + |u-w| + |w-\hat x| + |\hat x -x| \notag\\
  &\leq \Delta(\ell/c_1) + \ell |\yt| + |x_2^\theta| + d \notag\\
  &\leq 2|\yt|\ell
\end{align}
and similarly for $|u-x|$. Hence the right side of \eqref{Tgap} is bounded above by $24e^{-c_{13}\lambda\log |u-v|}$, which with \eqref{Upspoints} yields that for all $\lambda \geq c_{14}/2$ satisfying \eqref{reglam},
\begin{align}\label{Tgap2}
  P\Big(T&(v,0) - T(u,0) \geq \lambda\sigma(\Delta^{-1}(|u-v|))\log |u-v|\Big) \notag\\
  &\leq C_{26}e^{-C_{27}t\log t} + 24c_3 \ell^{d-1}e^{-c_{13}\lambda\log |u-v|} \notag\\
  &\leq c_{15}e^{-c_{16}\lambda\log |u-v|}.
\end{align}

It remains to consider ``very large $\lambda$,'' meaning \eqref{reglam} does not hold:
\begin{equation}\label{biglam}
  \lambda\sigma(\Delta^{-1}(|u-v|))\log |u-v| > 2h(u-v).
\end{equation}
Here we have using \eqref{expbd} that
\begin{align}\label{Tgap3}
  P\Big(&T(v,0) - T(u,0) \geq \lambda\sigma(\Delta^{-1}(|u-v|))\log |u-v|\Big) \notag\\
  &\leq P\left(T(u,v) \geq h(v-u) + \frac \lambda 2 \sigma(\Delta^{-1}(|u-v|))\log |u-v|\right) \notag\\
  &\leq 4\exp\left( -\frac{\lambda}{2}\, \frac{\sigma(\Delta^{-1}(|u-v|))}{\sigma(|u-v|)} \log |u-v|\right) \notag\\
  &\leq 4e^{-c_{17}\lambda\log |u-v|},
\end{align}
where the last inequality is similar to the bound for \eqref{Tgap}. 
With \eqref{Upspoints} and \eqref{Tgap2} this proves \eqref{Tchange}.

\section{Appendix. Proofs---Special types of bad geodesic behavior} \label{spec}

\subsection{Proof of Lemma \ref{fastseg}}  
We need a slightly modified definition to take into account that $\Gamma_{vw}$ might not be a slab geodesic, i.e.~ we might not have $x_{\theta,0}'(\Gamma_{vw})=v$ and $x_{\theta,R}''(\Gamma_{vw}) = w$:
\[
  \hat x_{\theta,i\ell}''(\Gamma) = \begin{cases} v &\text{if } i=0,\\
    x_{\theta,i\ell}''(\Gamma) &\text{if } 1\leq i \leq k-1\\ w &\text{if } i=k, \end{cases} 
    \quad\text{and}\quad
    \hat x_{\theta,i\ell}'(\Gamma) = \begin{cases} v &\text{if } i=0,\\ x_{\theta,i\ell}'(\Gamma) &\text{if } 1\leq i \leq k-1\\ w &\text{if } i=k, \end{cases}
\]
and define
\[
  Y_i(\Gamma) = \frac{T(\hat x_{\theta,(i-1)\ell}''(\Gamma),\hat x_{\theta,i\ell}''(\Gamma)) - k^{-1}ET(0,R\yt)}{\sigma(\ell)}, 
\]
\[
  S_k(\Gamma) 
    = \sum_{i=1}^k Y_i(\Gamma) = \frac{T(v,w) - ET(0,R\yt)}{\sigma(\ell)},
\]
\[
  \tilde Y_i(\Gamma) = \frac{T(\hat x_{\theta,(i-1)\ell}''(\Gamma),\hat x_{\theta,i\ell}''(\Gamma)) 
    - h(\hat x_{\theta,i\ell}''(\Gamma) - \hat x_{\theta,(i-1)\ell}''(\Gamma))}{\sigma(\ell)}.
\]
For technical convenience we will assume $\ell\yt$ is a lattice point; the adjustments when this is false are minor.

It follows from \eqref{powerlike} and \eqref{Rell} that provided $C_{38}$ is large enough in \eqref{Rell}, we have  
\begin{equation}\label{ellvsR1}
  \frac{\sigma(R)}{k\sigma(\ell)} \leq \frac{C_3}{k^{1-\chi_2}} \leq \frac{\eta}{16c_1\log R},
\end{equation}
with $c_1$ chosen so that $h(R\yt) \leq g(R\yt) + c_1\sigma(R)\log R$, from Proposition \ref{hmu}. With this we obtain that provided $C_{39}$ is small (so $\ell$ is large), for some $c_1$, for all $1\leq i\leq k$,
\begin{align}\label{hbig}
  h\Big(\hat x_{\theta,i\ell}''(\Gamma_{vw}) - \hat x_{\theta,(i-1)\ell}''(\Gamma_{vw})\Big) 
    &\geq g\Big(\hat x_{\theta,i\ell}''(\Gamma_{vw}) - \hat x_{\theta,(i-1)\ell}'(\Gamma_{vw})\Big) - c_2 \notag\\
  &= \frac 1k g\Big(k(\hat x_{\theta,i\ell}''(\Gamma_{vw}) - \hat x_{\theta,(i-1)\ell}'(\Gamma_{vw}))\Big) - c_2 \notag\\
  &\geq \frac 1k g(R\yt) - c_2 \notag\\
  &\geq \frac 1k \Big[ h(R\yt) - c_1\sigma(R) \log R \Big] - c_2\notag\\
  &\geq \frac 1k ET(0,R\yt) - \frac \eta 8 \sigma(\ell),
\end{align}
where the second inequality follows from $k(\hat x_{\theta,i\ell}''(\Gamma_{vw}) - \hat x_{\theta,(i-1)\ell}'(\Gamma_{vw})) \in H_{\theta,R}^+$.  Therefore
\begin{equation}\label{Ytilde}
  Y_i(\Gamma_{vw}) \geq \tilde Y_i(\Gamma_{vw}) - \frac \eta 8.
\end{equation}
The key point is that
\[
  S_k(\Gamma_{vw}) \geq (k-N_\theta(\Gamma_{vw}))\frac{\eta}{8} + N_\theta(\Gamma_{vw})\min_{i\leq k} Y_i(\Gamma_{vw})
    = k \frac\eta8+ N_\theta(\Gamma_{vw})\left(\min_{i\leq k} Y_i(\Gamma_{vw}) - \frac{\eta}{8}\right),
\]
which yields
\begin{align}\label{fastseg3}
  P&\left(N_\theta(\Gamma_{vw})\leq \frac{\eta}{32} k^{1-\lambda} \right) \notag\\
  &\leq P\left( S_k(\Gamma_{vw}) \geq \frac{\eta}{16}k \right) 
    + P\left( \min_{i\leq k} Y_i(\Gamma_{vw}) \leq \frac{\eta}{8} -2k^\lambda \right)   \notag\\
  &\leq P\left( T(v,w) - ET(0,R\yt) \geq \frac{k\eta\sigma(\ell)}{16} \right) 
    + P\left( \min_{i\leq k} Y_i(\Gamma_{vw}) \leq \frac{\eta}{8} -2k^\lambda \right).
\end{align}
To control the next-to-last probability we need to bound $h(w-v)-h(R\yt)$.  Let $\hat w,\hat v$ be the $\theta$--projections of $w,v$ into $H_{\theta,R}$ and $H_{\theta,0}$, respectively, so by \eqref{closeproj} we have 
$|w-\hat w|\leq d$ and $|v-\hat v| \leq d$. From \eqref{powerlike} and \eqref{wclose} we have
\begin{equation}\label{curvgap}
  \max\left( \frac{|v|^2}{R},\frac{|w - R\yt|^2}{R} \right) \leq C_{40}^2\eta k^{1-\chi_2}\sigma(R) \leq C_3C_{40}^2\eta k\sigma(\ell),
\end{equation}
and hence provided $C_{39}$ is small (so $\ell$ is large),
\begin{align*}
  \left( \frac{|(\hat w -\hat v) - R\yt|}{R} \right)^2 &\leq 4\left( \frac{|w - R\yt|}{R} \right)^2 + 4\left( \frac{|v|}{R} \right)^2 
    + 4\left( \frac{|w-\hat w|}{R} \right)^2 + 4\left( \frac{|v-\hat v|}{R} \right)^2\\
  &\leq 8C_3C_{40}^2\eta \frac{\sigma(\ell)}{\ell} + \frac{8d^2}{R^2} \leq 9C_3C_{40}^2\eta \frac{\sigma(\ell)}{\ell} < \ep_0^2.
\end{align*}
Therefore A3 applies and, provided we take $C_{40}$ small in \eqref{wclose}, we get
\begin{equation}\label{gwvhat}
  g(\hat w-\hat v) \leq g(R\yt) + c_3\frac{|(\hat w-\hat v) - R\yt|^2}{R} \leq g(R\yt) + 9c_3C_3C_{40}^2\eta k\sigma(\ell) 
    \leq g(R\yt) + \frac{\eta k\sigma(\ell)}{64}.
\end{equation}
From \eqref{curvgap},
\begin{equation}\label{wsize}
  |w-v| \leq |w-R\yt| + R|\yt| + |v| \leq 2\left( \frac{C_3C_{40}^2\eta \sigma(\ell)}{\ell} \right)^{1/2} R + R|\yt| \leq 2|\yt|R.
\end{equation}
With Proposition \ref{hmu}, \eqref{ellvsR1}, and \eqref{gwvhat} this gives the desired bound
\begin{align*}
  h(w-v) &\leq g(w-v) + C_{16}\sigma(2|\yt|R)\log(2|\yt|R) \leq g(\hat w-\hat v) + 2\mu\sqrt{d} + c_4 \sigma(R)\log R \\
  &\leq g(R\yt) + \frac{\eta k\sigma(\ell)}{64} + 2\mu\sqrt{d} + \frac{\eta k\sigma(\ell)}{64} \leq h(R\yt) + \frac{3\eta k\sigma(\ell)}{64}.
\end{align*}
Therefore in \eqref{fastseg3}, using \eqref{powerlike}, \eqref{expbd}, and \eqref{wsize} we have
\begin{align}\label{totalT}
  P\left( T(v,w) - ET(0,R\yt) \geq \frac{k\eta\sigma(\ell)}{16} \right) &\leq P\left( T(v,w) - ET(v,w) \geq \frac{\eta k\sigma(\ell)}{64} \right) \notag\\
  &\leq 4\exp\left( -\frac{\eta k\sigma(\ell)}{64\sigma(|w-v|)} \right) \notag\\
  &\leq 4\exp\left( -c_5\eta k^{1-\chi_2} \right).
\end{align}

It remains to bound the last probability in \eqref{fastseg3}. To do this, for each $i$ we split out the (unlikely) case in which the increment of the $i$th $\ell$--segment of $\Gamma_{vw}$ is so far from direction $\theta$ that its $g$--value exceds $3\ell$.  Write $\Lambda_i$ for $H_{\theta,i\ell}^{\rm fat}$,
which must contain $x_{\theta,i\ell}''(\Gamma_{vw})$.
For $t=k^{1-\chi_2}$ we thereby have using Proposition \ref{transfluct2} and \eqref{Ytilde} that
\begin{align}\label{veryfast}
  P&\left( \min_{i\leq k} Y_i(\Gamma_{vw}) \leq \frac{\eta}{8} -2k^\lambda \right) \notag\\
  &\leq P(\Gamma_{vw}\not\subset E_{\alpha,w_1^\alpha,t}) \notag\\
  &\qquad +\sum_{i=1}^k P\left( \Gamma_{vw}\subset E_{\alpha,w_1^\alpha,t}, 
    g\Big(\hat x_{\theta,i\ell}''(\Gamma_{vw}) - \hat x_{\theta,(i-1)\ell}''(\Gamma_{vw})\Big) 
    > 3\ell, Y_i(\Gamma_{vw}) \leq \frac{\eta}{8} -2k^\lambda \right) \notag\\
  &\qquad + \sum_{i=1}^k P\left( \Gamma_{vw}\subset E_{\alpha,w_1^\alpha,t}, 
    g\Big(\hat x_{\theta,i\ell}''(\Gamma_{vw}) - \hat x_{\theta,(i-1)\ell}''(\Gamma_{vw})\Big) 
    \leq 3\ell, \tilde Y_i(\Gamma_{vw}) \leq \frac{\eta}{4} -2k^\lambda \right) \notag\\
  &\leq C_{26}e^{-C_{27}t} + \sum_{i=1}^k\ \sum_{a \in \Lambda_{i-1} \cap E_{\alpha,w_1^\alpha,t} \cap \ZZ^d}\ 
    \sum_{{b \in \Lambda_i \cap E_{\alpha,w_1^\alpha,t} \cap \ZZ^d}\atop{g(b-a)>3\ell}}
    P\left( T(a,b) - \frac 1k ET(0,R\yt)\leq -k^\lambda \sigma(\ell) \right) \notag\\
  &\qquad + \sum_{i=1}^k\ \sum_{a \in \Lambda_{i-1} \cap E_{\alpha,w_1^\alpha,t} \cap \ZZ^d}\ 
    \sum_{{b \in \Lambda_i \cap E_{\alpha,w_1^\alpha,t} \cap \ZZ^d}\atop{g(b-a)\leq 3\ell}}
    P\left( T(a,b) - ET(a,b)\leq -k^\lambda \sigma(\ell) \right).
\end{align}
For the last probability we have from \eqref{expbd}
\[
  P\left( T(a,b) - ET(a,b)\leq -k^\lambda \sigma(\ell) \right) \leq 4\exp\left( -k^\lambda \frac{\sigma(\ell)}{\sigma(|b-a|)} \right)
    \leq 4e^{-c_6k^\lambda}.
\]
For the next-to-last probability in \eqref{veryfast} we have from subadditivity and Proposition \ref{hmu} that for all $a,b$ in the double sum, 
\[
  \frac 1k ET(0,R\yt) \leq ET(0,\ell\yt) \leq \ell + c_7\sigma(\ell) \log \ell \leq h(b-a) - \frac 12 g(b-a)
\]
so from \eqref{expbd} again,
\begin{align}\label{bigba}
  P\left( T(a,b) - \frac 1k ET(0,R\yt)\leq -k^\lambda \sigma(\ell) \right)
    &\leq P\left( T(a,b) - ET(a,b) \leq -\frac 12 g(b-a) \right) \notag\\
  &\leq 4e^{-g(b-a)/2\sigma(|b-a|)} \notag\\
  &\leq 4e^{-c_8\ell/\sigma(\ell)} \notag\\
  &\leq 4e^{-c_9k^\lambda}.
\end{align}
Here the last inequality follows from the fact that for large $\ell$, by \eqref{powerlike} and \eqref{Rell},
\[
  \frac{\ell}{\sigma(\ell)} \geq \left( \frac Rk \right)^{1-\chi_2} \geq c_{10}k^{1-\chi_2}.
\]
The number of $a$ or $b$ in the sums in \eqref{veryfast} is at most of order $R^d$, so provided $C_{38}$ is large enough in \eqref{Rell}, the right side of \eqref{veryfast} is bounded above by
\[
  C_{26}e^{-C_{27}k^{1-\chi_2}} + c_{11}kR^{2d} e^{-c_6k^\lambda} + c_{12}kR^{2d} e^{-c_8k^{1-\chi_2}}
    \leq c_{13}e^{-c_{14}k^\lambda}.
\]
With \eqref{fastseg3} and \eqref{totalT} this proves \eqref{fastseg2}.

\subsection{Proof of Lemma \ref{badbehavG5}}
We must deal with the inconvenience that we may have $\theta\neq\theta_0$.  
Suppose $\tau\in G_1$, and let $\Gamma$ be a $\theta_0$--slab geodesic from $x\in B_{\theta_0,{\rm home}}^{\rm rfat}$ to $z\in H_{\theta,2R}^{\rm fat}$ passing through $y=x_R''(\Gamma)\in B_{\theta_0,{\rm cross}}^{\rm fat}$, and let $u$ be the first site of $\Gamma$ with $u\notin Q_{R,n,\theta}$.  See Figure \ref{fig4-5}, particularly the upper right diagram. 
Then $\hat\Gamma_R:= \Gamma[x,y]$ is the pre--$H_{\theta_0,R}$ segment of $\Gamma$, and we recall that $\ol y$ is the center of $B_{\theta_0,{\rm cross}}$.  Using $\psi_{\theta\theta_0}\leq \ep_{\min}$ we obtain straightforwardly that 
\begin{equation}\label{ybaru}
  |\ol y| \leq 2|\ytz|R, \quad |u| \leq 3|\ytz|R.
\end{equation}
We now consider four cases.

{\it Case 1.} $u\in \Gamma[x,y]$, so that $|(u-x)_1^{\theta_0}| \leq R+2\sqrt{d}\mu$. Here we find a lower bound for $d(u,\Pi_{xy})$, establishing the presence of a $\delta$--fat triangle. 
Let $\varphi=(y-x)/|y-x|$.  Define intersection points 
\[
  v = \Pi_{xy}^\infty \cap H_{\theta_0,u_1^{\theta_0}},\ \ \hat v = L_\theta \cap H_{\theta_0,u_1^{\theta_0}},\ \ 
    w = \Pi_{xy}^\infty \cap H_{\theta,u_1^{\theta}},\ \ \hat w = L_\theta \cap H_{\theta,u_1^{\theta}},
\]
\[
    q = \Pi_{xy}^\infty \cap H_{\varphi,u_1^{\varphi}},\ \ \hat q = L_\theta \cap H_{\varphi,u_1^{\varphi}}.
\]
and note that all 3 hyperplanes referenced here pass through $u$, and $L_\theta=\Pi_{0\ol y}^\infty$. See Figure \ref{fig4-5}, upper left. By \eqref{dvsdth},
\begin{equation}\label{ddphi}
  \sqrt{d}d(u,\Pi_{xy}^\infty) \geq d_\varphi(u,\Pi_{xy}^\infty) = |u-q|,
\end{equation}
and from the definition of $u$,
\begin{equation}\label{udist}
  2n^{\beta_1}\Delta(R) \leq d_\theta(u,L_\theta) = |u-\hat w| \leq 3\sqrt{d-1}n^{\beta_1}\Delta(R).
\end{equation}
Further, $x,y$ being in fattened $\theta_0$--blocks centered on $L_\theta$ tells us that, first, provided $R$ is large the $\theta$--ratio of $y-x$ is at most $2\sqrt{d-1}n^{\beta_1}\Delta(R)/|\ytz|R<\ep_{\min}|\yt|/4$, so by \eqref{tanvsratio} we have $\psi_{\varphi\theta} <\ep_{\min}/2$ and hence also $\psi_{\varphi\theta_0}<\ep_{\min}$.
Second, since $v\in\Pi_{xy}$,
\begin{equation}\label{vvhat}
  |v-\hat v| = d_{\theta_0}(v,L_\theta) \leq \max(d_{\theta_0}(x,L_\theta),d_{\theta_0}(y,L_\theta)) \leq \sqrt{d-1}n^{-\beta_0}\Delta(R).
\end{equation}
Now by \eqref{linegap2},
\[
  |u-q| \geq |u-v| - |v-q| \geq |u-v| - C_{22}\psi_{\varphi\theta_0}|u-q| 
\]
so
\begin{equation}\label{uquv}
  |u-q| \geq \frac{|u-v|}{1+C_{22}\psi_{\varphi\theta_0}}.
\end{equation}
Combining \eqref{linegap2}, \eqref{ddphi}, \eqref{udist}, \eqref{vvhat}, \eqref{uquv} we get
\begin{align}\label{utoline}
  (1+C_{22}\psi_{\varphi\theta_0}) \sqrt{d}d(u,\Pi_{xy}^\infty) &\geq (1+C_{22}\psi_{\varphi\theta_0}) |u-q| \geq
    |u-v| \geq |u-\hat w| - |\hat w - \hat v| - |\hat v - v| \notag \\
  &\geq (1 - C_{22}\psi_{\theta\theta_0})|u-\hat w| - |\hat v - v| \notag \\
  &\geq 2(1 - C_{22}\psi_{\theta\theta_0})n^{\beta_1}\Delta(R) - \sqrt{d-1}n^{-\beta_0}\Delta(R) \geq \frac 32 n^{\beta_1}\Delta(R).
\end{align}
By shrinking some $\ep_i$'s we may assume $C_{22}\ep_{\min}<1/2$, so since $\psi_{\varphi\theta_0} \leq \ep_{\min}$, \eqref{utoline} yields
\[
  d(u,\Pi_{xy}) \geq d(u,\Pi_{xy}^\infty) \geq \frac{1}{\sqrt{d}} n^{\beta_1}\Delta(R),
\]
while from \eqref{ybaru},
\[
  |y-x| \leq|\ol y| + |x| + |y-\ol y| \leq 2|\ytz|R + 2\sqrt{d-1}n^{-\beta_0}\Delta(R) \leq 3|\ytz|R.
\]
By Lemma \ref{gtri} with $\delta = n^{\beta_1}\Delta(R)/\sqrt{d}|y-x|$ we then have
\[
  g(u-x) + g(y-u) - g(y-x) \geq \frac{C_{25}n^{2\beta_1}\Delta(R)^2}{|y-x|} \geq c_1n^{2\beta_1}\sigma(R).
\]

\begin{figure}
\includegraphics[width=16cm]{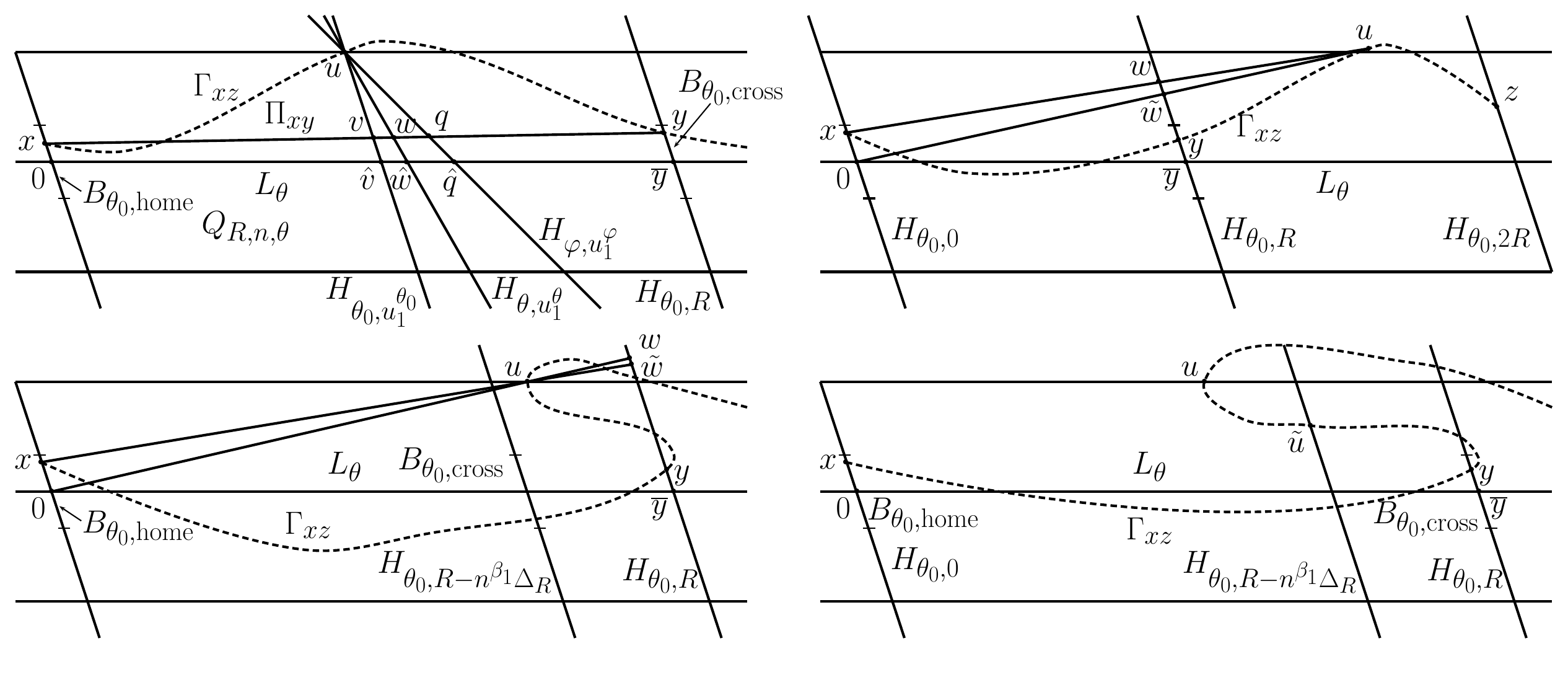}
\caption{ The four cases in Lemma 4.5. The top row shows Cases 1 and 2; the bottom shows 3 and 4. Note that the Case 2 diagram spans from $H_{\theta_0,0}$ to $H_{\theta_0,2R}$, whereas the other three show only the left half. Hash marks show the blocks $B_{\theta_0,{\rm home}}$ and $B_{\theta_0,{\rm cross}}$. }
\label{fig4-5}
\end{figure}

{\it Case 2.} $u\in \Gamma[y,z] \cap H_{\theta_0,R}^+$. This time we want a lower bound for $d(y,\Pi_{xu}^\infty)$. Let $\tilde w = \Pi_{0u}^\infty \cap H_{\theta_0,R}$ and $w = \Pi_{xu}^\infty \cap H_{\theta_0,R}$.  
Then $u_1^{\theta_0} \leq 2R+\sqrt{d}\mu\leq 3R$ so from \eqref{linegap2}
\begin{equation}\label{twby}
  |\tilde w - \ol y| = \frac{R}{u_1^{\theta_0}}d_{\theta_0}(u,L_\theta) 
    \geq \frac 13 d_{\theta_0}(u,L_\theta) \geq \frac 13 \left( 1 - C_{22}\psi_{\theta\theta_0} \right) d_{\theta}(u,L_\theta)
    \geq \frac 12 n^{\beta_1}\Delta(R).
  \end{equation}
Therefore, letting $\hat y$ denote the $\theta_0$--projection of $y$ into $H_{\theta_0,R}$, we get
\begin{equation}\label{twu}
  |\tilde w - \hat y| \geq |\tilde w - \ol y| - |\ol y - \hat y| 
    \geq \frac 12 n^{\beta_1}\Delta(R) - \sqrt{d-1}n^{-\beta_0}\Delta(R) \geq \frac 14 n^{\beta_1}\Delta(R),
\end{equation}
where the second inequality uses $\hat y\in B_{\theta_0,{\rm cross}}$.  Further, we have $|\tilde w-w| \leq |x| \leq \sqrt{d-1}n^{-\beta_0}\Delta(R)$ so using \eqref{twu},
\begin{equation}\label{dupi}
  \sqrt{d}d(\hat y,\Pi_{xu}) \geq d_{\theta_0}(\hat y,\Pi_{xu}^\infty) = |w-\hat y| \geq |\tilde w - \hat y| - |\tilde w-w| \geq 
    \frac 14 n^{\beta_1}\Delta(R).
\end{equation}
Thanks to \eqref{ybaru} we can apply
Lemma \ref{gtri} with $\delta = n^{\beta_1}\Delta(R)/4\sqrt{d}|u|$ and obtain
\begin{equation}\label{extrafar}
  g(y) + g(u-y) - g(u) \geq 
  g(\hat y) + g(u-\hat y) - g(u) - c_2 \geq \frac{C_{25}n^{2\beta_1}\Delta(R)^2}{16d|u|} \geq c_3n^{2\beta_1}\sigma(R),
\end{equation}
where the last inequality uses the readily-checked fact that $|u|\leq c_4R$.

{\it Case 3.} $u\in \Gamma[y,z]\cap \Omega_{\theta_0}(R-n^{\beta_1}\Delta(R),R)$, meaning there is a backtrack from $y$ to $u$ but not a large one.  Again let $\tilde w = \Pi_{0u}^\infty \cap H_{\theta_0,R}$ and $w = \Pi_{xu}^\infty \cap H_{\theta_0,R}$. 
(These intersections exist provided $\ep_{\min}$ is small.) 
Then $w,\tilde w\notin Q_{R,n,\theta}$ so using \eqref{linegap2},
\[
  |\tilde w - \ol y| = d_{\theta_0}(\tilde w,L_\theta) \geq (1 - C_{22}\psi_{\theta\theta_0})d_{\theta}(\tilde w,L_\theta) \geq n^{\beta_1}\Delta(R).
 \]
Then as in \eqref{twu} we have $|\tilde w - \hat y| \geq \frac 12 n^{\beta_1}\Delta(R)$.  Using $u\in \Omega_{\theta_0}(R-n^{\beta_1}\Delta(R),R)$ and $\psi_{\theta\theta_0}\leq \ep_{\min}$ we obtain readily that  $|u-w| \leq 2|\yt| n^{\beta_1}\Delta(R)$. Since $u\in H_{\theta_0,R-n^{\beta_1}\Delta(R)}^+$ we have $g(u) \geq R-n^{\beta_1}\Delta(R) \geq R/2$ so $|u| \geq \mu R/2\sqrt{d}$.
We can therefore conclude,  
using also similarity of the triangles $\Delta 0xu$ and $\Delta \tilde wwu$, that
\[
  |w-\tilde w| = \frac{|u-w|}{|u|} |x| 
    \leq \frac{4\sqrt{d}|\yt| n^{\beta_1}\Delta(R)}{\mu R} \sqrt{d-1}n^{-\beta_0}\Delta(R) \leq c_5n^{\beta_1-\beta_0}\sigma(R).
\]
Hence similarly to \eqref{dupi} we get $\sqrt{d}d(\hat y,\Pi_{xu}) \geq \sqrt{d}d(\hat y,\Pi_{xu}^\infty) \geq \frac 12 n^{\beta_1}\Delta(R)$.  As with \eqref{extrafar} we then get from \eqref{ybaru} and Lemma \ref{gtri} with $\delta = n^{\beta_1}\Delta(R)/2\sqrt{d}|u|$ that
\begin{equation}\label{extrafar2}
  g(y) + g(u - y) - g(u) \geq \frac{C_{25}n^{2\beta_1}\Delta(R)^2}{4d|u|} \geq c_6n^{2\beta_1}\sigma(R).
\end{equation}

{\it Case 4.} $u \in \Gamma[y,z] \cap H_{\theta_0,R-n^{\beta_1}\Delta(R)}^-$, meaning there is a large backtrack from $y$ to $u$.  Here we let $\tilde u$ be the first site of $\Gamma$ after $y$ which lies in $H_{\theta_0,R-n^{\beta_1}\Delta(R)}^-$. Then similarly to \eqref{ybaru} we have $|\tilde u| \leq 2R|\ytz|$,
while using \eqref{gsize} gives
\[
  \mu\sqrt{d}d(y,\Pi_{x\tilde u}) \geq \mu\sqrt{d}d(y,H_{\theta_0,R-n^{\beta_1}\Delta(R)}^-) 
    \geq d_g(y,H_{\theta_0,R-n^{\beta_1}\Delta(R)}^-) \geq n^{\beta_1}\Delta(R),
\]
so similarly to \eqref{extrafar} and \eqref{extrafar2} we get
\begin{equation}\label{extrafar3}
  g(y) + g(\tilde u - y) - g(\tilde u) \geq c_7n^{2\beta_1}\sigma(R).
\end{equation}

Thus in all of Cases 1--4 there are points $a$ preceding $b$ in $\Gamma[x,z]$, with $(a,b)=(u,y), (y,u)$, or $(y,\tilde u)$, satisfying $\max(|a-x|,|b-x|)\leq 3R|\ytz|$ and
\[
  g(a-x) + g(b-a) - g(b-x) \geq 2c_8n^{2\beta_1}\sigma(R).
\]
It follows from Proposition \ref{hmu} and \eqref{Rncond} that 
\[
  h(a-x) + h(b-a) - h(b-x) \geq c_8n^{2\beta_1}\sigma(R).
\]
With this we can proceed as in Case 1 of Proposition \ref{transfluct2} and sum over all $O(R^d)$ choices for each of $a,b$ to obtain
\begin{equation}\label{pg5a}
  P(G_1) \leq c_9R^{2d}\exp\left( -c_{10}n^{2\beta_1} \right) 
    \leq \exp\left( -\frac 12 c_{10}n^{2\beta_1} \right),
\end{equation}
which proves \eqref{allbound}.

\subsection{Proof of Lemma \ref{badbehavG15}}
Suppose $\tau\in G_2$ and $\Gamma = \Gamma[x,y]$ is a geodesic as described in $G_2$, from some $x\in B_{\theta_0,{\rm home}}^{\rm rfat}$ to $y \in B_{\theta_0,{\rm cross}}^{\rm fat}$, containing a site $u\in H_{\theta,s}^{\rm rfat} \bs B_{s,\theta,{\rm home},+}^{\rm rfat}$.  See Figure \ref{fig4-6}. Let $\hat u = (s,u_2^\theta)_\theta$ be the $\theta$--projection of $u$ into $H_{\theta,s}$ and $\hat x = (0,x_2^{\theta_0})_{\theta_0}$ the $\theta_0$--projection of $x$ into $H_{\theta_0,0}$, so $|u-\hat u| \leq d$ and $|x-\hat x| \leq d$, by \eqref{closeproj}. Let $\alpha= (y-x)/|y-x|$; we want to show that
\[
  D_{ \alpha,|(y-x)_1^\alpha| }(u-x) \geq n^{2\beta_1},
\]
so Proposition \ref{transfluct2} can be applied. To get this we need to convert information about $\theta$-- and $\theta_0$--coordinates to information about $\alpha$--coordinates.  Let 
\[
  w=(s,\hat x_2^\theta)_\theta = L_\theta(\hat x) \cap H_{\theta,s}, \quad z=(s,0)_\theta = L_\theta(0) \cap H_{\theta,s}, \quad
    q=(x_1^\theta,0)_\theta = L_\theta \cap H_{\theta,x_1^\theta},
\]
so
\[
  |w-\hat x| = (\hat u - \hat x)_1^\theta|\yt|, \quad |\hat u-w| = |(\hat u - \hat x)_2^\theta|, \quad \hat x - q = w-z.
\]
Since $\hat u\notin B_{s,\theta,{\rm home},+}$ we have $|\hat u-z|\geq 2\sqrt{d}n^{-\beta_0}\Delta(R)$.  Since $\hat x \in B_{\theta_0,{\rm home}}$ we have $|\hat x|\leq \sqrt{d}n^{-\beta_0}\Delta(R)$, and therefore  from \eqref{linegap2} we obtain
\[
  |w-z| = |\hat x - q| \leq |\hat x| + |q| \leq |\hat x| + C_{22}\psi_{\theta_0\theta} |\hat x-q| \quad\text{so}\quad  |w-z| \leq 
    (1 - C_{22}\ep_{\min})^{-1}\sqrt{d}n^{-\beta_0}\Delta(R)
\]
and then
\begin{equation}\label{hatux}
  |(\hat u - \hat x)_2^\theta| = |\hat u-w| \geq |\hat u-z| - |w-z| \geq (2 - (1 - C_{22}\ep_{\min})^{-1})\sqrt{d}n^{-\beta_0}\Delta(R) 
    \geq \frac 12 \sqrt{d}n^{-\beta_0}\Delta(R).
\end{equation}
Hence using \eqref{dvsdth} and \eqref{thetanorm},
\[
  |\hat u- \hat x|_{\alpha,\infty} \geq \frac{|(\hat u - \hat x)_2^\theta|}{\sqrt{d-1}(1+|y_\alpha|)} \geq c_1n^{-\beta_0}\Delta(R),
\] 
so using \eqref{nvsR2} and the last inequality in \eqref{nvsR},
\begin{equation}\label{Philow}
  \Phi(|\hat u- \hat x|_{\alpha,\infty}) \geq c_2\Phi(n^{-\beta_0}\Delta(R)) \geq c_2\Phi(R^{1/2}) \geq 2R^{(1-\chi_2)/2} \geq 2n^{2\beta_1}.
\end{equation}

\begin{figure}
\includegraphics[width=13cm]{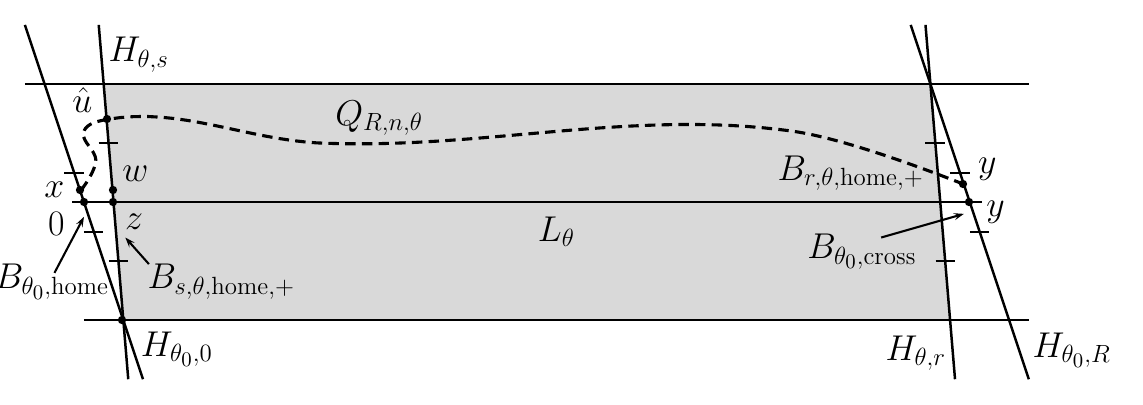}
\caption{ Illustration for the proof of Lemma \ref{badbehavG15}. The geodesic $\Gamma_{xy}$ evades the enlarged $\theta$--block $B_{s,\theta,{\rm home},+}$. $q$ (not shown) lies just left of 0. }
\label{fig4-6}
\end{figure}

Having \eqref{Philow}, by \eqref{minwhich}, to obtain a lower bound for $D_{\alpha,|(y-x)_1^\alpha|}(u-x)$ we need only obtain a lower bound for $|(\hat u - \hat x)_2^\alpha|^2/\Xi(|(\hat u - \hat x)_1^\alpha|)^2$, under the added condition
\begin{equation}\label{small2}
  |(\hat u - \hat x)_2^\alpha| \leq (\hat u - \hat x)_1^\alpha.
\end{equation}
From the definition of $s$, there must exist a point $v\in H_{\theta_0,0} \cap H_{\theta,s} \cap Q_{R,n,\theta}$, 
so using \eqref{linegap2} we have 
\begin{equation}\label{wxdist}
  |w-\hat x| \leq C_{22}\psi_{\theta_0\theta}|w-v|\leq 4C_{22}\psi_{\theta_0\theta}\sqrt{d}n^{\beta_1}\Delta(R).
\end{equation}
Now in view of \eqref{rslower} and Lemma \ref{coordchg} the $\theta$--ratio of $y-x$ satisfies
\[
  \frac{|y_2^\theta - x_2^\theta|}{|y_1^\theta - x_1^\theta|}  \leq \frac{|(y-\ol y)_2^\theta|+|x_2^\theta|}{r-s} 
    \leq \frac{2|(y-\ol y)_2^{\theta_0}|+2|x_2^{\theta_0}|}{r-s} \leq  \frac{5n^{-\beta_0}\Delta(R)}{R}
\]
so from \eqref{tanvsratio},
\begin{equation}\label{psisize}
  \psi_{\alpha\theta} \leq \frac{10n^{-\beta_0}\Delta(R)}{|\yt| R}.
\end{equation}
From Lemma \ref{coordchg}, \eqref{magsize}, and \eqref{small2} we obtain that, after reducing $\ep_{\min}$ if necessary,
\begin{align*}
  |(\hat u - \hat x)_1^\alpha| &\leq |(\hat u - \hat x)_1^\theta| + C_{23}\psi_{\alpha\theta}|\hat u - \hat x| \\
  &\leq |(\hat u - \hat x)_1^\theta| + \frac 12 |(\hat u - \hat x)_1^\alpha| 
\end{align*}
so from \eqref{wxdist},
\begin{equation}\label{ux1}
  |\yt|\,|(\hat u - \hat x)_1^\alpha| \leq 2|\yt|\,|(\hat u - \hat x)_1^\theta| = 2|w-\hat x| \leq c_3\ep_1 n^{\beta_1}\Delta(R).
\end{equation}
From \eqref{nRell} we have $R/\Delta(R) \geq 2c_4n^{\beta_1}$.
From this along with Lemma \ref{coordchg}, \eqref{hatux}, \eqref{small2}, \eqref{psisize} and \eqref{ux1}, we get
\begin{align*}
  |(\hat u - \hat x)_2^\alpha| &\geq |(\hat u - \hat x)_2^\theta| - C_{23}\psi_{\alpha\theta}|\hat u - \hat x| \\
  &\geq |(\hat u - \hat x)_2^\theta| - C_{23}\psi_{\alpha\theta} (1 + |y_\alpha|) |(\hat u - \hat x)_1^\alpha| \\
  &\geq  \frac{\sqrt{d}}{2} n^{-\beta_0}\Delta(R) \left( 1 - c_4\frac{n^{\beta_1}\Delta(R)}{R} \right) \\
  &\geq  \frac{\sqrt{d}}{4} n^{-\beta_0}\Delta(R).
\end{align*}
From \eqref{Rncond} and the first inequalities in \eqref{betas1} and \eqref{betas4} we have
\[
  \Delta(R)^{2(1-\chi_2)} \geq \Delta(R)^{1-\chi_2} \geq C_{46}^{-1/(d-1)}n^{(1-\chi_2)/(d-1)} \geq 
     2n^{(3+\chi_2)\beta_1+2\beta_0}
\]
and hence using \eqref{powerlike} and \eqref{ux1},
\begin{equation}\label{Xilow}
  \frac{|(\hat u - \hat x)_2^\alpha|^2}{\Xi(|(\hat u - \hat x)_1^\alpha|)^2} \geq 
    c_4\frac{ n^{-2\beta_0}\Delta(R)^2}{\Delta(n^{\beta_1}\Delta(R))^2\log R} 
    \geq \frac{\Delta(R)^{2(1-\chi_2)}}{n^{\beta_1(1+\chi_2)+2\beta_0}} \geq 2n^{2\beta_1}.
\end{equation}
Replacing $\hat u, \hat x$ with $u,x$ on the left sides of \eqref{Philow} and \eqref{Xilow} changes each by at most a factor of 2, so those two inequalities show that
\begin{equation}\label{Dlower3}
  \tau\in G_2 \implies D_{ \alpha,|(y-x)_1^\alpha| }(u-x) \geq n^{2\beta_1},
\end{equation}
and then summing \eqref{transfluct1} over $x\in B_{\theta_0,{\rm home}}^{\rm rfat}$ and $y \in B_{\theta_0,{\rm cross}}^{\rm fat}$ shows
\begin{equation}\label{pg15}
  P(G_2) \leq c_5R^{2(d-1)}\exp\left( -C_{27}n^{2\beta_1} \right) \leq \exp\left( -\frac 12 C_{27}n^{2\beta_1} \right).
\end{equation}

\subsection{Proof of Lemma \ref{badbehavG2G7}}
Since $\theta=\theta_0$ we'll simply call it $\theta$.
We can handle $G_3 \cup G_4$ as one, as follows.  Suppose $\tau\in G_3$.  This means there exists an $(\ell,\theta)$--interval geodesic $\Gamma_{xy}\subset Q_{R,n,\theta}$ containing sites $u$ preceding $v$ with $u_1^\theta - v_1^\theta \geq n^{-\beta_4}\ell/2$.  By choosing a different $v\in\Gamma$ we may assume also $n^{-\beta_4}\ell/2 \leq u_1^\theta - v_1^\theta \leq n^{-\beta_4}\ell$.  We have using \eqref{ellsigma} that
\[
  h(v-u) \geq g(v-u) \geq g((v_1^\theta - u_1^\theta)\yt) \geq \frac{n^{-\beta_4}\ell}{2} \geq \frac \eta 8 \sigma(\ell).
\]
This shows that for the event
\begin{align*}
  G_{13}: &\text{ ({\it bad-direction segment in a geodesic}) There exists an $(\ell,\theta)$--interval geodesic $\Gamma\subset Q_{R,n,\theta}$} \\
  &\qquad \text{and sites $u$ preceding $v$ in $\Gamma$  with $|(v-u)_1^{\theta}| \leq n^{-\beta_4}\ell$ and } \\
   &\qquad h(v-u) - h(|(v-u)_1^{\theta}|\yt)1_{\{(v-u)_1^{\theta}\geq 0\}} \geq \frac \eta 8 \sigma(\ell), 
\end{align*}
(which clearly satisfies $G_4\subset G_{13}$) we have $G_3 \cup G_4 \subset G_{13}$.

To bound $P(G_{13})$, suppose $\Gamma_{xy}$ is a target-directed $(\ell,\theta)$--interval geodesic from $H_{\theta,(i-1)\ell}^{\rm fat}$ to $H_{\theta,i\ell}^{\rm fat}$ for some $i\leq R/\ell$, containing sites $u$ preceding $v$ as in $G_{13}$, that is,
\begin{equation}\label{hgap}
  |(u-v)_1^{\theta}| \leq n^{-\beta_4}\ell,\qquad h(u-v) - h(|(u-v)_1^{\theta}|\yt)1_{\{(v-u)_1^{\theta}\geq 0\}} \geq \frac \eta 8 \sigma(\ell).
\end{equation}
Let $\alpha = (y-u)/|y-u|$; we want a lower bound for $D_{\alpha,|(y-u)_1^\alpha|}(v-u)$ so we can use Proposition \ref{transfluct2}; the complications come from having hypotheses in terms of $\theta$-- and $\theta_0$--coordinates and desired conclusions in terms of $\alpha$--coordinates.  We may assume $u_1^\theta\leq (i-1/2)\ell$, as the other case is symmetric. 
Then $u,y \in Q_{R,n,\theta}$ with $(y-u)_1^\theta\geq\ell/2$, so from \eqref{nRell}, provided $R$ is large the $\theta$--ratio of $y-u$ is at most $8\sqrt{d-1}n^{\beta_1}\Delta(R)/\ell \leq 8\sqrt{d-1}n^{-\beta_4}<|\yt|\ep_{\min}/2$, so by \eqref{tanvsratio}, 
\begin{equation}\label{psimax}
  \psi_{\alpha\theta}\leq\tan \psi_{\alpha\theta} \leq \frac{16n^{\beta_1}\Delta(R)}{|\yt|\ell} < \ep_{\min}.
\end{equation}
Let
\[
  \lambda = \frac{1}{2|\yt|(1+(\inf_\varphi |y_\varphi|)^{-1})}, \quad t_0 = 
    \frac{\lambda\mu\eta\sigma(\ell)}{32\sqrt{d}}, \quad t = \max\left(v_1^\theta - u_1^\theta, t_0 \right), \quad q = u+t\yt.
\]
From \eqref{ellsigma} and \eqref{hgap} we have
\begin{equation}\label{tvsell}
  t\leq n^{-\beta_4}\ell.
\end{equation}
Define the intersection points 
\[
   p = L_\alpha(u) \cap H_{\theta,q_1^\theta}, \quad w = L_\alpha(u) \cap H_{\alpha,v_1^\alpha}, \quad
     x = L_\alpha(u) \cap H_{\theta,v_1^\theta},
\]
so that $v-u = (w-u) + (v-w)$ is the decomposition of $v-u$ into $\alpha$--components. See Figure \ref{fig4-7case23}.
From \eqref{gsize} and \eqref{hgap}, provided $R$ (and hence $\ell$ and $|u-v|$) is large we have
\begin{equation}\label{vusize}
  |v-u| \geq \frac{\mu}{\sqrt{d}} g(v-u) \geq \frac{\mu}{2\sqrt{d}} h(v-u) \geq \frac{\mu\eta\sigma(\ell)}{16\sqrt{d}} 
    = \frac{2t_0}{\lambda}.
\end{equation}
From \eqref{linegap2} and \eqref{psimax} we have
\begin{equation}\label{wxgap}
  |w-x| \leq C_{22}\psi_{\alpha\theta}|v-w| \leq \frac{|y_\alpha|}{2}|v-w|
\end{equation}
while from \eqref{Vgap},
\begin{equation}\label{pqgap}
  |p-q| \leq C_{22}\psi_{\alpha\theta}|q-u| = C_{22}\psi_{\alpha\theta}|\yt|t.
\end{equation}
Using \eqref{thetanorm}, \eqref{ell-lower}, and \eqref{vusize} we then get
\begin{equation}\label{Dlower}
  \Phi(|v-u|_{\alpha,\infty}) \geq c_1\Phi(\sigma(\ell)) \geq \ell^{\chi_1(1-\chi_2)} \geq n^{4\beta_1}.
\end{equation}

As a shorthand we say a point $y$ is \emph{strictly behind} a hyperplane $H_{\varphi,t}$ if $y$ lies in the interior of $H_{\varphi,t}^-$, and $y$ is $\varphi$--\emph{behind} a point $z$ if $z_1^\varphi - y_1^\varphi$ is nonnegative.  
We now consider 3 cases; see Figure \ref{fig4-8}.

{\it Case 1.} $w \in H_{\alpha,u_1^\alpha}^-$, meaning $v$ (hence also $w$) is $\alpha$--behind $u$. Here using \eqref{minwhich} and \eqref{Dlower} we get
\begin{equation}\label{Dlower2}
  D_{\alpha,|(y-u)_1^\alpha|}(v-u) = \Phi(|v-u|_{\alpha,\infty}) \geq n^{4\beta_1},
\end{equation}
so as with \eqref{Dlower3} and \eqref{pg15}, from Proposition \ref{transfluct2},
\begin{equation}\label{C1}
  P(G_{13} \text{ and Case 1 holds}) \leq C_{26}e^{-C_{27}n^{2\beta_1}}.
\end{equation}

{\it Case 2.} $w \in H_{\alpha,u_1^\alpha}^+, v_1^\theta - u_1^\theta < t_0$. This means $t=t_0$, $v_1^\alpha = w_1^\alpha \geq u_1^\alpha$, and $v$ is strictly behind $H_{\theta,u_1^\alpha+t}$. (See Figure \ref{fig4-7case23}.) We may view this as the ``strongly sidestepping'' case in which the increment of $\Gamma_{xy}[u,v]$, when projected onto the $\alpha$ direction, is very short.
Now $p,u,w,x$ lie on the line $L_\alpha(u)$, with $u,x$ strictly behind $H_{\theta,u_1^\theta+t}$ and $\theta$--behind $p\in H_{\theta,u_1^\theta+t}$, and with $u$ $\alpha$--behind $w$.  If $|v_2^\alpha - u_2^\alpha| \geq v_1^\alpha - u_1^\alpha$ 
then by \eqref{minwhich}, \eqref{Dlower2} applies.  We will show by contradiction that the opposite case does not occur.  Thus we assume
\begin{equation}\label{bigwu}
  |w-u| = |y_\alpha|(v_1^\alpha - u_1^\alpha) > |y_\alpha| |v_2^\alpha - u_2^\alpha|  = |y_\alpha| |v-w|.
\end{equation}
If $w$ is $\alpha$--behind $p$ then $w\in H_{\theta,t}^-$ and $|w-u|\leq |p-u|$; if $w$ is not $\alpha$--behind $p$ then $w\in H_{\theta,u_1^\alpha+t}^+$ and $|w-p|\leq |w-x|$.  Either way, since $|q-u|=|\yt|t$ we have using \eqref{psimax}, \eqref{vusize}, \eqref{wxgap}, \eqref{pqgap}, \eqref{bigwu} that
\begin{align}\label{halfarg}
  |w-u| &\leq |p-u| + |w-p|1_{\{w\in H_{\theta,u_1^\theta+t}^+\}} \leq |q-u| + |p-q| + |w-x| \notag\\
  &\leq 2|\yt| t + \frac 12 |w-u| \leq \lambda |\yt| |v-u| +  \frac 12 |w-u|
\end{align}
so $|w-u| \leq 2\lambda |\yt| |v-u|$, which with \eqref{bigwu} yields
\[
  |v-u| \leq |v-w| + |w-u| \leq (1+|y_\alpha|^{-1})|w-u| \leq 2 \lambda |\yt|(1+|y_\alpha|^{-1}) |v-u| < |v-u|,
\]
which is the desired contradiction.  This means \eqref{bigwu} cannot hold, so \eqref{Dlower2} always applies in Case 2, and therefore as with \eqref{C1},
\begin{equation}\label{C2}
  P(G_{13} \text{ and Case 2 holds}) \leq C_{26}e^{-C_{27}n^{2\beta_1}}.
\end{equation}

\begin{figure}
\includegraphics[width=14cm]{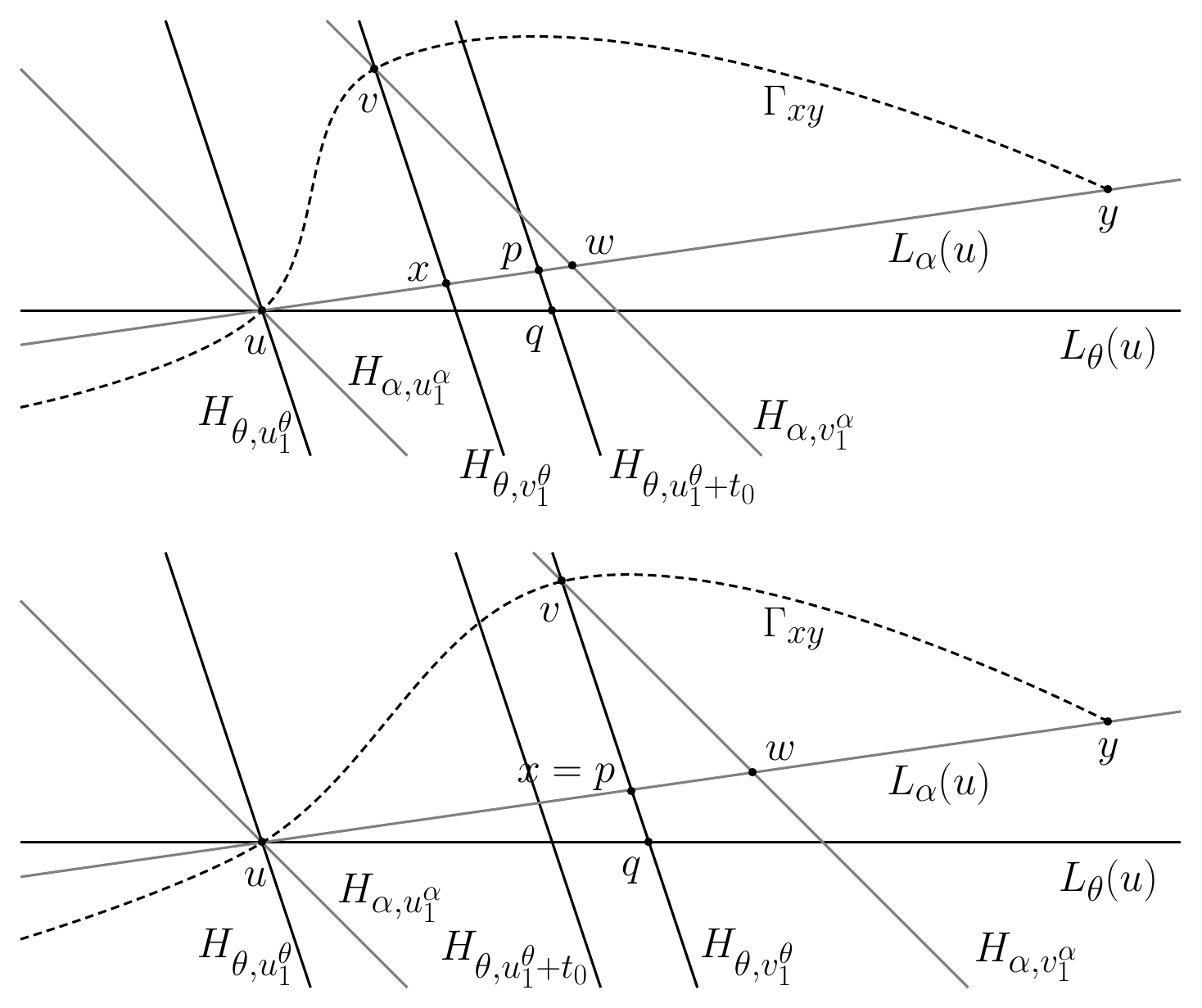}
\caption{ Illustrations for Lemma 4.7. Top: Case 2, in which $v$ lies strictly behind $H_{\theta,u_1^\theta+t_0}$ but ahead of $H_{\alpha,u_1^\alpha}$, and $t=t_0$. Bottom:  Case 3, in which $v$ lies ahead of both $H_{\theta,u_1^\theta+t_0}$ and $H_{\alpha,u_1^\alpha}$, and $t\geq t_0$.
In both cases $w$ may lie on either side of $x$ or $p$. In Case 1 (not shown) there is an $\alpha$--backtrack: $v$ lies behind $H_{\alpha,u_1^\alpha}$. }
\label{fig4-7case23}
\end{figure}

{\it Case 3.} $w \in H_{\alpha,u_1^\alpha}^+,\, v_1^\theta - u_1^\theta \geq t_0$. This is the ``weakly sidestepping'' case. From the definition of $G_{13}$, we have 
\begin{equation}\label{C3facts}
  t_0 \leq t = v_1^\theta - u_1^\theta \leq n^{-\beta_4}\ell, \quad v\in H_{\theta,u_1^\theta+t}, \quad x=p,
\end{equation}
and the indicator function in \eqref{hgap} is 1.
As in Case 2 we may assume \eqref{bigwu}.  As in \eqref{halfarg} we have
\[
  |w-u| \leq |q-u| + |p-q| +|w-x| \leq 2|\yt|t + \frac{|y_\alpha|}{2} |v-w| \leq 2|\yt|t + \frac 12 |w-u|,
\]
so 
\begin{equation}\label{wubound}
  |y_\alpha| |v_1^\alpha - u_1^\alpha| = |w-u| \leq 4|\yt|t.
\end{equation}

We claim that
\begin{equation}\label{vqclaim}
  |(v-u)_2^\theta| = |v-q| \geq n^{\beta_1}\Delta(t).
\end{equation}
Suppose not; we will show that the second inequality in \eqref{hgap} is contradicted. From \eqref{ell-lower}  and the fact that $t\geq t_0$, the $\theta$--ratio of $v-u$ satisfies
\begin{equation}\label{vuratio}
  \frac{|v_2^\theta - u_2^\theta|}{|v_1^\theta - u_1^\theta|} =\frac{|v-q|}{t} < \frac{n^{\beta_1}\Delta(t)}{t} 
    \leq c_2 \frac{n^{\beta_1}\Delta(\ell)}{\ell}
    \leq c_3n^{\beta_1}\ell^{-\chi_1(1-\chi_2)/2} < n^{-\beta_1}.
\end{equation}
Since $v$ lies in the tangent plane $H_{\theta,u_1^\theta+t}$ to $u+t\mkB_g$ at $u+t\yt=q$, it follows
from \eqref{curv}, using the first inequality in \eqref{vuratio}, that
\[
  d(v,u+t\mkB_g) \leq C_6n^{2\beta_1}\sigma(t).
\]
Hence by \eqref{gsize},
\begin{equation}\label{gdiff}
  g(v-u) \leq t + C_6 \mu\sqrt{d} n^{2\beta_1}\sigma(t)
    = g((v_1^\theta - u_1^\theta)\yt) + C_6 \mu\sqrt{d} n^{2\beta_1}\sigma(t).
\end{equation}
From \eqref{bigwu} and \eqref{wubound} we have
\begin{equation}\label{vut}
  |v-u| \leq (1 + |y_\alpha|^{-1}) |w-u| \leq 4|\yt|(1 + |y_\alpha|^{-1})t.
\end{equation}
From \eqref{Rncond} we have $\log t \leq n^{2\beta_1}$. 
With \eqref{powerlike}, \eqref{tvsell}, \eqref{C3facts}, \eqref{gdiff}, \eqref{vut}, the last inequality in \eqref{betas1}, and Proposition \ref{hmu} this shows that
\begin{align}
  h(v-u) - h((v_1^\theta - u_1^\theta)\yt) &\leq g(v-u) - g((v_1^\theta - u_1^\theta)\yt) + C_{16}\sigma(|v-u|)\log |v-u| \notag\\
  &\leq \left( C_6 \mu\sqrt{d} n^{2\beta_1} + c_4\log t \right) \sigma(t) \notag\\
  &\leq c_5n^{2\beta_1} \sigma(t) \notag\\
  &\leq c_6n^{2\beta_1 - \chi_1\beta_4} \sigma(\ell) \notag\\
  &< \frac \eta 8 \sigma(\ell),
\end{align}
which contradicts \eqref{hgap}; this proves our claim \eqref{vqclaim}. Note that \eqref{vut} is not affected by the contradiction established.

We then have using Lemma \ref{coordchg}, \eqref{psimax}, \eqref{vqclaim}, and \eqref{vut} that since $ |(v-u)_2^\theta|=|v-q|$,
\begin{equation}\label{vwlower2}
  |(v-u)_2^\alpha| \geq |v-q| - C_{23}\psi_{\alpha\theta}|v-u| \geq n^{\beta_1}\Delta(t) - c_7\frac{n^{\beta_1}\Delta(R)}{\ell}t
    \geq n^{\beta_1}\left( \Delta(t) - c_8\frac{t\Delta(R)}{\ell} \right).
\end{equation}
It follows from \eqref{powerlike}, the first inequality in \eqref{betas2}, and \eqref{C3facts} that
\[
  \left( \frac{\ell\Delta(t)}{t}\right)^2 = \frac{\ell^2\sigma(t)}{t} \geq C_3^{-1} \frac{\ell^{2-\chi_2}\sigma(\ell)}{t^{1-\chi_2}}
    \geq C_3^{-1} n^{(1-\chi_2)\beta_4} \ell\sigma(\ell) \geq C_3^{-2} n^{(1-\chi_2)\beta_4-(1+\chi_2)\beta_2} \Delta(R)^2
    \geq n^{\beta_2}\Delta(R)^2,
\]
which with \eqref{vwlower2} shows that
\[
  |(v-u)_2^\alpha| \geq \frac 12 n^{\beta_1}\Delta(t).
\]
From this and \eqref{wubound},
\[
   \frac{|v_2^\alpha -u_2^\alpha|^2}{\Xi(|v_1^\alpha -u_1^\alpha|)^2} \geq \frac{n^{2\beta_1}\Delta(t)^2}{4\Xi(8t)^2}
     \geq c_9\frac{n^{2\beta_1}}{\log t} \geq c_9\frac{n^{2\beta_1}}{\log R},
\]
where the first inequality uses $|\yt|\leq 2|y_\alpha|$, from \eqref{zangle}. 
Then in view of \eqref{Dlower} we conclude
\begin{equation}\label{Dlower4}
  D_{\alpha,|(y-u)_1^\alpha|}(v-u) \geq c_9\frac{n^{2\beta_1}}{\log R}.
\end{equation}
Therefore as with \eqref{C1},
\begin{equation}\label{C3}
  P(G_{13} \text{ and Case 3 holds}) \leq C_{26}\exp\left(-c_{10}\frac{n^{2\beta_1}}{\log R} \right),
\end{equation}
which with \eqref{C1} and \eqref{C2} yields
\begin{equation}\label{G7bound}
  P(G_3\cup G_4) \leq P(G_{13}) \leq 3C_{26}\exp\left(-c_{10}\frac{n^{2\beta_1}}{\log R} \right).
\end{equation}

\subsection{Proof of Lemma \ref{badbehavG9}}
Let $i\leq R/\ell$ and suppose $u,w\in H_{\theta,(i-1)\ell}^{\rm fat} \cap Q_{R,n,\theta}\cap\ZZ^d$ and $v,x \in H_{\theta,i\ell}^{\rm fat} \cap Q_{R,n,\theta}\cap\ZZ^d$ with
\begin{equation}\label{uvwxgaps}
  |u-w| \leq 2\sqrt{d}n^{-\beta_0}\Delta(R), \quad |v-x| \leq 2\sqrt{d}n^{-\beta_0}\Delta(R).
\end{equation}
We have 
\begin{equation}\label{Tcomp}
  |T(u,v) - T(w,x)| \leq |T(u,v) - T(u,x)| + |T(u,x) - T(w,x)|,
\end{equation}
and we want to use Proposition \ref{transTincr} to bound the probability that either difference on the right exceeds $\eta\sigma(\ell)/16$.  Let $\hat u,\hat v, \hat w,\hat x$ be the $\theta$--projections of $u,w,v,x$, with $\hat u,\hat w$ into $H_{\theta,(i-1)\ell}$ and $\hat v,\hat x$ into $H_{\theta,i\ell}$.  See Figure \ref{fig4-8}. 
Since $\hat u, \hat v,\hat x \in Q_{R,n,\theta}$ we have using \eqref{nRell}
\[
  |(\hat v-\hat u) - \ell\yt| \leq |\hat v - i\ell\yt|+ |\hat u - (i-1)\ell\yt| \leq 2n^{\beta_1}\Delta(R) \leq \ep_{\min}\ell,
\]
and using \eqref{closeproj}
\begin{equation}\label{vxgap}
  |\hat v - \hat x| \leq |v-x| + 2d \leq 3\sqrt{d} n^{-\beta_0}\Delta(R) < \ep_{\min}\ell,
\end{equation}
so Lemma \ref{gdiff1} gives
\[
  |g(\hat v - \hat u) - g(\hat x - \hat u)| \leq \frac{C_{29}(6\sqrt{d}n^{\beta_1-\beta_0}
    +9dn^{-2\beta_0})\Delta(R)^2}{n^{-\beta_2}R} \leq c_1n^{\beta_1+\beta_2-\beta_0}\sigma(R).
\]
Note that by the second inequality in \eqref{betas2}, the exponent on the right here is negative.
We may assume the points are labeled so that $g(\hat x - \hat u) \leq g(\hat v - \hat u)$; there then exists a point $\hat z \in \Pi_{\hat u\hat v}$ (close to $\hat v$) with $g(\hat z-\hat u) = g(\hat x - \hat u)$, and $z\in \ZZ^d$ with $|z-\hat z| \leq \sqrt{d}$.  We then have
\begin{equation}\label{gvz}
  g(\hat v-\hat z) = g(\hat v - \hat u) - g(\hat z - \hat u) \leq c_1n^{\beta_1+\beta_2-\beta_0}\sigma(R)
\end{equation}
and
\[
  |T(u,v) - T(u,x)| \leq |T(u,z) - T(u,x)| + T(z,v),
\]
and the latter shows that
\begin{equation}\label{splitbit}
  P\left( |T(u,v) - T(u,x)| \geq \frac{\eta}{16} \sigma(\ell) \right) \leq P\left( |T(u,z) - T(u,x)| \geq \frac{\eta}{32} \sigma(\ell) \right)
    + P\left( T(z,v) \geq \frac{\eta}{32} \sigma(\ell) \right).
\end{equation}
Checking the conditions of Proposition \ref{transTincr} for the first probability on the right of \eqref{splitbit}, we note first that from \eqref{closeproj},
\[
  |g(z-u) - g(x-u)| \leq g(z-\hat z) + 2g(u-\hat u) + g(x-\hat x) \leq 4\mu\sqrt{d}.
\]
Further, from \eqref{gsize},
\[
  |v-u| \geq |\hat v-\hat u| - |v-\hat v| - |u-\hat u| \geq \frac{1}{\mu\sqrt{d}}g(\hat v - \hat u) - 2d
    \geq \frac{\ell}{2\mu\sqrt{d}}  
\]
so using \eqref{Deltaratio}, \eqref{uvwxgaps}, and the second inequality in \eqref{betas2}, for $C_{40}$ from Proposition \ref{transTincr},
\[
  \Delta(|v-u|) \geq c_2\Delta(\ell) \geq c_3 n^{-\beta_2(1+\chi_2)/2} \Delta(R) 
    \geq 2C_{31}^{-1}\sqrt{d} n^{-\beta_0}\Delta(R) \geq C_{31}^{-1} |v-x|.
\]
In applying Proposition \ref{transTincr} to the first probability on the right of \eqref{splitbit} we take $\lambda = n^{2\beta_1}$, so we need to verify that for this $\lambda$,
\begin{equation}\label{OKlambda}
  \frac{\eta}{32} \sigma(\ell) \geq \lambda\sigma(\Delta^{-1}(|v-x|))\log |v-x|.
\end{equation}
Let
\[
  A = \left( \frac{32\lambda}{C_2\eta} \log |v-x| \right)^{1/\chi_1}.
\]
\begin{figure}
\includegraphics[width=14cm]{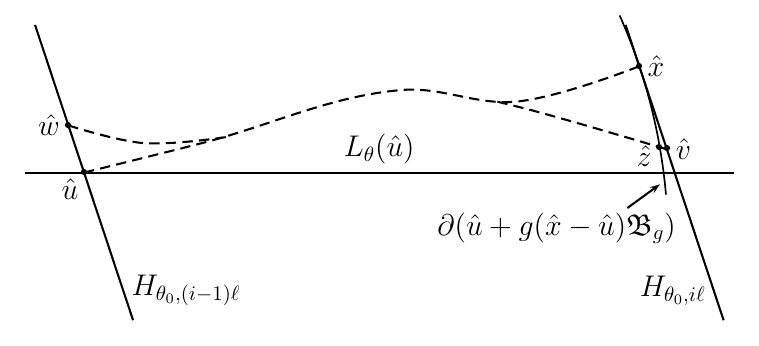}
\caption{ Illustration for Lemma 4.8. The arc containing $\hat x$ and $\hat z$ is the boundary of a $g$--ball centered at $\hat u$. The dashed lines are the geodesics $\Gamma_{uv}$ and $\Gamma_{wx}$, which are unlikely to have highly dissimilar passage times because they can follow the same path, once away from their endpoints.}
\label{fig4-8}
\end{figure}
If we can show that
\begin{equation}\label{OKlam2}
  \Delta(\ell) \geq C_3^{1/2}A^{(1+\chi_2)/2}|v-x|,
\end{equation}
then using \eqref{powerlike},
\[
  \Delta\left( \frac \ell A \right) \geq C_3^{-1/2} A^{-(1+\chi_2)/2}\Delta(\ell) \geq |v-x| \quad\text{and hence} \quad
    C_2A^{\chi_1} \leq C_2\left( \frac{\ell}{\Delta^{-1}(|v-x|)} \right)^{\chi_1} \leq \frac{\sigma(\ell)}{\sigma(\Delta^{-1}(|v-x|))}
\]
which is equivalent to \eqref{OKlambda}. In fact, using the second inequality in \eqref{betas2}, along with \eqref{Deltaratio}, \eqref{uvwxgaps}, and the fact that (by \eqref{Rncond}) $\log |v-x| \leq \log R \leq n^{2\beta_1}$, we obtain that provided $n$ is large (depending on $\eta$),
\begin{align*}
  C_3^{1/2}A^{(1+\chi_2)/2}|v-x| 
    &\leq c_4n^{-\beta_0}\Delta(R)\left( n^{4\beta_1} \right)^{(1+\chi_2)/2\chi_1}
    \leq c_5n^{-\beta_0+2\beta_1(1+\chi_2)/\chi_1}\Delta(R) \\
  &\leq c_6n^{-\beta_0+2\beta_1(1+\chi_2)/\chi_1+\beta_2(1+\chi_2)/2}\Delta(\ell) \leq \Delta(\ell),
\end{align*}
proving \eqref{OKlam2} and hence also \eqref{OKlambda}.
Now \eqref{OKlambda} and Proposition \ref{transTincr} give
\begin{equation}\label{firstbit}
  P\left( |T(u,z) - T(u,x)| \geq \frac{\eta}{32} \sigma(\ell) \right) \leq C_{33}\exp\left(-C_{34}n^{2\beta_1}\right).
\end{equation}

Turning now to the last probability in \eqref{splitbit}, we have using \eqref{powerlike}, the second inequality in \eqref{betas2}, and \eqref{gvz} that
\[
  h(v-z) \leq c_7 + 2g(v-z) \leq c_8 + 2g(\hat v - \hat z) \leq 2c_1n^{\beta_1+\beta_2-\beta_0}\sigma(R)
    \leq n^{\beta_1+\beta_2(1+\chi_2)-\beta_0}\sigma(\ell) \leq \frac{\eta}{64}\sigma(\ell)
\]
and similarly
\[
  |v-z| \leq n^{\beta_1+\beta_2(1+\chi_2)-\beta_0}\sigma(\ell),
\]
so from \eqref{powerlike}, \eqref{expbd}, and \eqref{ell-lower}, 
\begin{align*}
  P\left( T(z,v) \geq \frac{\eta}{32} \sigma(\ell) \right) &\leq P\left( T(z,v) - ET(z,v) \geq \frac{\eta}{64} \sigma(\ell) \right)
    \leq 4\exp\left( -\frac{\eta}{64} \frac{\sigma(\ell)}{\sigma( \sigma(\ell))} \right) \\
  &\leq \exp\left( -\ell^{\chi_1(1-\chi_2)} \right) \leq \exp\left( -n^{2\beta_1} \right).
\end{align*}
Combining this with \eqref{Tcomp}, \eqref{splitbit}, and \eqref{firstbit}, along with a similar computation for the second term on the right in \eqref{Tcomp} in place of the first term, we get 
\[
  P\left( |T(u,v) - T(w,x)| \geq \frac{\eta}{4} \sigma(\ell) \right) \leq c_9\exp\left( -c_{10}n^{2\beta_1} \right).
\]
Summing this over the $O(R^{4(d-1)})$ possible values of $(u,v,w,x)$ we get as in \eqref{pg5a} that
\begin{equation}\label{PG9}
  P(G_5) \leq c_9\exp\left( -\frac 12 c_{10}n^{2\beta_1} \right).
\end{equation}

\section{Table of notation}  \label{tableapp}
These are the most commonly used notations in the paper, with description and page of first appearance.  Many notations only used in a single proof are omitted.\\

\begin{longtable}{ l l l }
$H_{\theta,s},H_{\theta,s}^\pm$ & Hyperplane tangent to $g$--ball at direction $\theta$; resulting halfspaces & \pageref{Hhyp} \\
$\EE^d$ & Bonds of $\ZZ^d$ & \pageref{EEd} \\
$g(\cdot),\mkB_g$ & Norm based on asymptotic passage times, and its unit ball & \pageref{gx} \\
$\yt,z_\theta$ & Vector in $\pa \mkB_g$; vector $\perp$ tangent plane at $\yt$ & \pageref{yth} \\
$\nu_\bullet(\cdot)$ & Subpolynomial correction factors (various subscripts $\bullet$) & \pageref{nui} \\
$\chi^\pm,\chi$ & Upper and lower fluctuation exponents; common value & \pageref{chis} \\
$\lambda^\pm(r)$ & Upper and lower fluctuation scales for passage times & \pageref{lams} \\
$J(\cdot,\cdot)$ & Cone of angles & \pageref{jth} \\
$L_\theta(u),L_\theta$ & Lines through $u$ and 0 at angle $\theta$ & \pageref{lth} \\
$\rho_\bullet,\bar\rho_\bullet$ & Crossing densities (various subscripts $\bullet$) & \pageref{rhos} \\
$\Delta(r),\xi$ & Scale of transverse wandering over distance $r$; growth exponent & \pageref{del} \\
$(u_1^\theta,u_2^\theta)_\theta$ & $\theta$--coordinates, using tangent plane in direction $\theta$ & \pageref{coord} \\
$\chi_1,\chi_2$ & Sandwiching values $\chi_1 < \chi<\chi_2$ & \pageref{chii} \\
$\sigma(\cdot),\sigma_-(\cdot)$ & Regular upper and lower bounds for $\lambda^\pm(\cdot)$ &\pageref{sigs} \\
$\pi(\cdot)$ & The subpolynomial function $\sigma(\cdot)/\sigma_-(\cdot)$ & \pageref{sigs} \\
$h(x)$ & $ET(0,x)$ & \pageref{hofx} \\
$|\cdot|_{\theta,\infty}$ & Sup norm based on $\theta$--coordinates & \pageref{tnorm} \\
$d_\theta(\cdot,\cdot)$ & Distance via hyperplane $H_{\theta,\cdot}$ & \pageref{hyp} \\
$\psi_{ab}$ & Angle in $[0,\pi]$ between vectors $a,b$ & \pageref{ang} \\
$\mu$ & Time constant $g(e_1)$ & \pageref{mug} \\
$\Omega_\theta(\cdot,\cdot)$ & Slab & \pageref{oth} \\
$H_{\theta,s}^{\rm fat},H_{\theta,s}^{\rm rfat}$ & Fattened hyperplanes & \pageref{fat} \\
$\Phi(s)$ & Cost of geodesic deviation by $s$, with macroscopic direction error & \pageref{cphi} \\
$\Xi(s)$ & Scale of geodesic deviation of cost $\log s$, for small direction error & \pageref{cxi} \\
$D_\theta(u),D_{\theta,r}(u)$ & Overall cost of geodesic deviating through $u$; symmetrized version & \pageref{cxi} \\
$\Pi_{ab},\Pi_{ab}^\infty$ & Line segment $a$ to $b$, line through $a,b$ & \pageref{pil}, \pageref{piline} \\
$E_{\theta,r,c}$ & Tube--and--cylinders region, level set of $D_{\theta,r}$ & \pageref{tcyl} \\
$\mathbb{S}_\theta(\cdot)$ & Intersection of tube and cylinder, in tube--and--cylinders region & \pageref{zcyl} \\
$\CC_{\theta,c}$ & 0--cylinder, one cylinder of tube--and--cylinders region & \pageref{zcyl} \\
$\gamma[u,v]$ & Segment of path $\gamma$ between $u$ and $v$ & \pageref{gab} \\
$x_{\theta,s}''(\Gamma),x_{\theta,s}'(\Gamma)$ & Hyperplane entry point and preceding site & \pageref{xpri} \\
$S_{\theta,i}(\Gamma)$ & $\ell$--segment (length $\ell$) of geodesic $\Gamma$ & \pageref{seg} \\
$N_\theta(\Gamma)$ & Number of fast $\ell$--segments in geodesic $\Gamma$ & \pageref{nth} \\
$\beta_i$ & Powers applied to number $n$ of geodesics under consideration & \pageref{betas} \\
$B_{\theta,{\rm home}}$ & Home $\theta$--block in $H_{\theta,0}$, of side $2n^{-\beta_0}\Delta(R)$, centered on $L_\theta$ & \pageref{hblk} \\
$Q_{R,n,\theta}$ & Target $\theta$--box, a square tube around line $L_\theta$ & \pageref{tgt} \\
$B_{\theta_0,{\rm cross}},\,\ol y$ & $\theta_0$--block in $H_{\theta_0,R}$ (midway to $H_{\theta_0,2R}$), center $\ol y$ where $L_\theta$ crosses & \pageref{blkc} \\
$B_{r,\theta,{\rm home},+}$ & Enlarged home $\theta$--block in $H_{\theta,r}$& \pageref{blkr} \\
\end{longtable}

\noindent The following notation appears in the proof of Proposition \ref{jammed1}.

\begin{longtable}{ l l l }
$\ol y,y^*,\theta^*$ & Centers of $B_{\theta_0,{\rm cross}}$ and $B_{2R,m}$; angle of $y^*-\ol y$ & \pageref{yb},\pageref{yst} \\
$B_{2R,j}$ & $j$th $\theta_0$--block intersecting $H_{\theta_0,2R}\cap Q_{R,n,\theta}$ & \pageref{jblk} \\
$\mathfrak{G}$  & Crowded set: geodesics from $B_{\theta_0,{\rm home}}^{\rm rfat}$ to some $B_{2R,m}$ via $B_{\theta_0,{\rm cross}}$ & \pageref{crs} \\
$g_n$ & Minimum number of geodesics in a crowded set & \pageref{crs} \\
$Q_{R,n,\theta^*}^*$ & Square tube around $L_{\theta^*}(\ol y) = \Pi_{\ol yy^*}^\infty$ & \pageref{qst} \\
$\Gamma^{\rm pri},\Gamma^{\rm sec}$ & Segments of $\Gamma$ before and after ``midway'' hyperplane $H_{\theta_0,R}$ & \pageref{gpr} \\
$\mathfrak{G}_i$ & Subset of $\mathfrak{G}$ based on certain midway entry points & \pageref{Gi} \\
$\tilde B_{s^*,\theta^*,{\rm home},+}$ & Enlarged home $\theta^*$--block in $H_{\theta^*,s^*}$ & \pageref{tilb} \\
$\mT_i(j,k)$ & Set of $\ell$--segments transitioning block $j$ to $k$ in $i$th interval & \pageref{Tijk} \\
$\mO(S_{\theta,i}(\Gamma),S_{\theta,i}(\hat\Gamma))$ & Projected overlap interval of two $\ell$--segments & \pageref{oss} \\
$N_L,N_R$ & Number of lattice points in $B_{\theta,{\rm home}}^{\rm rfat}$ and $B_{\theta,{\rm cross}}^{\rm fat}$ & \pageref{nlnr} \\
$\Omega_i$ & Slab (typically) containing $i$th $\ell$--segment & \pageref{nlnr} \\
$\mathfrak{P}_i(x,y,\gamma_0)$ & Paths $x\to y$ with no or small intersection with $\gamma_0$ & \pageref{lowover} \\
$T^{\dis,i}(\gamma,\gamma_0)$ & Time $T(\gamma)$ with value for $\gamma\cap\gamma_0$ replaced by ``projected mean'' & \pageref{Td} \\
$T^{\dis,i}(x,y,\gamma_0)$ & Minimum of $T^{\dis,i}(\cdot,\gamma_0)$ over paths $x\to y$ & \pageref{Tdis} \\
$a_j,b_k$ & Lattice sites in starting and midway blocks $B_{\theta,{\rm home}}^{\rm rfat}$ and $B_{\theta,{\rm cross}}^{\rm fat}$ & \pageref{ajb} \\
$n_0$ & Required minimum number of fast $\ell$--segments & \pageref{Hjk} \\
$H_{jk}$ & Event that $\Gamma_{a_jb_k}$ is well--behaved regarding fast $\ell$--segments & \pageref{Hjk} \\
$\mI(\gamma,\tau_\gamma)$ & Set of indices of fast $\ell$--segments & \pageref{Iset} \\
$H_\gamma,H_{\gamma,I}$ & Events that certain $\ell$--segments have certain partners & \pageref{Hev} \\
$\mM(\gamma)$ & Set of triples corresponding to the transitions of $\gamma$ & \pageref{Mga} \\
$F_{M,I}$ & Event: certain transitions in $M$ made by some ``nice'' geodesic & \pageref{fmi} \\
$\ol F_{M,I}$ & Event: certain transitions in $M$ made by specific ``nice'' geodesics & \pageref{fmi} \\
$F_{M,i}^*$ & Event: specific geodesic making $i$th transition in $M$ is fast & \pageref{Fst} \\
\end{longtable}

\end{appendices}


\vskip 2mm

\begin{thebibliography}{99}

\bibitem{AH16} Ahlberg, D. and Hoffman, C. (2016). Random coalescing geodesics in first-passage percolation. arXiv:1609.02447 [math.PR] 

\bibitem{Al97} Alexander, K.S. (1997). Approximation of subadditive functions and rates of convergence in limiting shape results.  \emph{Ann. Probab.} \textbf{24} 30--55.

\bibitem{Al20a} Alexander, K. S. (2020). Uniform fluctuation and wandering bounds in first passage percolation. arXiv:2011.07223[math.PR] 

\bibitem{AOF14} Alves, S. G., Oliveira, T. J., and Ferreira, S. C. (2018). Universality of fluctuations in the Kardar-Parisi-Zhang class in high dimensions and its upper critical dimension. \emph{Phys. Rev. E} \textbf{90} 020103.  arXiv:1405.0974 [cond-mat.stat-mech]

\bibitem{BFP14} Baik, J., Ferrari, P. L., and P\'ech\'e, S. (2014). Convergence of the two-point function of the stationary TASEP. \emph{Singular phenomena and scalingin mathematical models}, 91--110, Springer, Cham.  arXiv:1209.0116 [math-ph]

\bibitem{BBS19} Bal\'azs, M., Busani, O., and Sepp\"al\"ainen, T. (2020).  Nonexistence of bi-infinite geodesics in the exponential corner growth model.  \emph{Forum Math. Sigma} \textbf{8}, Paper No.~e46, 34 pp. arXiv:1909.06883 [math.PR] 

\bibitem{BHS18} Basu, R., Hoffman, C., and Sly, A. (2018). Nonexistence of bigeodesics in integrable models of last passage percolation.  arXiv:1811.04908 [math.PR]

\bibitem{BSS19} Basu, R., Sarkar, S., and Sly, A. (2019).  Coalescence of geodesics in exactly solvable models of last passage percolation.  \emph{J. Math. Phys.} \textbf{60} 093301, 22 pp.  arXiv:1704.05219 [math.PR] 

\bibitem{BSS16} Basu, R., Sidoravicius, V., and Sly, A. (2016).  Last passage percolation with a defect line and the solution of the slow bond problem.  arXiv:1408.3464 [math.PR] 

\bibitem{BR08} Bena\"im, M. and Rossignol, R. (2008).  Exponential concentration for first passage percolation through modified Poincar\'e inequalities.    \emph{Ann. Inst. Henri Poincar\'e Probab. Stat.} \textbf{44} 544--573.  arXiv:math/0609730 [math.PR]

\bibitem{BKS03} Benjamini, I., Kalai, G., and Schramm, O. (2003). First passage percolation has sublinear distance variance. \emph{Ann. Probab.} \textbf{31} 1970--1978.  arXiv:math/0203262 [math.PR] 

\bibitem{BS20} Busani, O. and Sepp\"al\"ainen, T. (2020). Non-existence of bi-infinite polymer Gibbs measures. arXiv:2010.11279 [math.PR]

\bibitem{C13} Chatterjee, S. (2013).  The universal relation between scaling exponents in first-passage percolation.  \emph{Ann. of Math. (2)} {\bf 127}, no.~2, 663--697.  arXiv:1105.4566 [math.PR]

\bibitem{CLW16} Corwin, I., Liu, Z., and Wang, D. (2016). Fluctuations of TASEP and LPP with general initial data. \emph{Ann. Appl. Probab.} {\bf 26}, 2030--2082.  arXiv:1412.5087 [math.PR]

\bibitem{CD81} Cox, J. T. and Durrett, R. (1981).  Some limit theorems for percolation processes with necessary and suficient conditions.  \emph{Ann. Probab.} {\bf 4}, 583--603.

\bibitem{DH14} Damron, M. and Hanson, J. (2014). Busemann functions and infinite geodesics in two-dimensional first-passage percolation. \textit{Comm. Math. Phys.} {\bf 325}, no. 3, pp.~917--963.  arXiv:1209.3036 [math.PR]

\bibitem{DH17} Damron, M. and Hanson, J. (2017). Bigeodesics in first-passage percolation. \textit{Comm. Math. Phys.} {\bf 349}, no. 2, pp.~753--776.  arXiv:1512.00804 [math.PR]

\bibitem{DK14} Damron, M. and Kubota, N. (2014). Gaussian concentration for the lower tail in first-passage percolation under low moments. \emph{Stoch. Proc. Appl.} {\bf 126}, 3065--3076.  arXiv:1406.3105 [math.PR]

\bibitem{Fo08}  Fogedby, H. C. (2006). Kardar-Parisi-Zhang equation in the weak noise limit: Pattern formation and upper critical dimension. \emph{Phys. Rev. E} \textbf{73} 031104.  arXiv:cond-mat/0510268 [cond-mat.stat-mech]

\bibitem{Ga19} Gangopadhyay, U. (2020).  Fluctuations of transverse increments in two--dimensional first passage percolation. arXiv:2011.14686 [math.PR]

\bibitem{Ge88} Georgii, H. O. (1998).  \emph{Gibbs Measures and Phase Transitions}.  de Gruyter Studies in Mathematics \textbf{9}, de Gruyter, Berlin.

\bibitem{GRS17} Georgiou, N., Rassoul-Agha, F., and Sepp\"al\"ainen, T. (2017). Geodesics and the competition interface for the corner growth model.  \emph{Prob. Theory Rel. Fields} {\bf 169} 223--255.  arXiv:1510.00860 [math.PR]

\bibitem{Ke93} Kesten, H. (1993).  On the speed of convergence in first-passage percolation.  \emph{Ann. Appl. Probab.} \textbf{3} 296--338.

\bibitem{KK14} Kim, S.-W. and Kim, J. M. (2014). A restricted solid-on-solid model in higher dimensions. \emph{J. Stat. Mech.} \textbf{2014} P07005.  

\bibitem{LW05} Le Doussal, P. and Wiese, K. J. (2005). Two-loop functional renormalization for elastic manifolds pinned by disorder in $N$ dimensions. \emph{Phys. Rev. E} \textbf{72} 035101.  arXiv:cond-mat/0501315 [cond-mat.dis-nn]

\bibitem{LR10} Ledoux, M. and Rider, B. (2010)  Small deviations for beta ensembles. \emph{Electron. J. Probab.} {\bf 15}, 1319--1343.  arXiv:0912.5040 [math.PR]

\bibitem{LN96} Licea, C. and Newman, C. M. (1996). Geodesics in two-dimensional first-passage percolation.  \emph{Ann. Probab.} {\bf 24}, 399--410.

\bibitem{LNP96} Licea, C., Newman, C. M. and Piza, M. S. T., Superdiffusivity in first-passage percolation, \textit{Probab. Theory Rel. Fields} \textbf{106} (1996), 559--591.

\bibitem{LM01} Lo\"we, M. and Merkl, F. (2001). Moderate deviations for longest increasing subsequences: The upper tail. \emph{Comm. Pure Appl. Math.} {\bf 54}, 1488--1519.

\bibitem{LMR02} Lo\"we, M., Merkl, F., and Rolles, S. (2002).  Moderate deviations for longest increasing subsequences: The lower tail. \emph{J. Theor. Probab.} {\bf 15}, 1031--1047.

\bibitem{MPPR02} Marinari, E., Pagnani, A., Parisi, G., R\'acz, Z. (2002). Width distributions and the upper critical dimension of Kardar-Parisi-Zhang interfaces. \emph{Phys. Rev. E} \textbf{65} 026136.  arXiv:cond-mat/0105158 [cond-mat.stat-mech]

\bibitem{N95} Newman, C. M., A surface view of first passage percolation.  \emph{Proceedings of the International Congress of Mathematicians}, Vol.~1, 2 (Z\"urich, 1994), 1047--1023, Birkh\"auser, Basel (1995).

\bibitem{NP95} Newman, C. M. and Piza, M. S. T. (1995). Divergence of shape fluctuations in two dimensions. \textit{Ann. Probab.} \textbf{23}, 977--1005.

\bibitem{P15} Pimentel, L. (2015). Duality between coalescence times and exit points in last-passage percolation models.  \emph{Ann. Probab.} {\bf 44}, 3187--3206. arXiv:1307.7769 [math.PR]

\bibitem{ROM15} Rodrigues, E. A., Oliveira, F. A., and Mello, B. A. (2015). On the existence of an upper critical dimension for systems within the KPZ universality class.  \emph{Acta. Phys. Pol. B} \textbf{46}, 1231--1234.  arXiv:cond-mat/1502.06121 [cond-mat.stat-mech]

\bibitem{SS19} Sepp\"al\"ainen, T. and Shen, X. (2020). Coalescence estimates for the corner growth model with exponential weights. \emph{Electron. J. Probab.} \textbf{25}, Paper No.~85, 31 pp.  arXiv:1911.03792 [math.PR]  

\bibitem{SS21} Sepp\"al\"ainen, T. and Sorensen, E. (2021). Busemann process and semi-infinite geodesics in Brownian last-passage percolation. arXiv:2103:01172 [math.PR]

\bibitem{Ta95} Talagrand, M. (1995).  Concentration of measure and isoperimetric inequalities in product spaces.  \emph{Publications Math\'ematiques de l'Institut des hautes Etudes Scientifiques} {\bf 81}(1), 73--205.  

\bibitem{Te18}  Tessera, R. (2018). Speed of convergence in first passage percolation and geodesicity of the average distance. \emph{Ann. Inst. Henri Poincar\'e Probab. Stat.} {\bf 54}, 569--586.  arXiv:1410.1701 [math.PR]

\bibitem{Zh20} Zhang, L. (2020). Optimal exponent for coalescence of finite geodesics in exponential last passage percolation.  arXiv:1912.07733 [math.PR]

\end{thebibliography}
\end{document}